\documentclass[twoside]{article}
\usepackage[accepted]{aistats2025}
\usepackage{graphicx} 

\usepackage[x11names, dvipsnames]{xcolor}
\usepackage[utf8]{inputenc} 
\usepackage[T1]{fontenc}    
\usepackage{hyperref}       
\usepackage{url}            
\usepackage{booktabs}       
\usepackage{amsfonts}       
\usepackage{amsmath,amssymb,amsthm, bbm}
\usepackage{stmaryrd}
\usepackage{mathrsfs}
\usepackage{nicefrac}       
\usepackage{microtype}      
\usepackage{algorithm}
\usepackage{algpseudocode}
\usepackage{tikz}
\usepackage{enumitem}
\usepackage{subcaption} 
\usepackage[round]{natbib}

\begin{document}
\newtheorem{theorem}{Theorem}
\newtheorem{lemma}{Lemma}
\newtheorem{proposition}{Proposition}
\newtheorem{observation}{Observation}
\newtheorem{remark}{Remark}
\newtheorem{example}{Example}
\newtheorem{definition}{Definition}
\newtheorem{corollary}{Corollary}
\newtheorem{assumption}{Assumption}
\newtheorem{condition}{Condition}
\newtheorem{claim}{Claim}

\renewcommand{\AA}[1]{A^{(#1)}}
\newcommand{\HH}[1]{H^{(#1)}}
\newcommand{\nub}[1]{(\nu_{#1} + b_{#1}^2)}

\newcommand{\ssigma}[1]{\rho^{(#1)}}
\newcommand{\mxbernsymbol}{\eta}

\renewcommand{\Pr}{\mathbb{P}}
\newcommand{\E}{\mathbb{E}}
\newcommand{\N}{\mathbb{N}}
\newcommand{\R}{\mathbb{R}}
\newcommand{\Z}{\mathbb{Z}}

\newcommand{\bbB}{\mathbb{B}}
\newcommand{\bbD}{\mathbb{D}}
\newcommand{\bbG}{\mathbb{G}}
\newcommand{\bbJ}{\mathbb{J}}
\newcommand{\bbK}{\mathbb{K}}
\newcommand{\bbL}{\mathbb{L}}
\newcommand{\bbO}{\mathbb{O}}
\newcommand{\bbP}{\mathbb{P}}
\newcommand{\bbR}{\mathbb{R}}
\newcommand{\bbT}{\mathbb{T}}
\newcommand{\bbU}{\mathbb{U}}

\newcommand{\Rnonneg}{\R_{\ge 0}}

\newcommand{\calA}{\mathcal{A}}
\newcommand{\calB}{\mathcal{B}}
\newcommand{\calC}{\mathcal{C}}
\newcommand{\calD}{\mathcal{D}}
\newcommand{\calE}{\mathcal{E}}
\newcommand{\calH}{\mathcal{H}}
\newcommand{\calJ}{\mathcal{J}}
\newcommand{\calL}{\mathcal{L}}
\newcommand{\calN}{\mathcal{N}}
\newcommand{\calP}{\mathcal{P}}
\newcommand{\calQ}{\mathcal{Q}}
\newcommand{\calR}{\mathcal{R}}
\newcommand{\calS}{\mathcal{S}}
\newcommand{\calT}{\mathcal{T}}
\newcommand{\calX}{\mathcal{X}}
\newcommand{\calM}{\mathcal{M}}
\newcommand{\calI}{\mathcal{I}}
\newcommand{\calV}{\mathcal{V}}

\newcommand{\scrC}{\mathscr{C}}

\newcommand{\Ahat}{\hat{A}}
\newcommand{\Ihat}{{\hat{I}}}
\newcommand{\Jhat}{{\hat{J}}}
\newcommand{\Ohat}{\hat{O}}
\newcommand{\Phat}{\hat{P}}
\newcommand{\Rhat}{\hat{R}}
\newcommand{\Xhat}{\hat{X}}
\newcommand{\dhat}{\hat{d}}
\newcommand{\uhat}{\widehat{u}}
\newcommand{\vhat}{\widehat{v}}
\newcommand{\what}{\widehat{w}}
\newcommand{\rhohat}{\hat{\rho}}
\newcommand{\sigmahat}{\hat{\sigma}}
\newcommand{\tauhat}{\hat{\tau}}
\newcommand{\nuhat}{\hat{\nu}}
\newcommand{\thetahat}{\widehat{\theta}}
\newcommand{\chat}{\hat{c}}

\newcommand{\ccheck}{\check{c}}
\newcommand{\Xcheck}{\check{X}}
\newcommand{\Aml}{\mathring{A}}
\newcommand{\Vml}{\mathring{V}}
\newcommand{\Xml}{\mathring{X}}
\newcommand{\wml}{\mathring{w}}

\newcommand{\Abar}{\bar{A}}
\newcommand{\Pbar}{\bar{P}}
\newcommand{\Xbar}{\bar{X}}
\newcommand{\cbar}{\bar{c}}
\newcommand{\nubar}{\bar{\nu}}
\newcommand{\bbar}{\bar{b}}
\newcommand{\vbar}{\bar{v}}

\newcommand{\Atilde}{\tilde{A}}
\newcommand{\UAtilde}{U_{\Atilde}}
\newcommand{\SAtilde}{S_{\Atilde}}
\newcommand{\Otilde}{\tilde{O}}
\newcommand{\Ptilde}{\tilde{P}}
\newcommand{\Xtilde}{\tilde{X}}
\newcommand{\dtilde}{\tilde{d}}
\newcommand{\wtilde}{\tilde{w}}
\newcommand{\rhotilde}{\tilde{\rho}}
\newcommand{\calDtilde}{\tilde{\calD}}
\newcommand{\calXtilde}{\tilde{\calX}}

\newcommand{\dH}{d_{\text{H}}}
\newcommand{\dHtilde}{\dtilde_{\text{H}}}

\newcommand{\bB}{\mathbf{B}}
\newcommand{\bZ}{\mathbf{Z}}

\newcommand{\Xstar}{X^*}
\newcommand{\bstar}{b^*}
\newcommand{\lambdastar}{\lambda^\star}
\newcommand{\nustar}{\nu^*}
\newcommand{\nmin}{n_{\min}}

\newcommand{\wopt}{\wml}

\newcommand{\SA}{S_A}
\newcommand{\SM}{S_M}
\newcommand{\SP}{S_P}
\newcommand{\UP}{U_P}
\newcommand{\UA}{U_A}
\newcommand{\UM}{U_M}

\newcommand{\RDPG}{\operatorname{RDPG}}
\newcommand{\GRDPG}{\operatorname{GRDPG}}
\newcommand{\ASE}{\operatorname{ASE}}
\newcommand{\KL}{\operatorname{KL}}
\newcommand{\SNR}{\operatorname{SNR}}
\newcommand{\VAR}{\operatorname{Var}}
\newcommand{\COV}{\operatorname{Cov}}
\newcommand{\wtd}{\operatorname{wtd}}
\newcommand{\unif}{\operatorname{unif}}
\newcommand{\Bernoulli}{\operatorname{Bern}}
\newcommand{\supp}{\operatorname{supp}}
\newcommand{\rank}{\operatorname{rank}}
\newcommand{\diag}{\operatorname{diag}}
\newcommand{\tr}{\operatorname{tr}}
\newcommand{\vech}{\operatorname{vech}}
\newcommand{\opnorm}{\operatorname{op}}
\newcommand{\frobnorm}{\mathrm{F}}

\newcommand{\dtildetti}{\dtilde_{\tti}}

\newcommand{\Phatwtd}{\Ptilde}
\newcommand{\Phatunif}{\Pbar}

\newcommand{\onevec}{\vec{\mathbf{1}}}
\newcommand{\ivec}{\vec{i}}
\newcommand{\jvec}{\vec{j}}
\newcommand{\indicator}{\mathbb{I}}
\newcommand{\indic}{\indicator}

\newcommand{\marginal}[1]{\marginpar{\raggedright\scriptsize #1}}
\newcommand{\as}{\text{~~~a.s.}}
\newcommand{\whp}{\text{~~~w.h.p.}}
\newcommand{\inlaw}{\xrightarrow{\calL}}
\newcommand{\inprob}{\xrightarrow{P}}
\newcommand{\iid}{\stackrel{\text{i.i.d.}}{\sim}}
\newcommand{\eqdist}{\stackrel{d}{=}}
\newcommand{\Perm}{\Pi}


\newcommand{\vone}{\boldsymbol{1}}
\newcommand{\mA}{\boldsymbol{A}}
\newcommand{\mB}{\boldsymbol{B}}
\newcommand{\mC}{\boldsymbol{C}}
\newcommand{\mD}{\boldsymbol{D}}
\newcommand{\mE}{\boldsymbol{E}}
\newcommand{\mF}{\boldsymbol{F}}
\newcommand{\mG}{\boldsymbol{G}}
\newcommand{\mH}{\boldsymbol{H}}
\newcommand{\mI}{\boldsymbol{I}}
\newcommand{\mZero}{\boldsymbol{0}}
\newcommand{\mJ}{\boldsymbol{J}}
\newcommand{\mK}{\boldsymbol{K}}
\newcommand{\mM}{\boldsymbol{M}}
\newcommand{\mO}{\boldsymbol{O}}
\newcommand{\mP}{\boldsymbol{P}}
\newcommand{\mQ}{\boldsymbol{Q}}
\newcommand{\mR}{\boldsymbol{R}}
\newcommand{\mS}{\boldsymbol{S}}
\newcommand{\mU}{\boldsymbol{U}}
\newcommand{\mV}{\boldsymbol{V}}
\newcommand{\mVhat}{\widehat{\mV}}
\newcommand{\mW}{\boldsymbol{W}}
\newcommand{\mX}{\boldsymbol{X}}
\newcommand{\mY}{\boldsymbol{Y}}
\newcommand{\mZ}{\boldsymbol{Z}}
\newcommand{\mLambda}{\boldsymbol{\Lambda}}
\newcommand{\mLambdastar}{{\mLambda^\star}}
\newcommand{\mTheta}{\boldsymbol{\Theta}}
\newcommand{\mGamma}{\boldsymbol{\Gamma}}
\newcommand{\mPi}{\boldsymbol{\Pi}}
\newcommand{\mXhat}{\hat{\boldsymbol{X}}}
\newcommand{\mBstar}{\mB^\star}
\newcommand{\Bstar}{B^\star}
\newcommand{\mMstar}{\mM^\star}
\newcommand{\Mstar}{M^\star}
\newcommand{\mYt}{\widetilde{\mY}}
\newcommand{\Yt}{\widetilde{Y}}
\newcommand{\Vt}{\widetilde{V}}
\newcommand{\Ut}{\widetilde{U}}
\newcommand{\mWt}{\widetilde{\mW}}
\newcommand{\Wt}{\widetilde{W}}

\newcommand{\vx}{\boldsymbol{x}}
\newcommand{\vs}{\boldsymbol{s}}
\newcommand{\sstar}{s^\star}
\newcommand{\vd}{\boldsymbol{d}}
\newcommand{\vy}{\boldsymbol{y}}
\newcommand{\ve}{\boldsymbol{e}}
\newcommand{\vu}{\boldsymbol{u}}
\newcommand{\vustar}{\boldsymbol{u}^\star}
\newcommand{\vutilde}{\widetilde{\vu}}
\newcommand{\ustar}{u^\star}
\newcommand{\vuhat}{\boldsymbol{\widehat{u}}}
\newcommand{\vv}{\boldsymbol{v}}
\newcommand{\vo}{\boldsymbol{0}}
\newcommand{\vvhat}{\boldsymbol{\widehat{v}}}
\newcommand{\vw}{\boldsymbol{w}}
\newcommand{\vbt}{\widetilde{\vb}}
\newcommand{\lambdat}{\widetilde{\lambda}}
\newcommand{\lambdahat}{\widehat{\lambda}}
\newcommand{\va}{\boldsymbol{a}}
\newcommand{\vz}{\boldsymbol{z}}
\newcommand{\vm}{\boldsymbol{m}}
\newcommand{\valpha}{\boldsymbol{\alpha}}
\newcommand{\alphahat}{\widehat{\alpha}}
\newcommand{\vtheta}{\boldsymbol{\theta}}
\newcommand{\vsigma}{\boldsymbol{\sigma}}
\newcommand{\vgamma}{\boldsymbol{\gamma}}
\newcommand{\vbeta}{\boldsymbol{\beta}}
\newcommand{\veta}{\boldsymbol{\eta}}
\newcommand{\vlambda}{\boldsymbol{\lambda}}
\newcommand{\mUt}{\widetilde{\mU}}
\newcommand{\mVt}{\widetilde{\mV}}
\newcommand{\mSig}{\boldsymbol{\Sigma}}
\newcommand{\mDelta}{\boldsymbol{\Delta}}
\newcommand{\tmSig}{\tilde{\boldsymbol{\Sigma}}}
\newcommand{\tmLambda}{\tilde{\boldsymbol{\Lambda}}}
\newcommand{\tildV}{\tilde{V}}
\newcommand{\Opq}{\mathcal{O}_{p,q}}
\newcommand{\Od}{\mathbb{O}_d}
\newcommand{\bbS}{\mathbb{S}}
\newcommand{\tB}{\mathcal{B}}
\newcommand{\bM}{\mathbb{M}}
\newcommand{\bU}{\mathbb{U}}
\newcommand{\bV}{\mathbb{V}}
\newcommand{\bbV}{\bV}
\newcommand{\vb}{\boldsymbol{b}}
\newcommand{\vh}{\boldsymbol{h}}
\newcommand{\vg}{\boldsymbol{g}}
\newcommand{\vxi}{\boldsymbol{\xi}}
\newcommand{\vzeta}{\boldsymbol{\zeta}}
\newcommand{\vdelta}{\boldsymbol{\delta}}
\newcommand{\vnu}{\boldsymbol{\nu}}
\newcommand{\vmu}{\boldsymbol{\mu}}
\newcommand{\I}{{I_\star}}
\newcommand{\mXi}{\boldsymbol{\Xi}}
\newcommand{\mUstar}{{\mU^\star}}
\newcommand{\Ustar}{{U^\star}}
\newcommand{\mUhat}{\widehat{\mU}}
\newcommand{\mPhi}{\boldsymbol{\Phi}}

\newcommand{\gsup}{S_{\mathcal{G}}}
\newcommand{\cG}{\mathcal{G}}
\newcommand{\zvec}{\mathbf{0}}
\newcommand{\veps}{\boldsymbol{\varepsilon}}
\newcommand{\vepsilon}{\boldsymbol{\epsilon}}
\newcommand{\vDelta}{\boldsymbol{\Delta}}
\newcommand{\sgn}[1]{\operatorname{sgn}\left(#1\right)}
\newcommand{\tti}[1]{\|#1\|_{2,\infty}}
\newcommand{\bora}{\hat{\vbeta}^{o}}
\newcommand{\thora}{\hat{\vtheta}^{o}}
\newcommand{\vora}{\hat{\vv}^{o}}
\newcommand{\gevent}{\mathbb{G}(\lambda_n)}
\newcommand{\bbM}{\mathbb{M}}
\newcommand{\bbC}{\mathbb{C}}
\newcommand{\conev}{\bbC_{\vv^*}(\bbM, \bbM^{\perp})}

\newcommand{\Shat}{\widehat{S}}
\newcommand{\calG}{\mathcal{G}}

\newcommand{\wDelta}{\widetilde{\mDelta}}
\newcommand{\wGamma}{\widetilde{\mGamma}}
\newcommand{\wW}{\widetilde{\mW}}

\newcommand{\inner}[1]{\langle #1 \rangle}
\newcommand{\sI}{\mathcal{I}}

\newcommand{\mWl}{\mW^{(\ell)}}
\newcommand{\mMl}{\mM^{(\ell)}}
\newcommand{\mHl}{\mH^{(\ell)}}
\newcommand{\mUl}{\mU^{(\ell)}}
\newcommand{\mEl}{\mE^{(\ell)}}
\newcommand{\lDelta}{\mDelta^{(\ell)}}
\newcommand{\lLambda}{\mLambda^{(\ell)}}
\newcommand{\llambda}{\lambda^{(\ell)}}
\newcommand{\mZhat}{\widehat{\mZ}}
\newcommand{\mZtilde}{\widetilde{\mZ}}
\newcommand{\mZstar}{\mZ^\star}
\newcommand{\mPsi}{\boldsymbol{\Psi}}

\newcommand{\rmF}{\mathrm{F}}

\newcommand{\sighat}{\widehat{\sigma}}

\newcommand{\mMhat}{\widehat{\mM}}
\newcommand{\Mhat}{\widehat{M}}

\newcommand{\dF}{d_{\frobnorm}}
\newcommand{\dtti}{d_{2,\infty}}
\newcommand{\dop}{d_{\mathrm{op}}}

\newcommand{\Hess}{\operatorname{Hess}}
\newcommand{\grad}{\operatorname{grad}}
\newcommand{\baf}{\bar{f}}

\newcommand{\vertiii}[1]{{\left\vert\kern-0.25ex\left\vert\kern-0.25ex\left\vert #1 
    \right\vert\kern-0.25ex\right\vert\kern-0.25ex\right\vert}}

\newcommand{\Iscr}{\mathscr{I}}

\newcommand{\calU}{\mathcal{U}}
\newcommand{\mUperp}{\mU_{\perp}}
\newcommand{\mHperp}{\mH_{\perp}}
\newcommand{\mHt}{\widetilde{\mH}}
\newcommand{\thetastar}{\theta^\star}
\newcommand{\vustart}{\vu^{\star \top}}
\newcommand{\PO}[1]{\mathcal{P}_{\Omega}\left(#1\right)}
\newcommand{\POc}[1]{\mathcal{P}_{\Omega^c}\left(#1\right)}
\newcommand{\tvustar}{\widetilde{\vustar}}
\newcommand{\tvustart}{\widetilde{\vustar}^\top}
\newcommand{\Poff}[1]{\mathcal{P}_{\text{off-diag}}\left(#1\right)}
\newcommand{\C}{\mathbb{C}}
\newcommand{\mWp}{\mW'}
\newcommand{\vwp}{\vw'}
\newcommand{\Wp}{W'}
\newcommand{\mWpp}{\mW''}
\newcommand{\Wpp}{W''}
\newcommand{\IE}{\mathbb{I E}}

\newcommand{\vc}{\boldsymbol{c}}
\newcommand{\vect}[1]{\operatorname{vec}\left(#1\right)}
\newcommand{\tX}{\widetilde{X}}
\newcommand{\Gnew}{G^{\text{new}}}
\newcommand{\Gnewb}{\bar{G}^{\text{new}}}
\newcommand{\Gnewt}{\widetilde{G}^{\text{new}}}
\newcommand{\vbstar}{\vb^\star}
\newcommand{\vbstart}{\vb^{\star \top}}
\newcommand{\vbhat}{\widehat{\vb}}
\newcommand{\tilO}{\widetilde{O}}
\newcommand{\jr}[1]{j^{(r)}_{#1}}
\newcommand{\vjj}[1]{\boldsymbol{j}^{(#1)}}
\newcommand{\vj}{\boldsymbol{j}}

\newcommand{\rulesep}{\unskip\ \vrule\ }

\newcommand{\rbar}{\bar{r}}
\newcommand{\lbar}{\bar{l}}
\newcommand{\pibar}{\bar{\pi}}
\newcommand{\Gb}{G_{\mathrm{b}}}

\newcommand{\vxb}{\bar{\vx}}
\newcommand{\xb}{\bar{x}}
\newcommand{\vyb}{\bar{\vy}}
\newcommand{\yb}{\bar{y}}

\newcommand{\mUstart}{\mU^{\star \top}}
\newcommand{\calK}{\mathcal{K}}
\newcommand{\Deltastar}{\Delta^\star}
\twocolumn[

\aistatstitle{Improved dependence on coherence in eigenvector and eigenvalue estimation error bounds}

\aistatsauthor{ Hao Yan \And Keith Levin}

\aistatsaddress{ Department of Statistics\\
  University of Wisconsin--Madison\\
  Madison, WI 53706} ]

\begin{abstract}
Spectral estimators are fundamental in low-rank matrix models and arise throughout machine learning and statistics, with applications including network analysis, matrix completion and PCA.
These estimators aim to recover the leading eigenvalues and eigenvectors of an unknown signal matrix observed subject to noise.
While extensive research has addressed the statistical accuracy of spectral estimators under a variety of conditions, most previous work has assumed that the signal eigenvectors are incoherent with respect to the standard basis.
This assumption typically arises because of suboptimal dependence on coherence in one or more concentration inequalities.
Using a new matrix concentration result that may be of independent interest, we establish estimation error bounds for eigenvector and eigenvalue recovery whose dependence on coherence significantly improves upon prior work.
Our results imply that coherence-free bounds can be achieved when the standard deviation of the noise is comparable to its Orlicz 1-norm (i.e., its subexponential norm).
This matches known minimax lower bounds under Gaussian noise up to logarithmic factors.
\end{abstract}

\section{INTRODUCTION}

Spectral methods are widely employed in modern data science and engineering \cite{chen2021spectral}. 
A common framework where these methods excel involves recovering a low-rank signal matrix $\mMstar \in \R^{n\times n}$ from a noisy observation
\begin{equation}\label{eq:model}
    \mM = \mMstar + \mH, 
\end{equation}
where $\mH \in \R^{n\times n}$ represents additive noise, typically assumed independent of $\mMstar$.
The key principle behind spectral methods is that the eigenspace and eigenvalues (or the singular subspace and singular values) of the observed matrix $\mM$ capture essential structural information about the signal matrix $\mMstar$.
By accurately estimating the structure of $\mMstar$, valuable insights about the underlying data can be uncovered.
Indeed, Equation~\eqref{eq:model} encompasses a range of tasks, including low-rank matrix denoising \citep{donoho2014minimax, bao2021singular}, factor analysis \citep{cai2013sparse, fan2021robust, zhang2022heteroskedastic}, 
community detection \citep{rohe2011spectral, sussman2013consistent, emmanuel2018community}
and matrix completion \citep{sourav2015matrix, athey2021matrix}. 
The widespread use of spectral methods has driven extensive research into how eigenspaces and eigenvalues (resp.~singular spaces and singular values) change under matrix perturbation \citep[see, for example,][]{bao2022statistical, fan2022asymptotic, xia2021normal, yu2014useful, cai2018rate-optimal, cape2019two-to-infinity, o2023matrices}. 
For a comprehensive review of spectral methods, see \cite{chen2021spectral}. 

Roughly speaking, the coherence of a vector with respect to the standard basis characterizes how its magnitude is distributed across their components (see Equation~\eqref{eq:mu-define} for a formal definition).
In perturbation analysis of eigenvectors and eigenvalues, as well as in tasks such as matrix completion, the coherence of the signal eigenvectors plays a key role \citep{bhardwaj2024matrix, chen2021asymmetry, cheng2021tackling}. 
Indeed, it has been conjectured that certain perturbation bounds, such as entrywise eigenvector bounds, depend inherently on the coherence \citep{bhardwaj2024matrix}.
Recent work by \cite{YanLev2024} disproves this conjecture under the model in Equation~\eqref{eq:model} when $\mMstar$ and $\mH$ are symmetric, the signal matrix $\mMstar$ has suitably large eigenvalues and the entries of the noise matrix $\mH$ are homoscedastic.
The present paper aims to improve our understanding of the effect of coherence by deriving non-asymptotic eigenvalue and eigenvector perturbation bounds under Equation~\eqref{eq:model} in the asymmetric case, with improved dependence on coherence, while relaxing the homoscedasticity assumptions of previous work.

\subsection{Background and notation}

Throughout this paper, we assume that $\mMstar$ is a rank-$r$ symmetric matrix with spectral decomposition $\mMstar = \mUstar \mLambdastar \mUstart$, where $\mUstar \in \R^{n\times r}$ has orthonormal columns, and $\mLambdastar = \diag\left(\lambdastar_1, \dots, \lambdastar_r\right)$ is a diagonal matrix. 
Without loss of generality, we assume that $\lambdastar_{\max} = |\lambdastar_1| \geq \cdots \geq |\lambdastar_r| = \lambdastar_{\min}$. 
Additionally, we let the top $r$ eigenvalues of $\mM$ be $\lambda_1, \dots, \lambda_r$, indexed so that $|\lambda_1| \geq \cdots \geq |\lambda_r|$, and their corresponding normalized eigenvectors be $\vu_1, \dots, \vu_r$. 
When $r = 1$, we let $\mMstar = \lambdastar \vustar \vustart$ and denote the leading eigenvalue and eigenvector of $\mM$ as $\lambda$ and $\vu$, respectively.

The incoherence parameter of a rank-$r$ symmetric matrix $\mMstar = \mUstar \mLambdastar \mUstart$ is defined as
\begin{equation} \label{eq:mu-define} 
\mu = n\left\|\mUstar\right\|_{\infty}^2 \in [1, n],
\end{equation}
where $\|\cdot\|_{\infty}$ denotes the entrywise $\ell_{\infty}$ norm.
Broadly speaking, $\mu$ quantifies how the mass is distributed within the eigenspace of $\mMstar$. 
A common intuition holds that larger values of $\mu$ make estimating $\mMstar$ more challenging, as $\mUstar$ aligns more closely with the standard basis vectors, rendering $\mMstar$ nearly sparse and thus more susceptible to noise-induced perturbation. 
\begin{remark} \label{rem:coh:def}
    We remind the reader of an alternative definition for the incoherence parameter \citep{candes2012exact},
\begin{equation*}
    \mu_0 = \frac{n}{r} \|\mUstar\|_{2,\infty}^2 \in [1, n/r],
\end{equation*}
where $\|\cdot\|_{2, \infty}$ denotes the maximum row-wise $\ell_2$ norm.
These two definitions agree in the rank-one case, with $\mu_0 \leq \mu \leq \mu_0 r$ in general. 
Since our primary focus is eigenvector recovery rather than eigenspace recovery, we adopt the definition given in Equation~\eqref{eq:mu-define}.
\end{remark}
Many eigenvector and eigenvalue perturbation bounds rely on the incoherence parameter $\mu$. 
For instance, when $\mMstar$ is rank-one, under Assumption~\ref{assump:1}, Corollary 1 in \cite{chen2021asymmetry} shows that
\begin{equation} \label{eq:ev-bound-suboptimal}
    |\lambda - \lambdastar| \lesssim \max \left\{\sigma \sqrt{n \log n}, B \log n\right\} \sqrt{\frac{\mu}{n}},
\end{equation}
where $\sigma^2$ represents an upper bound on the variance of the additive noise $\mH$, while $B$ controls its tail behavior (precise definitions are give in Assumption~\ref{assump:1}). 
In a variety of settings, $\mu$ may be large.
Such settings include stochastic block models with unequal cluster sizes \citep{mukherjee2023recovering}, sparse PCA \citep{agterberg2022entrywise} and matrix completion~\citep{chen2014coherent, yan2024entrywiseinferencemissingpanel}.
When $\mu$ is large, the bound in Equation~\eqref{eq:ev-bound-suboptimal} seems improvable.
Indeed, due to the presence of $\mu$ in the eigenvalue and eigenvector estimation error bounds, many papers adopt the incoherence assumption $\mu = O(1)$ \citep{xie2024higherorderentrywise, agterberg2024estimating, mao2021estimatingmixed, lei2016goodness, rubin2022statistical}, making it unclear whether similar results hold for large $\mu$.

We note in passing that the problem of recovering the eigenvalues and eigenvectors of $\mMstar$ under the model in Equation~\eqref{eq:model} bears some resemblance to the problem setting considered in the noisy power method literature \citep[see, e.g.,][]{HarPri2014}.
That line of work typically assumes that we access only a subsample of the entries of a large (typically Hermitian) matrix, subject to noise.
This is in contrast to the setting considered here, in which we assume that we observe the matrix $\mM$ in its entirety. 

As outlined above, we aim to improve the dependence on $\mu$ for a range of eigenvalue and eigenvector estimation bounds. 
This, in turn, would provide a clearer understanding of how the estimation error depends on the eigenspace structure and the tail behavior of the noise distribution.
In Corollary~\ref{cor:eigenvalue:one} and Equation~\eqref{eq:eigenvalue:one} below, we demonstrate that a coherence-free version of Equation~\eqref{eq:ev-bound-suboptimal} is attainable for a wide range of noise distributions.
This is accomplished via two key tools, described in Section~\ref{sec:main}:
(i) a refined combinatorial argument utilizing graphical tools to obtain non-asymptotic bounds on high-order moments of the noise matrix $\mH$; 
(ii) a novel decomposition which breaks any unit vector $\vx$ into a sum of vectors with similar $\ell_{\infty}$-norms. 

The remainder of the paper is organized as follows. 
In Section~\ref{sec:setup}, we introduce the assumptions and preliminaries for our problem.
In Section~\ref{sec:main}, we establish our new concentration results for both asymmetric and symmetric noise matrices. 
In Section~\ref{sec:applications}, we apply our new results to provide new bounds for various eigenvalue and eigenvector estimation problems. 
Numerical experiments are provided in Section~\ref{sec:numeric}. 
We conclude in Section~\ref{sec:discussions} with a discussion of the limitations of our study and directions for future work. 
Detailed proofs of all results can be found in the supplementary materials.

\paragraph{Notation} 

Before proceeding, we pause to establish notation.
We use $c$ and $C$ to denote positive constants whose precise values may change from line to line.
For a positive integer $n$, we write $[n] = \{1,2,\dots,n\}$.
$|\calA|$ denotes the cardinality of a set $\calA$. 
For a vector $\vv=\left(v_1,v_2, \dots, v_n\right)^\top \in \R^n$, we denote the norms $\|\vv\|_2=\sqrt{\sum_{i=1}^n v_i^2}$ and $\|\vv\|_{\infty}=\max_i\left|v_i\right|$.
We denote the standard basis vectors of $\R^n$ by $\ve_i \in \R^{n}$ for $i \in [n]$ and write $\bbS^{n-1} = \left\{ \vu \in \R^n: \|\vu\|_2 = 1 \right\}$ for the unit sphere.
For a matrix $\mM \in \R^{n\times n}$, 
$\|\mM\|$ denotes its operator norm. 
$\mI_{n} \in \R^{n\times n}$ denotes the $n$-by-$n$ identity matrix.
We use both standard Landau notation and asymptotic notation: for positive functions $f(n)$ and $g(n)$, we write $f(n) \gg g(n)$, $f(n) = \omega(g(n))$ or $g(n) = o( f(n) )$ if $f(n)/g(n) \rightarrow \infty$ as $n \rightarrow \infty$.
We write $f(n) \gtrsim g(n)$, $f(n) = \Omega( g(n) )$ or $g(n) = O( f(n) )$ if for some constant $C > 0$, we have $f(n)/g(n) \ge C$ for all sufficiently large $n$.
We write $f(n) = \Theta( g(n) )$ if both $f(n) = O(g(n))$ and $g(n) = O(f(n))$. 
We write $f(n) = \Otilde\left(g(n)\right)$ if there exists a constant $\alpha$ such that $f(n) = O(g(n)\log^\alpha n)$.  
 \label{sec:intro}

\section{MODEL AND ASSUMPTIONS} \label{sec:setup} 

In this section, we present our model and underlying assumptions. 
Our primary focus is on the scenario where the signal matrix $\mMstar$ is observed with noise represented by an asymmetric matrix, modeled as
    \begin{equation}\label{eq:asymm-model}
        \mM = \mMstar + \mH \in \R^{n\times n}. 
    \end{equation}
    In Section~\ref{sec:symmetric}, we extend our results to
    \begin{equation}\label{eq:symm-model}
        \mM = \mMstar + \mW, 
    \end{equation}
    where the noise matrix $\mW \in \R^{n\times n}$ is symmetric. 
    While the assumption of an asymmetric noise matrix $\mH$ may seem unconventional, it arises naturally in situations where multiple copies of $\mM$ from Equation~\eqref{eq:symm-model} are arranged asymmetrically.  
    For example, analyzing Equation~\eqref{eq:asymm-model} can lead to improved estimation methods for multilayer network analysis.
    This field that has recently attracted a lot of research interest \citep{LevAthTanLyzYouPri2017, draves2020bias, macdonald2022latent, lei2023computational, xie2024biascorrected}, and we intend to explore this direction in future work. 
    Additionally, our analysis offers valuable theoretical insights for general low-rank signal-plus-noise models.  

\subsection{Assumptions}

    We concentrate on the case where $(H_{i j})_{1 \leq i, j \leq n}$ are independently generated, without assuming identical distributions or homoscedasticity. 
    Instead, we consider the conditions in Assumption~\ref{assump:1}. 
\begin{assumption}\label{assump:1}
\hfill
\begin{enumerate}
    \item\label{assump:center}\textbf{(Zero mean)} $\E \left[H_{i j}\right]=0$ for all $1 \leq i, j \leq n$;
    \item\label{assump:var}  \textbf{(Variance)} $\VAR\left(H_{i j}\right)=\E \left[H_{i j}^2\right] = \sigma_{ij}^2 \leq \sigma^2$ for all $1 \leq i, j \leq n$;
    \item\label{assump:mag}  \textbf{(Magnitude)} Each $H_{i j}\; (1 \leq i, j \leq n)$ satisfies at least one of the following conditions:
    \begin{enumerate}[label = (\alph*), ref=\theenumi{} (\alph*)]
        \item \label{item:assump-3a} $\left|H_{i j}\right| \leq B$;
        \item \label{item:assump-3b} $\Pr\left\{\left|H_{i j}\right|>B\right\} \leq c_b n^{-12}$ for some universal constant $c_b>0$. 
    \end{enumerate}
    \item \label{assump:symm} \textbf{(Symmetry)} $H_{ij}$ has a symmetric distribution about $0$ for all $1 \leq i, j \leq n$.
\end{enumerate}
\end{assumption}
Throughout this paper, we assume condition~\ref{item:assump-3a} for technical simplicity. 
All our results continue to hold under condition~\ref{item:assump-3b} by applying a standard truncation argument to the entries of $\mH$. 
We remind the reader that $\sigma = \sigma_n$ and $B = B_n$ are allowed to depend on $n$, but we suppress this dependence to simplify notation. 

\begin{remark} \label{rem:symm-about-0}
    Conditions~\ref{assump:center} through \ref{assump:mag} in Assumption~\ref{assump:1} closely follow Assumption 1 in \cite{chen2021asymmetry}, and are standard assumptions in the literature \citep[see also][]{cheng2021tackling,zhou2023deflatedheteropcaovercomingcurse}.  
    We impose an additional symmetry assumption in Condition~\ref{assump:symm} for technical convenience. 
    This ensures that the moments of $\mH$ and $\mW$ have a simpler combinatorial structure, a fact we use in the proofs of Theorem~\ref{thm:IExHky-moment} and Lemma~\ref{lem:symm-mean}.
    We anticipate that similar results will hold without Condition~\ref{assump:symm}, at the cost of a more complicated combinatorial argument than the one we give here.
    In Section~\ref{sec:numeric}, we provide empirical evidence to support this conjecture, leaving a formal proof for future work.
\end{remark}

It often suffices to consider the setting where
\begin{equation} \label{eq:B-sigma-bound}
    B \ll \sigma \sqrt{\frac{n}{\log^3 n}},
\end{equation}
as the following two examples illustrate.

\begin{example}[Low-rank matrix completion] \label{ex:matrix-completion}
In matrix completion with random sampling, $\mM$ is given by
\begin{equation*}
    M_{i j}= \begin{cases}\frac{1}{p} \Mstar_{i j}, & \text { with probability } p \\ 0, & \text { with probability } 1-p.\end{cases}
\end{equation*}
That is, each entry of $\mMstar$ is observed with probability $p$ independently. 
In this setup, we do not distinguish between observed zeros (where the true entry of $\mMstar$ is zero) and missing entries, which are represented as zeros in $\mM$. 
One can verify that when $\mMstar = \lambdastar \vustar \vustart$ is rank-one, $B = \lambdastar\mu/np$ and $\sigma = \lambdastar\mu / n\sqrt{p}$.
Under this setting, it is typically required that $p \gtrsim (\mu \log^2 n )/ n$ \citep{chen2014coherent}, making the growth rate assumption in Equation~\eqref{eq:B-sigma-bound} reasonable. 
The observation probability $p$ depends on the incoherence parameter because a highly coherent $\vustar$ implies that the signal matrix is (approximately) sparse. 
Recovering sparse entries is impossible unless they are sampled.
Thus, when $\mu$ is large, $p$ must be larger as well.
\end{example}

\begin{example}[Low-rank networks] \label{ex:sparse-net}
    In network analysis, the model in Equation~\eqref{eq:asymm-model} is typically written as 
\begin{equation*}
    A_{ij} \sim \operatorname{Bernoulli}(P_{ij}), \quad i,j \in [n]
\end{equation*}
for an adjacency matrix $\mA$ and a probability matrix $\mP$. 
One has $B = 1$ and $\sigma = \sqrt{\rho}$, where $\rho = \max_{i,j} P_{ij}$. 
It is typically required that $\rho = \omega\left(\log n/n\right)$ for meaningful estimation to be feasible \citep[see, e.g.,][]{abbe2020entrywise, chen2021spectral}, and Equation~\eqref{eq:B-sigma-bound} is reasonable. 
\end{example}

\subsection{Preliminaries}
Before presenting our main results, we first establish several important preliminary facts.
Under Condition~\ref{assump:mag}, the higher-order moments of $H_{ij}$ are bounded by
\begin{equation} \label{eq:high-moment}
    \E |H_{ij}|^\ell \leq B^{\ell-2} \E H_{ij}^2 \leq \sigma^2 B^{\ell-2}, \quad \ell \geq 2, 
\end{equation}
which satisfies the Bernstein moment condition \citep[see Equation (2.15) in ][]{wainwright2019high}. 
Another immediate result of Assumption~\ref{assump:1} is the following high-probability spectral norm bound for $\mH$.
\begin{lemma} \label{lem:W-op}
    Under Assumption~\ref{assump:1}, there exists a universal constant $c_1$ such that with probability at least $1 - O(n^{-20})$, 
    \begin{equation*}
        \|\mH\| \leq c_1 \max\left\{\sigma \sqrt{n \log n}, B \log n \right\}.
    \end{equation*}
    The same bound also holds for $\|\mW\|$. 
\end{lemma}
Lemma~\ref{lem:W-op} is a standard result, a consequence of the matrix Bernstein inequality.
See, for example, Theorem 3.4 in \cite{chen2021spectral}. 

Finally, we restate a classic result in Theorem~\ref{thm:neumann}, which is useful in deriving bounds for spectral estimators \citep{eldridge2018unperturbed}.
See Theorem 2 in \cite{chen2021asymmetry} for a proof. 
\begin{theorem}[Neumann trick] \label{thm:neumann}
    Consider the model in Equation~\eqref{eq:asymm-model} and suppose that $\|\mH\| \leq |\lambda_l|$ for some $l \in [n]$. 
    Then 
    \begin{equation} \label{eq:neumann:rank-r}
        \vu_l=\sum_{j=1}^r \frac{\lambda_j^{\star}}{\lambda_l}\left(\vu_j^{\star \top} \vu_l\right)\left\{\sum_{k=0}^{\infty} \frac{1}{\lambda_l^k} \mH^k \vu_j^{\star}\right\}. 
    \end{equation}
    In particular, if $\mMstar$ is rank one and $\|\mH\| \leq |\lambda_1|$, then
    \begin{equation} \label{eq:neumann:rank-1}
        \vu_1 = \frac{\lambdastar_1}{\lambda_1}\left(\vustart_1 \vu_1\right) \left\{ \sum_{k=0}^\infty \frac{1}{\lambda_1^k} \mH^k \vustar_1\right\}
    \end{equation}
\end{theorem}

\section{MAIN RESULTS} \label{sec:main}

In light of Theorem~\ref{thm:neumann}, perturbation bounds for eigenvectors can be derived by controlling terms of the form $\vx^\top \mH^k \vy$, where $\vx$ and $\vy$ are fixed vectors.
Several results concerning this term exist \citep{tao2013outliers, erdos2013spectral, mao2021estimatingmixed, chen2021asymmetry, fan2022asymptotic, xie2024higherorderentrywise}, but they lack the necessary precision to yield optimal dependence on the coherence parameter $\mu$. 
In this section, we present refined results for controlling this term under both the asymmetric and symmetric noise settings.
In Section~\ref{sec:applications}, we apply these refined results to obtain several novel eigenvector and eigenvalue perturbation bounds.

\subsection{Asymmetric noise matrix} \label{sec:asym}

We adapt the proof strategy of Lemma 6.5 in \cite{erdos2013spectral} and first obtain a higher-order moment bound for the term $\vx^\top \mH^k \vy$ in Theorem~\ref{thm:IExHky-moment}. 
\begin{theorem}\label{thm:IExHky-moment}
    Consider a random matrix $\mH \in \R^{n\times n}$ with independent entries satisfying Assumption~\ref{assump:1} and the bound in Equation~\eqref{eq:B-sigma-bound}. 
    For any fixed vectors $\vx, \vy \in \R^{n}$, 
    any integer $k \geq 2$ and any even integer $p \geq 2$ such that $kp \le \log^3 n$, it holds for all suitably large $n$ that $\E \left(\vx^\top \mH^k \vy - \E\vx^\top \mH^k \vy \right)^{p}$ is bounded, up to a universal constant, by
    \begin{equation} \label{eq:moment-bound-2p}
    \begin{aligned}
    &2^{(k+1)p}(k p)^{k p} p (\|\vx\|_{\infty}\|\vy\|_{\infty})^p  \\
    &~~~~~~~~~\times\max\biggl\{ 
    \frac{\sigma^{k p} n^{\frac{k p}{2}}\!\left(\|\vx\|_0 \|\vy\|_0\right)^{p / 2}}{n^{\frac{p}{2}}(k p)^{k p / 2}}, \\
    &~~~~~~~~~~~~~~~~~~~~~
    \frac{ B^{pk}\sigma^{2 k} n^{k} \|\vx\|_0 \|\vy\|_0}{B^{2k} n (k p)^k} \biggr\}.
    \end{aligned}
    \end{equation}
\end{theorem}

The proof of Theorem~\ref{thm:IExHky-moment} relies on combinatorial techniques. 
Our approach begins by expanding the higher-order moment into a sum of terms, each involving products of entries from $\mH$, $\vx$ and $\vy$. 
While many of these terms have an expectation of zero and vanish, the remaining ones require careful bounding, with their magnitude depending on how frequently different entries of $\mH$ appear in the product. 
Since these products represent moments of $\mH$, the indices in non-vanishing terms must satisfy certain structural constraints.
To systematically handle these terms, we analyze their structure and derive upper bounds in terms of $\vx, \vy, \sigma$ and $B$.
Furthermore, we employ graphical tools to enforce index constraints and facilitate counting of terms associated with each possible bound.
By carefully combining these elements, we ultimately establish the bound in Equation~\eqref{eq:moment-bound-2p}. 
Full details of the proof are provided in the supplementary materials.

Applying a higher-order Markov inequality and a union bound over all integers $2 \leq k \leq 20\log n$, we can translate Theorem~\ref{thm:IExHky-moment} into a high probability bound.
\begin{corollary} \label{cor:sparse-bound}
Under the same setting as Theorem~\ref{thm:IExHky-moment}, for all integers $2 \leq k \leq 20 \log n$ ,
\begin{equation} \label{eq:sparse-bound}
\begin{aligned}
    &\left|\vx^\top \mH^k \vy - \E \vx^\top \mH^k \vy \right| \lesssim 
    c_2^k \|\vx\|_{\infty}\|\vy\|_{\infty} \times \\
    &\max\Biggl\{(\sigma^2 n \log^3 n)^{\frac{k}{2}} \sqrt{\frac{\|\vx\|_0 \|\vy\|_0}{n} }, \left(B\log^3 n\right)^k\Biggr\}
\end{aligned}
\end{equation}
holds with probability at least $1 - \Otilde(n^{-40})$, where $c_2 > 0$ is a universal constant.  
\end{corollary}

As discussed in Section~\ref{sec:intro}, another critical element in establishing our improved bounds is a novel decomposition technique for unit vectors.  
For any fixed unit vector $\vx$, we can decompose it into a sum of $m = O(\log n)$ vectors $\{\vx^{(i)}\}_{i=1}^{m}$, where each $\vx^{(i)}$ satisfies $\|\vx^{(i)}\|_{\infty} \sqrt{\|\vx^{(i)}\|_{0}} = O(1)$.
Using this decomposition, our main moment bound follows immediately from Corollary~\ref{cor:sparse-bound}.
Proof details, including the unit vector decomposition, are in the supplementary materials.

\begin{theorem}\label{thm:main}
Under the same setting as Theorem~\ref{thm:IExHky-moment}, for any fixed unit vectors $\vx$ and $\vy \in \R^n$, it holds with probability at least $1 - \Otilde(n^{-40})$ that for all integers $k$ satisfying $2 \leq k \leq 20\log n$,
\begin{equation} \label{eq:main-bound}
\begin{aligned}
    &\left|\vx^\top \mH^k \vy - \E \vx^\top \mH^k \vy \right| \lesssim c_2^k (\log n)^{2} \times\\
    &\max\Biggl\{\! \sqrt{ \frac{ (\sigma^2 n\log^3 \! n)^{k} }{n}},  (B\log^3 n)^k \|\vx\|_{\infty}\|\vy\|_{\infty} \! \Biggr\} .
\end{aligned} \end{equation}
\end{theorem}

Corollary~\ref{cor:sparse-bound} and Theorem~\ref{thm:main} establish a concentration result of $\vx^\top \mH^{k} \vy$ about its mean. 
In order to control $\vx^\top \mH^{k} \vy$ (and thus control the vectors in THeorem~\ref{thm:neumann}), it remains for us to bound $\E \vx^\top \mH^{k} \vy$.
Lemma~\ref{lem:xHky-mean} below provides such a bound.
To establish Lemma~\ref{lem:xHky-mean}, we follow a similar proof strategy as Theorem~\ref{thm:main}.
We first establish a coars bound, then refine it via the unit vector decomposition discussed after Corollary~\ref{cor:sparse-bound} above.
Details are provided in the supplementary materials.

\begin{lemma} \label{lem:xHky-mean}
    Under the same setting as Theorem~\ref{thm:main}, for all $k$ such that $2 \leq k\leq 20\log n$, 
    \begin{equation} \label{eq:xHky-mean-bound}
    \begin{aligned}
        &\left|\E \left[\vx^\top \mH^k \vy\right]\right| \lesssim 
        (\log n)^2 (2k)^k \times \\
        &~~~~~~~~~~~~~\max \left\{ \frac{\sigma^2 B^{k-2}}{k} , \left(\frac{\sigma^2 n}{k}\right)^{k/2-2}\right\}. 
    \end{aligned}
    \end{equation}
\end{lemma}

\begin{remark} \label{rem:xHky-bound}
    Note that under the bound in Equation~\eqref{eq:B-sigma-bound}, for all $k \geq 2$, Equation~\eqref{eq:xHky-mean-bound} is of smaller order compared to the bound given in Equation~\eqref{eq:main-bound}.
    We conclude that for all $2 \leq k \leq 20 \log n$,
    \begin{equation} \label{eq:xHky-bound}
    \begin{aligned}
        &\left|\vx^\top \mH^k \vy\right| \lesssim c_2^k (\log n)^2 \times \\
        &~~\max\left\{\! \sqrt{\! \frac{ \left(\sigma^2 n \log^3 \! n\right)^{k}}{n}} , (\!B\!\log^3\! n)^k \|\vx\|_{\infty} \|\vy\|_{\infty} \!\right\}
    \end{aligned}
    \end{equation}
    holds with probability at least $1 - \Otilde(n^{-40})$. 
    The same bound as in Equation~\eqref{eq:xHky-bound} can be derived by directly applying our proof strategy for Theorems~\ref{thm:IExHky-moment} and~\ref{thm:main} to the term $\vx^\top\mH^k \vy$.
    We opt to work with $\vx^\top\mH^k \vy - \E \vx^\top \mH^k \vy$ because this approach allows us to extend our result to the case where $\mH$ is a symmetric matrix, as detailed in Theorem~\ref{thm:symm-concentration}.
\end{remark}

We pause to provide some intuitive explanation as to why the form of Equation~\eqref{eq:xHky-bound} is to be expected.
The noise parameter $B$ comes into play when there are insufficient noise entries for $\vx^\top \mH^k \vy$ to concentrate around its population mean.
This happens when the mass of $\vx$ and $\vy$ are concentrated on a few entries (i.e., $\vx$ and $\vy$ are localized). 
As such, we expect $B$ to appear in conjunction with $\|\vx\|_{\infty}$ and $\|\vy\|_{\infty}$. 
In contrast, $\sigma$ characterizes the second moment of the noise entries. 
It represents an average quantity that should not be influenced by any individual element of $\vx$ or $\vy$. 
As a result, we expect $\sigma$ to appear independently of $\|\vx\|_{\infty}$ or $\|\vy\|_{\infty}$, which is indeed the case in Equation~\eqref{eq:xHky-bound}. 

Finally, we control $\vx^\top \mH^k \vy$ when $k = 1$, which can be easily established by Bernstein's inequality.
\begin{lemma} \label{lem:xHy-bound}
    Under the same setting as Theorem~\ref{thm:main},
    \begin{equation} \label{eq:xHy-bound}
        \left|\vx^\top \!\mH \vy\right| \leq c_2\!\max \!\left\{\sigma \!\sqrt{\log n}, \|\vx\|_{\infty} \|\vy\|_{\infty} B \!\log n\right\}
    \end{equation}
    holds with probability at least $1 - O(n^{-40})$, where $c_2$ is the same universal constant as in Equation~\eqref{eq:xHky-bound}.
\end{lemma}

\subsection{Symmetric noise matrix} \label{sec:symmetric}
Our result in Theorem~\ref{thm:main} can be directly extended to the case where the noise matrix is symmetric.
Theorem~\ref{thm:symm-concentration} and Lemma~\ref{lem:symm-mean} provide analogues of our previous results, specialized to the symmetric noise case.
Recalling from Equation~\eqref{eq:symm-model} that we use $\mW$ to denote a symmetric noise matrix, Theorem~\ref{thm:symm-concentration} establishes concentration results for $\vx^\top \mW^k \vy$ around $\E \vx^\top \mW^k \vy$, while Lemma~\ref{lem:symm-mean} controls $\E \vx^\top \mW^k \vy$.

\begin{theorem} \label{thm:symm-concentration}
    Consider a symmetric noise matrix $\mW \in \R^{n\times n}$ with independent entries $(W_{ij})_{1\leq i \leq j\leq n}$ satisfying Assumption~\ref{assump:1} and the assumption in Equation~\eqref{eq:B-sigma-bound}. 
    For any fixed unit vectors $\vx, \vy \in \R^{n}$ and all integers $1 \leq k \leq 20\log n$, we have 
    \begin{equation} \label{eq:symm-high-prob-bound}
    \begin{aligned}
        &\left|\vx^\top \mW^k \vy - \E \vx^\top \mW^k \vy\right| \lesssim (2c_2)^k(\log n)^2 \times \\
        &\max \!\left\{ \!\sqrt{ \frac{ \left(\sigma^2 n \log^3 n\right)^{k} }{ n } }, (B \log^3 n)^k\|\vx\|_{\infty}\|\vy\|_{\infty}\right\}
    \end{aligned}
    \end{equation}
    holds with probability at least $1 - \Otilde(n^{-20})$, where $c_2$ is the same universal constant as in Theorem~\ref{thm:main}.
\end{theorem}

\begin{lemma} \label{lem:symm-mean}
    Suppose that $\mW \in \R^{n\times n}$ is a symmetric matrix with independent entries $(W_{ij})_{1\leq i\leq j\leq n}$ satisfying Assumption~\ref{assump:1}. 
    Let $\Poff{\mA}$ denote the operator that sets the diagonal entries of $\mA$ to $0$.
    Then for any fixed unit vectors $\vx, \vy \in \R^n$, 
    \begin{equation}\label{eq:Poff-k}
        \E \vx^\top \Poff{\mW^k} \vy = 0.
    \end{equation}
    Further, if $k$ is odd,
    \begin{equation}\label{eq:on-diag:odd}
        \E \vx^\top \mW^k \vy = 0, 
    \end{equation}
    and if $k$ is even,
    \begin{equation}\label{eq:on-diag}
        \left|\E \vx^\top \mW^k \vy\right|  \lesssim C_{k/2} \left(\sigma^2 n\right)^{k/2},
    \end{equation}
    where
    \begin{equation*}
        C_{k / 2} =\frac{k!}{(k / 2+1)!(k / 2)!}
    \end{equation*}
    is the $(k/2)$-th Catalan number. 
    In addition, if $\mW$ has i.i.d.~on-diagonal and i.i.d.~off-diagonal entries, then
    \begin{equation} \label{eq:iid-k-moment}
        \E \vx^\top \mW^k \vy = \frac{\vx^\top\vy}{n}  \E \tr\left(\mW^k\right). 
    \end{equation}
\end{lemma}

\begin{remark} \label{rem:symm}
    Combining Equation~\eqref{eq:symm-high-prob-bound} with Equation~\eqref{eq:on-diag}, one sees that the dependence of $\vx^\top \mW^k \vy$ on $\|\vx\|_{\infty}$ and $\|\vy\|_{\infty}$ is ignorable when $B/\sigma$ is not too large. 
    Thus, we would expect various coherence-free results also to hold under a symmetric noise setting.
    One can find some direct applications of Theorem~\ref{thm:symm-concentration} in \cite{fan2022asymptotic, xie2024higherorderentrywise}.
    Nevertheless, noting that $\E\vx^\top \mW^k \vy$ in Equation~\eqref{eq:on-diag} dominates the variation $\vx^\top \mW^k \vy - \E\vx^\top \mW^k \vy$ in Equation~\eqref{eq:symm-high-prob-bound}, it is evident that the challenge under symmetric noise lies mostly in reducing the bias of $\E\vx^\top \mW^k \vy$, which has minimal dependence on coherence.
    It is worth mentioning that the bias of $\vx^\top \mW^k \vy$ comes mainly from the diagonal of $\mW^k$ given in Equation~\eqref{eq:Poff-k}.
    Thus, reducing the large bias in the symmetric noise case might be possible through some higher-order diagonal deletion algorithm, which is commonly employed for $k = 2$ \citep{zhang2022heteroskedastic, zhou2023deflatedheteropcaovercomingcurse, xie2024biascorrected}. 
    Under stronger assumptions, such as $\mW$ having i.i.d.\! entries, the bias can also be estimated and subsequently removed.
\end{remark}

\begin{remark} \label{rem:limit-law}
    Thanks to the study of the semicircle law \citep{tao2012topics, bai2010spectral}, there are numerous classical results concerning the moments of $\mW$, as moment methods are commonly used to establish these results. 
    The results in Lemma~\ref{lem:symm-mean} are obtained straightforwardly from these classical results.
    In contrast, fewer results on the moments of $\mH$ exist \cite[though see][for one such recent result in this direction]{byun2024spectral}.
    This is likely because the limiting spectral distribution of $\mH$, which also follows the circular law \citep{bai1997circular}, cannot be derived using the moment method. 
    For more details, see Section 2.8 of \cite{tao2012topics}.
\end{remark}

\section{APPLICATIONS} \label{sec:applications}
Since $\vx^\top \mH^k \vy$ is a key quantity in establishing many technical results, our new results in Theorem~\ref{thm:main} have a range of applications. 
This section highlights a few of them in the context of asymmetric noise. As noted in Remark~\ref{rem:symm}, our results also apply to symmetric noise, but we leave these applications to future work.

\subsection{Rank-one perturbation analysis}
Our new concentration result in Theorem~\ref{thm:main} directly improves upon Theorem 3 in \cite{chen2021asymmetry}.
\begin{theorem}\label{thm:linear-form}
    Consider a rank-one symmetric matrix $\mMstar = \lambdastar \vustar \vustart \in \R^{n\times n}$ with incoherence parameter $\mu$. 
    Suppose the noise matrix $\mH$ obeys Assumption~\ref{assump:1}, and let $C_1 > 0$ be a constant obeying
    \begin{equation} \label{eq:lambdastar:assump}
        |\lambdastar| \geq C_1 \max \left\{\sigma \sqrt{n \log^3 n}, B \log^3 n\right\}.
    \end{equation}
    Let 
    \begin{equation} \label{eq:k0-define}
        k_0 = \left\lceil\frac{\log \mu + 2\log \|\va\|_{\infty}}{2\log \left(\sigma/B\right) +  \log n - 3\log \log n}\right\rceil.
    \end{equation}
    Then for any fixed vector $\va \in \R^n$ with $\|\va\|_2 = 1$, with probability at least $1 - O(n^{-20})$ one has
    \begin{equation} \label{eq:master:rank-one}
    \begin{aligned}
        &\left|\va^\top \left(\vu - \frac{\vustart \vu}{\lambda / \lambdastar} \vustar\right)\right| \lesssim \frac{\log^2 n}{\sqrt{n}}\times\\
        &\max\left\{\! \frac{B \|\va\|_{\infty} \mu^{1/2} \log^3\! n}{\left|\lambdastar\right|}, \left(\frac{\sigma \sqrt{n} \log^{3/2} n}{\left|\lambdastar\right|}\right)^{\!k_0} \right\}. 
    \end{aligned}
    \end{equation}
\end{theorem}

\begin{remark} \label{rem:coherence-free}
A similar bound appears as Equation (15) in \cite{chen2021asymmetry}, which serves as a master bound for deriving bounds for eigenvalues and linear forms of eigenvectors.
The most important difference between Equation~\eqref{eq:master:rank-one} and their result is that we remove the dependence on $\mu$ for the term involving $\sigma$. 
An immediate result is that if $B = \Otilde(\sigma)$, then we have
\begin{equation} \label{eq:coherence-free-master-rank-one}
    \left|\va^\top \left(\vu - \frac{\vustart \vu}{\lambda / \lambdastar} \vustar\right)\right| \lesssim \frac{\sigma \log^{5} n}{\left|\lambdastar\right|}. 
\end{equation}
The condition $B = \Otilde(\sigma)$ is satisfied by a wide range of distributions such as the Gaussian distribution and any distribution whose standard deviation is comparable to its Orlicz $1$-norm. 
Notably, this coherence-free result has not been proven previously. 
More generally, the bound in Equation~\eqref{eq:coherence-free-master-rank-one} also holds if one has $B \|\va\|_{\infty} \lesssim \sigma \sqrt{n/\mu}$. 
While our removal of the dependence on $\mu$ introduces additional logarithmic factors, we conjecture that these can be eliminated through a more refined analysis.
As a final remark, since the results in \cite{chen2021asymmetry} also hold under our setting, one can take the minimum of the two bounds to further reduce some logarithmic factors when the coherence $\mu$ is small. 
For clarity of presentation, we omit these refinements in our stated results.
\end{remark}

\subsubsection{Eigenvalue perturbation: Rank-one}
Theorem~\ref{thm:linear-form} yields an improved perturbation bound regarding the leading eigenvalue $\lambda$ of $\mM$.
\begin{corollary}\label{cor:eigenvalue:one}
    Under the assumptions of Theorem~\ref{thm:linear-form}, with probability at least $1 - O(n^{-20})$ we have
    \begin{equation} \label{eq:eigenvalue:one}
    \begin{aligned}
        &\left|\lambda - \lambdastar\right| \lesssim \\
        &\max\!\left\{\! \frac{B\mu \log^5\! n}{n}, \! \left(\frac{\sigma \sqrt{n} \log^{3/2}\! n}{\left|\lambdastar\right|}\right)^{\!k_0} \! \frac{|\lambdastar|\log^2 \! n}{\sqrt{n}}\! \right\}.
    \end{aligned}
    \end{equation}
\end{corollary}

\begin{remark}
    Under the setting where $B = \Otilde(\sigma)$, Equation~\eqref{eq:eigenvalue:one} becomes
    \begin{equation} \label{eq:coherence-free-eigenvalue}
        \left|\lambda - \lambdastar\right| \lesssim \sigma \log^{5} n,
    \end{equation}
    which matches the minimax lower bound in Lemma 4 of \cite{chen2021asymmetry} up to logarithmic factors when $\mH$ has entrywise i.i.d.~standard Gaussian distribution.
    More generally, Equation~\eqref{eq:coherence-free-eigenvalue} still holds under $B \lesssim \sigma n/\mu$, which implies that for a wide range of noise distributions satisfying condition~\eqref{eq:B-sigma-bound}, the leading eigenvalue of $\mMstar$ can be estimated accurately in the asymmetric case when $\mu = O(\sqrt{n})$.
\end{remark}

\subsection{Rank-\texorpdfstring{$r$}{r} perturbation analysis}
Assume that the $r$ non-zero eigenvalues of $\mMstar$ obey $\lambdastar_{\max} = |\lambdastar_1| \geq |\lambdastar_2| \geq \cdots \geq |\lambdastar_r| = \lambdastar_{\min}$. 
Theorem~\ref{thm:linear-form-rank-r} extends Theorem~\ref{thm:linear-form} to the rank-$r$ case.

\begin{theorem}\label{thm:linear-form-rank-r}
 Consider a rank-$r$ symmetric matrix $\mMstar \in \R^{n\times n}$ with incoherence parameter $\mu$. 
 Define $\kappa := \lambdastar_{\max} / \lambdastar_{\min}$. 
 Suppose the noise matrix $\mH$ obeys Assumption~\ref{assump:1}, and assume the existence of some sufficiently large constant $C_2 \geq 0$ such that
 \begin{equation} \label{eq:lambdastar-max}
     \frac{ \lambdastar_{\max} }{ \kappa }
     \geq C_2 \max \left\{\sigma \sqrt{n \log^3 n}, B \log^3 n\right\} .
 \end{equation}
 Then for any fixed unit vector $\va\in\R^n$ and for any $l \in [r]$, with probability at least $1 - O(r n^{-20})$, one has
 \begin{equation} \label{eq:master:rank-r}
    \begin{aligned}
        &\left|\va^\top \vu_l - \sum_{j=1}^r \frac{\lambdastar_j \vustart_j \vu_l}{\lambda_l} \va^\top \vustar\right| \lesssim \sqrt{\frac{\kappa^2 r \log^4 n}{n}}\times\\
        &
        \max\Bigg\{\frac{\mu^{\frac{1}{2}} B \log^3 n}{\lambdastar_{\min}} \|\va\|_{\infty}, \left(\frac{\sigma \sqrt{n \log ^3 n}}{\lambdastar_{\min}}\right)^{k_0} \Bigg\}.
    \end{aligned}
    \end{equation}
\end{theorem}

\begin{remark}\label{rem:eigenvector}
    Similar to the rank-one case in Theorem~\ref{thm:linear-form}, we remove the dependence on $\mu$ from the term involving $\sigma$. 
    The dependence on $\kappa$ is likely an artifact of the proof technique, and we do not aim to optimize this dependence in the current analysis.
    Using Theorem~\ref{thm:linear-form-rank-r}, one can establish results for entrywise bounds and linear forms of eigenvectors when $\mMstar$ is rank $r$. 
    We expect this approach to improve the bounds in Theorem 1 of \cite{cheng2021tackling} and provide coherence-free bounds under certain conditions, but we leave this to future work, since eigenvector estimation involves an additional debiasing procedure \citep[see Section 3.2 of][]{cheng2021tackling}, which incurs additional proof complexity.
\end{remark}

\subsubsection{Eigenvalue perturbation: Rank-$r$}

As a final application, we improve Theorem 2 in \cite{cheng2021tackling}. 
Unlike the rank-one setting, the eigen-gaps play an important role in eigenvalue estimation for the rank-$r$ case. 
Formally speaking, the eigen-gap with respect to the $l$-th eigenvalue of $\mMstar$ is defined as
\begin{equation} \label{eq:eigen-gap:define}
    \Delta_l^{\star}:= \begin{cases}
    \min _{1 \leq k \leq r, k \neq l}\left|\lambdastar_l -\lambdastar_k\right|, & \text { if } r>1 \\ \infty, & \text { otherwise }.\end{cases}
\end{equation}
Combining Theorem~\ref{thm:main} with tools from random matrix theory, we obtain Theorem~\ref{thm:eigenvalue:rank-r}.
\begin{theorem}\label{thm:eigenvalue:rank-r}
    Suppose that the noise parameters defined in Assumption~\ref{assump:1} satisfy
    \begin{equation} \label{eq:gap-condition}
    \begin{aligned}
        &\Deltastar_l \! \geq \!
        c_3 r^2 \! \max\!\left\{\! \frac{\kappa \mu B\! \log ^3\! n}{n},\! \frac{\left(\! \kappa \sigma \sqrt{n \log ^3 n}\right)^{k_0}}{\left(\lambdastar_{\max }\right)^{k_0-1}\sqrt{n}} \! \right\}
    \end{aligned}
    \end{equation}
    and $\lambdastar_{\max}$ satisfies condition~\eqref{eq:lambdastar-max} for some sufficiently large constants $c_3 > 0$ and $C_2 > 0$. 
    Then given any $l \in [r]$, with probability $1 - O(r^2 n^{-20})$, the eigenvalue $\lambda_l$ is real-valued, and
    $|\lambda_l - \lambdastar_l|$ is upper bounded by
    \begin{equation}\label{eq:eigenvalue-rank-r}
    \begin{aligned}
        c_3 r^2\max\Bigg\{\frac{\kappa \mu B \log ^3 n}{n},
        \frac{\left(\kappa \sigma \sqrt{n \log ^3 n}\right)^{k_0}}{\left(\lambdastar_{\max }\right)^{k_0-1}\sqrt{n}} \Bigg\}.
    \end{aligned}
    \end{equation}
\end{theorem}

\begin{remark} \label{rem:challenge:asymm}
Unlike the case of symmetric noise matrices, where Weyl's inequality allows one to relate the eigenvalues of $\mM$ to $\mMstar$ straightforwardly, the analysis of the asymmetric noise matrix demands a more nuanced approach.
This is because the parallel result to Weyl's inequality, the Bauer-Fike Theorem \citep[see Theorem 1 in][]{chen2021asymmetry}), does not precisely specify the locations of $\lambda_1, \lambda_2, \dots, \lambda_r$.
A common strategy is to leverage powerful tools from random matrix theory to show that, under certain conditions, the sample eigenvalues $\lambda_l$ are indeed close to their population counterparts $\lambdastar_l$ for all $l \in [r]$ with high probability.
\end{remark}


\section{NUMERICAL RESULTS} \label{sec:numeric}
In this section, we apply our results to three matrix estimation examples with $\mMstar = \lambdastar \vustar \vustart$ and verify our findings through numerical experiments.
All experiments were run in a distributed environment on commodity hardware without GPUs.
We focus on empirically verifying our findings in Corollary~\ref{cor:eigenvalue:one} for eigenvalue estimation.
Our primary goal is to see how the estimation error $|\lambda - \lambdastar|$ varies with $\mu$. 
To generate $\vustar$ for a given $\mu$, we consider two schemes:
\begin{enumerate}
    \item\label{scheme:I} Pick a random sparse unit vector $\vustar \in \R^n$ with $\|\vustar\|_{0} = \lfloor n/\mu \rfloor =: m$. 
    Its nonzero entries are sampled uniformly over $\bbS^{m-1}$. 
    \item\label{scheme:II} Generate an $m$-sparse vector $\vv^{(1)} \in \R^{n}$ whose nonzero entries are $\pm 1/ \sqrt{m}$ with equal probability.
    Generate another random unit vector $\vv^{(2)}\in \bbS^{n-1}$ and set $\vustar$ to be a linear combination $0.7 \vv^{(1)} + 0.3 \vv^{(2)}$, normalized to have unit norm.
\end{enumerate}
These two approaches both guarantee that with high probability, $\|\vu\|_{\infty} \asymp \sqrt{\mu/n}$ up to logarithmic factors. 

\subsection{Gaussian matrix denoising}
Suppose $\mH$ has independent entries drawn from $\calN(0,\sigma_{ij}^2)$, where the $\sigma_{ij}$ are generated independently from a uniform distribution over $[0.7, 1]$. 
This allows $\mH$ to be heteroskedastic, with $\sigma_{ij} \leq \sigma = 1$. 
Under this construction, $\mH$ satisfies condition~\ref{item:assump-3b} in Assumption~\ref{assump:1} with $B \asymp \sigma\sqrt{\log n}$, ensuring that Equation~\eqref{eq:B-sigma-bound} holds. 
We pick $\lambdastar = \sqrt{n \log n}$ to satisfy condition~\eqref{eq:lambdastar:assump}. 
For a given $\mu$, we follow scheme~\ref{scheme:I} to generate $\vustar$. 

Numerical results are shown in Figure~\ref{fig:gauss}.
Each line corresponds to a fitted linear model between the estimation error and $n$.
As $n$ increases, even though $\lambdastar$ scales as $\sqrt{n \log n}$, the average estimation error remains confined within a narrow range of $[0.6, 0.8]$, showing no clear trend or noticeable dependence on $\mu$. 
This behavior aligns well with our bound in Equation~\eqref{eq:eigenvalue:one}. 

\begin{figure}[ht]
    \centering
    \includegraphics[width=\linewidth]{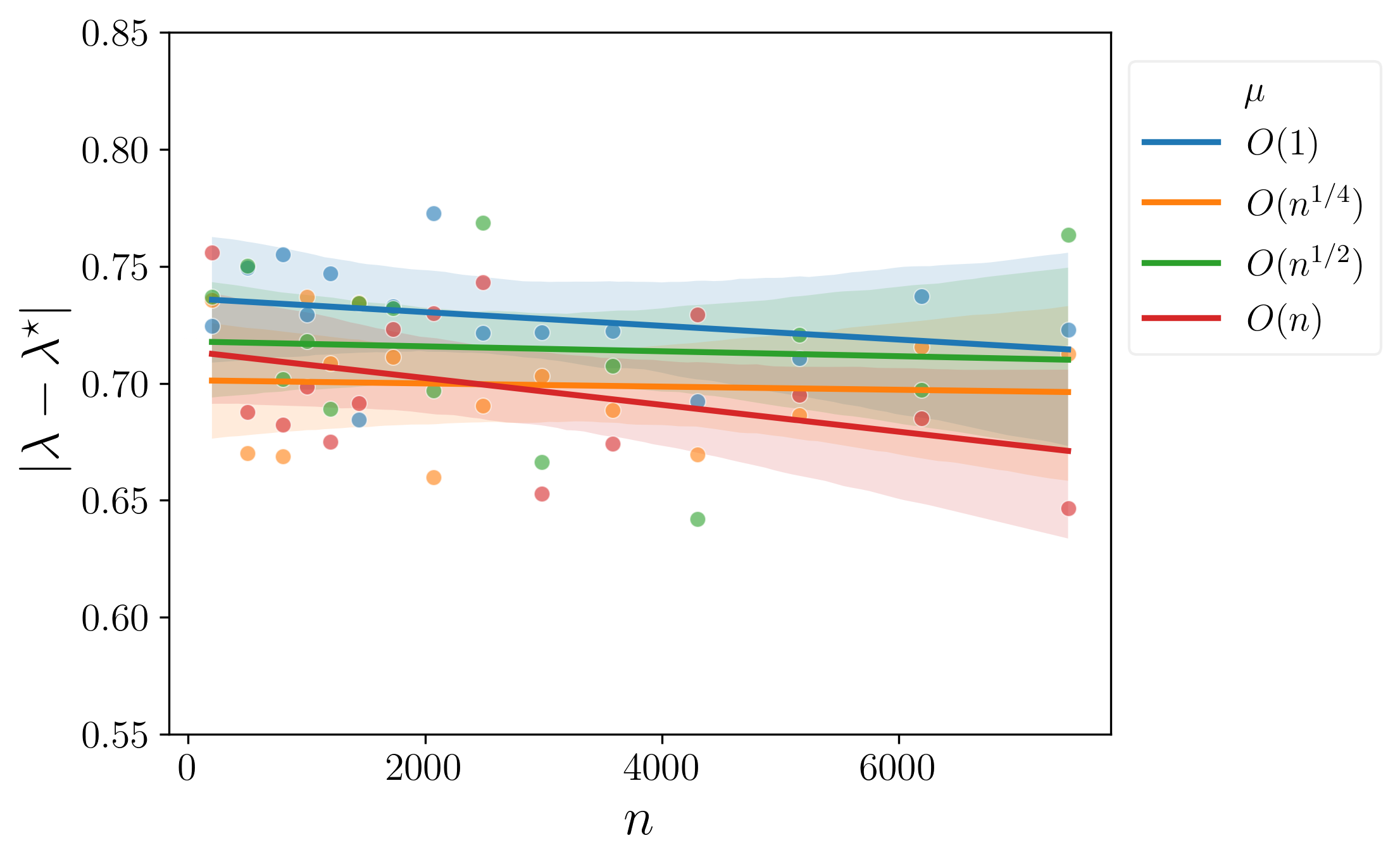}
    \caption{Eigenvalue estimation error by the sample eigenvalue of $\mM$ as a function of $n$. Different colors represent different levels of coherence $\mu \in \{O(1), O(n^{1/4}), O(n^{1/2}), O(n)\}$. Each point represents the average error over 300 independent trials. Shaded bands indicate $95\%$ bootstrap CIs for the fitted lines.}
    \label{fig:gauss} 
\end{figure} 

\subsection{Rank-one matrix completion}
Consider the setting in Example~\ref{ex:matrix-completion}.  
Although our result does not strictly apply here due to the violation of condition~\ref{assump:symm}, we expect a similar result to hold, and we verify it in this context.
We set $\lambdastar = 1$.
To satisfy condition~\eqref{eq:lambdastar:assump}, we require $p \gtrsim \mu^2\log n/n$. 
For $p$ to be in the range of $(0,1)$, we restrict $\mu \in [1, \sqrt{n}]$. 
Given a specific $\mu$, we generate $\vustar$ following scheme~\ref{scheme:II}. 
Setting $p = \mu^2 \log n/n$, our upper bound in Equation~\eqref{eq:eigenvalue:one} suggests that, up to some logarithmic factors, 
\begin{equation} \label{eq:mc-rate}
    |\lambda - \lambdastar| \lesssim \frac{1}{\sqrt{n}},
\end{equation}
indicating a coherence-free bound. 
The simulation results in Figure~\ref{fig:mc} support our findings in Equation~\eqref{eq:mc-rate}.
The estimation error lines for different choices of $\mu$ are almost parallel, indicating that the error rate has little dependence on $\mu$.
This behavior aligns well with the predicted coherence-free bound.

\begin{figure}[ht]
    \centering
    \includegraphics[width=\linewidth]{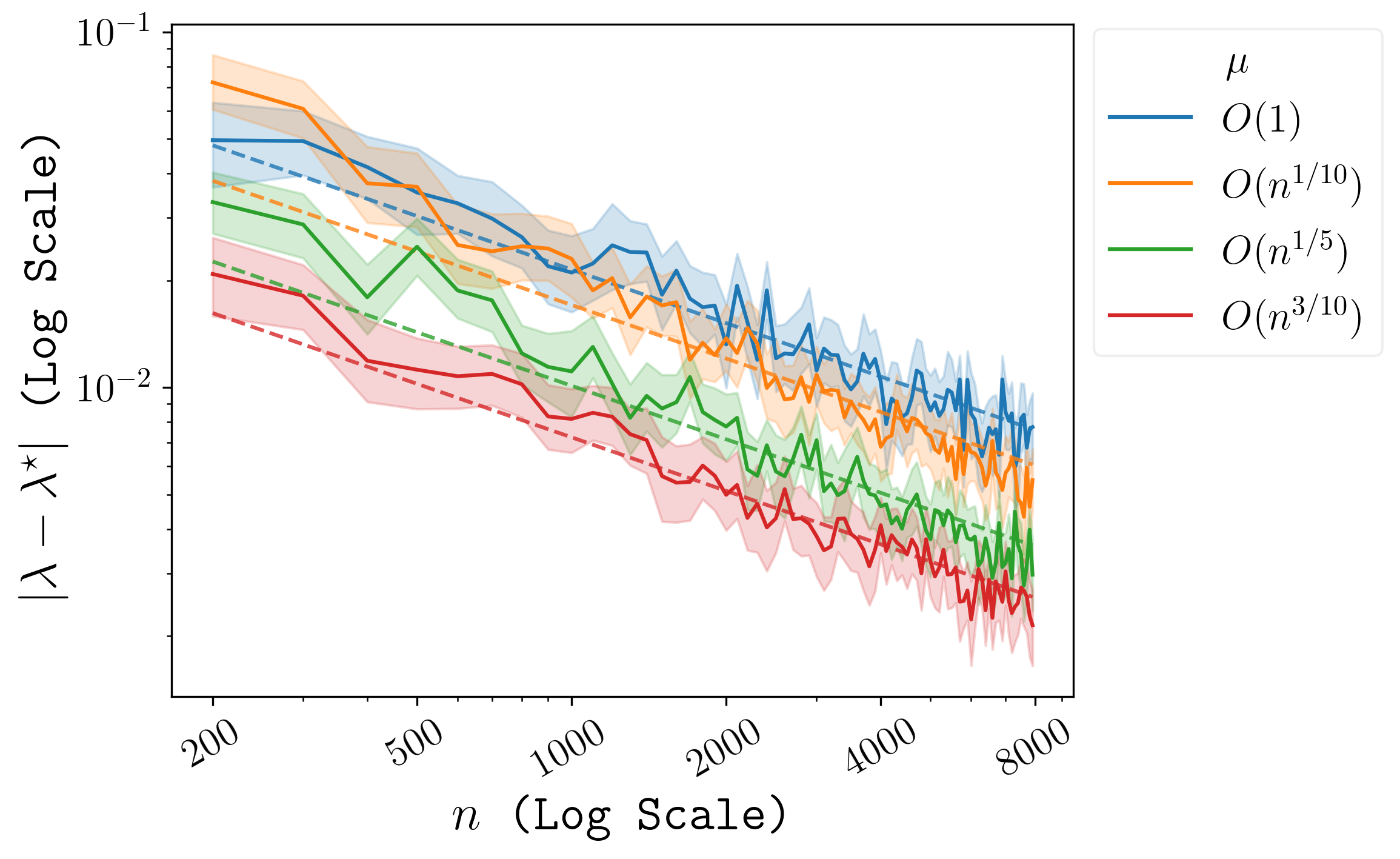}
    \caption{Eigenvalue estimation error as a function of $n$. Different colors represent varying levels of $\mu \in \{O(1), O(n^{0.1}), O(n^{0.2}), O(n^{0.3})\}$. The dashed lines correspond to the predicted rate from Equation~\eqref{eq:mc-rate}. Shaded bands indicate $95\%$ bootstrap CIs.}
    \label{fig:mc} 
\end{figure}

\subsection{Random rank-one network estimation}
Consider the setting in Example~\ref{ex:sparse-net}, with $\mP = \lambdastar \vustar \vustart$. 
In this case, $\rho \asymp \lambdastar \mu/n$. 
To satisfy condition~\eqref{eq:lambdastar:assump}, we require $\lambdastar \gtrsim \mu $. 
Additionally, for $\mP$ to have entries between $(0,1)$, we restrict $\mu \in [1, \sqrt{n}]$.
For a given $\mu$, we generate $\vustar$ following scheme~\ref{scheme:II}.
Taking $\lambdastar = \max\{\mu, \log n\}$, our upper bound in Equation~\eqref{eq:eigenvalue:one} suggests that, up to some logarithmic factors, 
\begin{equation} \label{eq:net-rate}
    |\lambda - \lambdastar| \lesssim \frac{\mu}{\sqrt{n}},
\end{equation}
which contrasts with the $O(n^{-1/2} \mu^{3/2})$ rate derived from Equation~\eqref{eq:ev-bound-suboptimal} (i.e., the suboptimal bound established in previous work).
While our bound indicates that the estimation error decreases as $n$ increases for all $\mu \in [1, \sqrt{n}]$, Equation~\eqref{eq:ev-bound-suboptimal} implies that the error should increase with $n$ when $\mu = \Omega(n^{1/3})$. 
The numerical results presented in Figure~\ref{fig:net} support our theoretical findings over the previously established bound.
The dashed lines indicate the rates predicted by Equation~\eqref{eq:net-rate}.
We see that they agree with the empirical error, with only minor deviations that are likely due to inaccuracies caused by logarithmic factors and small values of $n$.

\begin{figure}[ht]
    \centering
    \includegraphics[width=\linewidth]{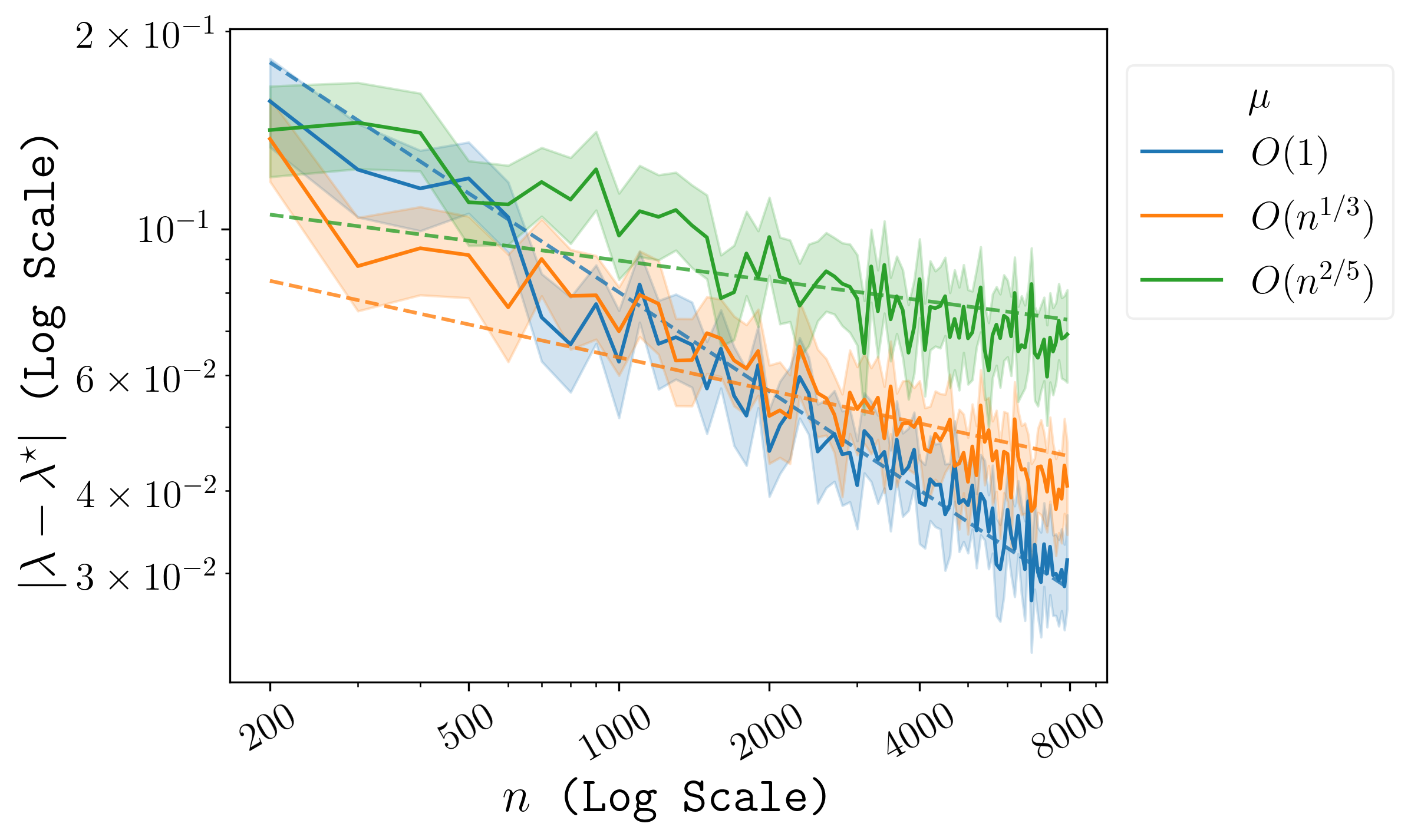}
    \caption{Eigenvalue estimation error as a function of $n$. Different colors represent different levels of coherence $\mu \in \{O(1), O(n^{1/3}), O(n^{2/5})\}$. The dashed lines indicate the predicted rate from Equation~\eqref{eq:net-rate}. Shaded bands indicate $95\%$ bootstrap CIs.} 
    \label{fig:net}
\end{figure}

\section{DISCUSSION} \label{sec:discussions}
We have presented a new concentration result in Theorem~\ref{thm:main}, which leads to upper bounds for eigenvalue and eigenvector estimation in signal-plus-noise matrix models with improved dependence on the coherence $\mu$.
Under some regimes, our upper bounds are sharp up to log-factors. 
Simulations for rank-one eigenvalue estimation further suggest that our new perturbation bounds accurately capture the correct dependence on $\mu$.
One limitation of our result is that it assumes that the noise distribution is symmetric about 0 in Assumption~\ref{assump:1}.
Future work will aim to relax this condition.
Additionally, while our result extends to symmetric noise matrices, a key challenge remains in removing the bias introduced by such matrices.
Future research will investigate potential methods to eliminate this bias.

\subsubsection*{Acknowledgements}
Support for this research was provided by the University of Wisconsin--Madison, Office of the Vice Chancellor for Research and Graduate Education.
In addition, HY was supported in part by NSF DMS 2023239 and KL was supported in part by NSF DMS 2052918.
The authors wish to thank Joshua Agterberg, 
Joshua Cape,
Junhyung Chang,
Rishabh Dudeja,
Pragya Sur
and Anru Zhang
for helpful discussions
and the anonymous reviewers for their suggestions and hard work, which greatly improved this paper.

\bibliography{bib}

\section*{Checklist}



 \begin{enumerate}

 \item For all models and algorithms presented, check if you include:
 \begin{enumerate}
   \item A clear description of the mathematical setting, assumptions, algorithm, and/or model. [{\bf Yes}/No/Not Applicable]
   See Sections~\ref{sec:intro} and~\ref{sec:setup}
   \item An analysis of the properties and complexity (time, space, sample size) of any algorithm. [{\bf Yes}/No/Not Applicable]
   Our results are centered on understanding the estimation rate of a certain class of spectral methods. Central to those results is dependence on the sample size. See Section~\ref{sec:main} and~\ref{sec:applications}.
   \item (Optional) Anonymized source code, with specification of all dependencies, including external libraries. [Yes/No/{\bf Not Applicable}]
   This is a theory paper; reported experiments are straightforward to implement in any programming language, though we will be happy to provide code if deemed necessary.
 \end{enumerate}

 \item For any theoretical claim, check if you include:
 \begin{enumerate}
   \item Statements of the full set of assumptions of all theoretical results. [{\bf Yes}/No/Not Applicable]
   See our main results in Section~\ref{sec:main} as well as setup and background discussed in Section~\ref{sec:setup}.
   \item Complete proofs of all theoretical results. [{\bf Yes}/No/Not Applicable]
   Detailed proofs are included in the supplemental materials.
   \item Clear explanations of any assumptions. [{\bf Yes}/No/Not Applicable]     
   See Section~\ref{sec:main} as well as setup and background discussed in Section~\ref{sec:setup}.
 \end{enumerate}

 \item For all figures and tables that present empirical results, check if you include:
 \begin{enumerate}
   \item The code, data, and instructions needed to reproduce the main experimental results (either in the supplemental material or as a URL). [{\bf Yes}/No/Not Applicable]
   Instructions for our simulation setup are included, and are straightforward to implement in any modern programming language. 
   \item All the training details (e.g., data splits, hyperparameters, how they were chosen). [Yes/No/{\bf Not Applicable}]
         \item A clear definition of the specific measure or statistics and error bars (e.g., with respect to the random seed after running experiments multiple times). [{\bf Yes}/No/Not Applicable]
         See Section~\ref{sec:numeric}, specifically the figure captions.
         \item A description of the computing infrastructure used. (e.g., type of GPUs, internal cluster, or cloud provider). [{\bf Yes}/No/Not Applicable]
         See Section~\ref{sec:numeric}. 
 \end{enumerate}

 \item If you are using existing assets (e.g., code, data, models) or curating/releasing new assets, check if you include:
 \begin{enumerate}
   \item Citations of the creator If your work uses existing assets. [Yes/No/{\bf Not Applicable}]
   This is a theory paper and does not use existing assets.
   \item The license information of the assets, if applicable. [Yes/No/{\bf Not Applicable}]
   This is a theory paper and does not use existing assets.
   \item New assets either in the supplemental material or as a URL, if applicable. [Yes/No/{\bf Not Applicable}]
   This is a theory paper and does not use existing assets.
   \item Information about consent from data providers/curators. [Yes/No/{\bf Not Applicable}]
   This is a theory paper and does not use existing assets.
   \item Discussion of sensible content if applicable, e.g., personally identifiable information or offensive content. [Yes/No/{\bf Not Applicable}]
   This is a theory paper and does not use existing assets.
 \end{enumerate}

 \item If you used crowdsourcing or conducted research with human subjects, check if you include:
 \begin{enumerate}
   \item The full text of instructions given to participants and screenshots. [Yes/No/{\bf Not Applicable}]
   This is a theory paper and does not use human subjects.
   \item Descriptions of potential participant risks, with links to Institutional Review Board (IRB) approvals if applicable. [Yes/No/{\bf Not Applicable}]
   This is a theory paper and does not use human subjects.
   \item The estimated hourly wage paid to participants and the total amount spent on participant compensation. [Yes/No/{\bf Not Applicable}]
   This is a theory paper and does not use human subjects.
 \end{enumerate}

 \end{enumerate}

\appendix
\numberwithin{equation}{section}
\onecolumn

\section{PROOF OF THEOREM~\ref{thm:IExHky-moment}}
We first present an outline of the proof of Theorem~\ref{thm:IExHky-moment}. 
Some technical lemmas are deferred to Section~\ref{sec:thm2:tech}. 
\begin{proof}
We begin by outlining our proof strategy.
First, we will expand the left-hand side of Equation~\eqref{eq:moment-bound-2p}.
Many terms in this expansion have an expectation of zero.
For the terms that do not disappear, we establish an upper bound.
This upper bound will depend on the specifics of each term, roughly corresponding to how many times different entries of $\mH$ appear in the product.
We will characterize how this structure corresponds to upper bound expressions in $\vx,\vy,\sigma$ and $B$.
We will use graphical tools to count how many terms in this expansion have each such possible upper bound.
Finally, by combining these elements, we arrive at Equation~\eqref{eq:moment-bound-2p}. 
Detailed steps are provided below. 

\textbf{Step 1.} Our first step is to expand the left hand side of Equation~\eqref{eq:moment-bound-2p}. 

For any $\vj = (j_1,j_2,\dots,j_{k+1}) \in [n]^{k+1}$, define 
\begin{equation} \label{eq:vzeta-vj-define}
    \vzeta_{\vj} := H_{j_1 j_2} H_{j_2 j_3} \cdots H_{j_k j_{k+1}}
\end{equation}
we have 
\begin{equation*}
    \vx^\top \mH^k \vy - \E \vx^\top \mH^k \vy  = \sum_{\vj \in [n]^{k+1}} x_{j_{1}} y_{j_{k+1}} \left(\vzeta_{\vj} - \E \vzeta_{\vj}\right).
\end{equation*}
Expanding $\E \left(\vx^\top \mH^k \vy - \E \vx^\top \mH^k \vy\right)^{p}$, we have 
\begin{equation}\label{eq:moment-i}
\begin{aligned}
    \E  &\left[\prod_{r=1}^p \sum_{\vjj{r}\in [n]^{k+1}} x_{\jr{1}} y_{\jr{k+1}} \left(\vzeta_{\vjj{r}} - \E \vzeta_{\vjj{r}}\right) \right] \\
    &= \sum_{\vjj{1}\in [n]^{k+1}}\cdots \sum_{\vjj{p}\in [n]^{k+1}} \E \Bigg[\prod_{r=1}^p x_{j^{(r)}_1} y_{\jr{k+1}} \left(\vzeta_{\vjj{r}} - \E \vzeta_{\vjj{r}}\right)\Bigg] \\
    &= \sum_{\vjj{1}\in [n]^{k+1}}\cdots \sum_{\vjj{p}\in [n]^{k+1}} \left(\prod_{r=1}^p x_{j^{(r)}_1} y_{\jr{k+1}}\right) \E \Bigg[\prod_{r=1}^p \left(\vzeta_{\vjj{r}} - \E \vzeta_{\vjj{r}}\right)\Bigg].
\end{aligned}
\end{equation}

As mentioned in the beginning of the proof, our key technique in establishing Equation~\eqref{eq:moment-bound-2p} is the use of graphical tools. 

\textbf{Step 2.} 
Every $\mJ = (\vjj{1},\vjj{2},\dots,\vjj{p}) \in [n]^{p(k+1)}$
has a corresponding term 
\begin{equation} \label{eq:single-moment}
        \E \Bigg[\prod_{r=1}^p \left(\vzeta_{\vjj{r}} - \E \vzeta_{\vjj{r}}\right)\Bigg]
    \end{equation} 
appearing in Equation~\eqref{eq:moment-i}, which we associate with an undirected, hollow graph $G$ determined by the choice of $\mJ = (\vjj{1},\vjj{2},\dots,\vjj{p}) \in [n]^{p(k+1)}$. 
The purpose of this is two-fold, 
\begin{enumerate}
    \item Many of the terms in Equation~\eqref{eq:moment-i} have expectation zero.
    Thus, we wish to characterize when Equation~\eqref{eq:single-moment} is nonzero. 
    This is addressed in Lemma~\ref{lem:G-cc-size-2}, proved below in Section~\ref{sec:moment-bounds}.
    \item When the quantity in Equation~\eqref{eq:single-moment} is nonzero, its upper bound will depend on the choice of $\mJ \in [n]^{p(k+1)}$, as $\mJ$ determines how many times different entries of $\mH$ appear in the product.
    Lemma~\ref{lem:single-moment-bound}, proved in Section~\ref{sec:moment-bounds}, characterizes these bounds, and Lemma~\ref{lem:term-bound}, provided in Section~\ref{sec:results}, counts how many $\mJ$ share the same upper bound.
\end{enumerate}

The graph $G$ has $pk$ vertices, which we identify with the elements of the set 
\begin{equation} \label{eq:def:calI}
    \calI = \left\{(r, l): r \in [p], \; l \in [k]\right\}.
\end{equation}
Given $\mJ\in [n]^{p(k+1)}$, we construct the edges of $G$ according to the following rule:
\begin{enumerate} [label=(\roman*)]
    \item \label{rule:(i)} For $(r_1,l_1),(r_2, l_2) \in \calI$ distinct, we form an (undirected) edge between $(r_1,l_1)$ and $(r_2, l_2)$ if and only if $\left(j^{(r_1)}_{l_1}, j^{(r_1)}_{l_1+1}\right) = \left(j^{(r_2)}_{l_2}, j^{(r_2)}_{l_2+1}\right)$.
\end{enumerate}
Denote by $\calG(\calI)$ the set of all undirected hollow graphs with vertex set $\calI$, and let $\psi: [n]^{p(k+1)} \to \calG(\calI)$ denote the function that maps a particular $\mJ$ to an associated graph $G \in \calG(\calI)$ according to the construction in Rule~\ref{rule:(i)}. 
We also let $\psi^{-1}$ to denote the inverse mapping of $\psi$, which maps a given graph $G \in \calG(\calI)$ to the set of all $\mJ$ that produces the same $G$ under Rule~\ref{rule:(i)}.

\textbf{Step 3.} With the construction of $G$, we turn our attention back to controlling Equation~\eqref{eq:moment-i}.
By Lemma~\ref{lem:single-moment-bound}, two graphs constructed from two different choices of $\mJ$ (i.e., two different expectation terms in the sum in Equation~\eqref{eq:moment-i}, corresponding to different instantiations of the expression in Equation~\eqref{eq:single-moment}) share the same upper bound if they have the same number of connected components. 
Letting $c(G)$ denote the number of connected components in $G$, we may reorganize the sum in Equation~\eqref{eq:moment-i} as 
\begin{equation*}
\begin{aligned}
    \sum_{L=1}^{pk/2} \sum_{G: c(G) = L} \sum_{\mJ \in \psi^{-1}(G)} \left(\prod_{r=1}^p x_{j^{(r)}_1} y_{\jr{k+1}}\right) \E \Bigg[\prod_{r=1}^p \left(\vzeta_{\vjj{r}} - \E \vzeta_{\vjj{r}}\right)\Bigg],
\end{aligned}
\end{equation*}
where the innermost sum runs over all $\mJ \in [n]^{p(k+1)}$ such that $\psi(\mJ) = G$, the second sum runs over all graphs $G$ with the same number of connected components $L$ and the outermost sum runs over all possible values of $L$.
The upper bound $L \leq pk/2$ comes from Lemma~\ref{lem:G-cc-size-2}. 

Taking absolute value and applying the triangle inequality,
\begin{equation}\label{eq:grand-decomposition}
\begin{aligned}
    &\left| \sum_{L=1}^{pk/2} \sum_{G: c(G) = L} \sum_{\mJ \in \psi^{-1}(G)} \left(\prod_{r=1}^p x_{j^{(r)}_1} y_{\jr{k+1}}\right) \E \Bigg[\prod_{r=1}^p \left(\vzeta_{\vjj{r}} - \E \vzeta_{\vjj{r}}\right)\Bigg] \right| \\
    &~~~~~~~~~\le \sum_{L=1}^{pk/2} \sum_{G: c(G) = L} \sum_{\mJ \in \psi^{-1}(G)}
    \left(\prod_{r=1}^p \left|x_{j^{(r)}_1} y_{\jr{k+1}}\right|\right) \left|\E \Bigg[\prod_{r=1}^p \left(\vzeta_{\vjj{r}} - \E \vzeta_{\vjj{r}}\right)\Bigg]\right|\\
    &~~~~~~~~~\leq \sum_{L=1}^{pk/2} \sum_{G: c(G) = L} \sum_{\mJ \in \psi^{-1}(G)} \|\vx\|_{\infty}^p \|\vy\|_{\infty}^p 2^p \sigma^{2 L} B^{p k-2 L},
\end{aligned}
\end{equation}
where the second inequality follows from Lemma~\ref{lem:single-moment-bound}. 
Applying Lemmas~\ref{lem:G-cc-count} (proved in Section~\ref{sec:moment-bounds}) and~\ref{lem:term-bound}, we obtain the bound
\begin{equation*}
\begin{aligned}
    &2^{p(k+1)} \|\vx\|_{\infty}^{p} \|\vy\|_{\infty}^{p} \sum_{L=1}^{pk/2}  (\sigma^2L)^{L} (BL)^{p k-2 L} \tau_L,
\end{aligned}
\end{equation*}
where $\tau_L$ is defined in Equation~\eqref{eq:term-bound} of Lemma~\ref{lem:term-bound} in Section~\ref{sec:results}, which bounds the number of terms associated with $G$ when $c(G) = L$. 

\textbf{Step 4.} We expand $\tau_L$ and simplify the expression. Letting $N_x = \|\vx\|_0$ and $N_y = \|\vy\|_0$, and define $S_{xy} = |\{i: x_i \neq 0, y_i \neq 0, i\in [n]\}|$, expanding $\tau_L$ according to Equation~\eqref{eq:term-bound} yields that Equation~\eqref{eq:single-moment} is bounded by
\begin{equation*}
\begin{aligned}
    2^{p(k+1)} (Bkp)^{pk} \|\vx\|_{\infty}^{p} \|\vy\|_{\infty}^{p} & \Biggl( \frac{\sigma^2 S_{xy}}{B^2 kp} + \sum_{L=2}^{k-1} \left(\frac{\sigma^2 n}{B^2 kp}\right)^L \frac{N_x N_y}{n^2} + \sum_{t=1}^{p/2} \left(\frac{\sigma^2 n}{B^2 kp}\right)^{tk} \left(\frac{N_x N_y}{n}\right)^t \\
    &~~~~~~~~~~~+ \sum_{t=2}^{p/2} \sum_{L=(t-1)k+1}^{tk-1} \frac{1}{N_x}\left(\frac{\sigma^2 n}{B^2kp}\right)^{L} \left(\frac{N_x N_y}{n}\right)^t \Biggr).
\end{aligned}
\end{equation*}
Under the assumption in Equation~\eqref{eq:B-sigma-bound}, we have 
\begin{equation*}
    \sum_{L=a}^{b} \left(\frac{\sigma^2 n}{B^2 kp}\right)^L \lesssim \left(\frac{\sigma^2 n}{B^2 kp}\right)^b 
\end{equation*}
for any $a \leq b$. 
Thus, the above term inside the bracket can be simplified to
\begin{equation}\label{eq:middle:bound:I}
\begin{aligned}
    &\frac{\sigma^2 S_{xy}}{B^2 kp} 
    + 
    \sum_{t=1}^{p/2} 
    \left(\frac{N_x N_y}{n}\right)^t \left(\frac{\sigma^2 n}{B^2kp}\right)^{tk}
\end{aligned}
\end{equation}
up to some constant factor. 
Furthermore,  since
\begin{equation*}
    \frac{\sigma^2 S_{xy}}{B^2 kp} \leq \frac{\sigma^2 (N_x N_y)^{1/2}}{B^2 kp}
\end{equation*}
and
\begin{equation*}
    \left(\frac{N_x N_y}{n}\right) \left(\frac{\sigma^2 n}{B^2kp}\right)^{k} \geq \left(\frac{N_x N_y}{n}\right) \left(\frac{\sigma^2 n}{B^2kp}\right) = \frac{\sigma^2 N_x N_y}{B^2 kp}  \geq \frac{\sigma^2 (N_x N_y)^{1/2}}{B^2 kp},
\end{equation*}
Equation~\eqref{eq:middle:bound:I} can be further bounded by 
\begin{equation} \label{eq:middle:bound:II}
    \sum_{t=1}^{p/2} \left(\frac{\sigma^{2} n^{1-1/k} (N_x N_y)^{1/k}}{B^{2} kp}\right)^{kt}
\end{equation}
up to some constant. 

\textbf{Step 5.} The bound in Equation~\eqref{eq:middle:bound:II} can be further simplified under different conditions.
If $B^2 kp \leq \sigma^{2} n^{1-1/k} (N_x N_y)^{k}$, Equation~\eqref{eq:middle:bound:II} is bounded by
\begin{equation*}
    \frac{p}{2} \left(\frac{\sigma^{2} n^{1-1/k} (N_x N_y)^{1/k}}{B^{2} kp}\right)^{kp/2} = \frac{p\sigma^{kp} n^{\frac{kp}{2} - \frac{p}{2}} (N_x N_y)^{p/2} }{2B^{kp} (kp)^{kp/2}}
\end{equation*}
Otherwise, it is bounded by
\begin{equation*}
    \frac{p}{2} \left(\frac{\sigma^{2} n^{1-1/k} (N_x N_y)^{1/k}}{B^{2} kp}\right)^{k} = \frac{p\sigma^{2k} n^{k-1} (N_x N_y)}{2B^{2k} (kp)^k}. 
\end{equation*}

Combining the above two cases, Equation~\eqref{eq:moment-i} is up to some constant factor bounded by
\begin{equation*}
    p 2^{p(k+1)} ~(k p)^{k p} ~ (\|\vx\|_{\infty}\|\vy\|_{\infty})^p ~\max\biggl\{
    \frac{\sigma^{k p} n^{\frac{k p}{2}-\frac{p}{2}}\left(N_x N_y\right)^{p / 2}}{(k p)^{k p / 2}},  \frac{ B^{(p-2) k}\sigma^{2 k} n^{k-1} N_x N_y}{(k p)^k} \biggr\}
\end{equation*}
as desired.
\end{proof}
\section{PROOF OF COROLLARY~\ref{cor:sparse-bound}}
\begin{proof}
    Similar to the proof of Corollary 4 in \cite{chen2021asymmetry}, we assume that $20 \log n$ is an integer to avoid the clumsy notation $\lfloor 20 \log n\rfloor$. 
    Extending to the case where $20 \log n$ is not an integer is straightforward.

Applying a high-order Markov inequality for any even integer $p$ and $t > 0$,
\begin{equation*}
    \Pr\left(\left|\vx^\top \mH^k \vy - \E \vx^\top \mH^k \vy \right| \geq t \right) \leq \frac{\E\left(\vx^{\top} \mH^k \vy - \E \vx^{\top} \mH^k \vy\right)^p}{t^p}.
\end{equation*}
Using Theorem~\ref{thm:IExHky-moment} on the right-hand side yields the bound  
\begin{equation*}
\begin{aligned}
    &p 2^{p(k+1)} (\|\vx\|_{\infty}\|\vy\|_{\infty})^p \max\biggl\{
    \frac{(\sigma^{2} n kp)^{kp/2} n^{-\frac{p}{2}}\left(\|\vx\|_0 \|\vy\|_0\right)^{p / 2}}{t^p},  \frac{ (Bkp)^{kp } (\sigma^2 n/B^2 kp)^{k} n^{-1}\|\vx\|_0 \|\vy\|_0}{t^p} \biggr\}.
\end{aligned}
\end{equation*}
Taking 
\begin{equation*}
\begin{aligned}
    t &= c_2^k \|\vx\|_{\infty}\|\vy\|_{\infty} \max\Biggl\{(\sigma^2 n \log^3 n)^{\frac{k}{2}} \left(\frac{\|\vx\|_0 \|\vy\|_0}{n}\right)^{1/ 2}, \left(\frac{\sigma^2 n}{B^2}\right)^{\frac{k}{p}} \left(\frac{\|\vx\|_0 \|\vy\|_0}{n}\right)^{\frac{1}{p}} (B\log^3 n)^k\Biggr\}
\end{aligned}
\end{equation*}
into the above display yields that 
\begin{equation*}
    p 2^{p(k+1)}\max \left\{ \frac{(kp)^{\frac{kp}{2}}}{c_2^{kp} (\log^3 n)^{\frac{kp}{2}}}, \frac{(kp)^{kp-k}}{c_2^{kp} (\log^3 n)^{kp}} \right\}
    \leq p \max \left\{ \left(\frac{16 kp}{c_2^{2} \log^3 n}\right)^{\frac{kp}{2}}, (kp)^{-k} \left(\frac{4kp}{c_2\log^3 n}\right)^{kp} \right\} .
\end{equation*}
For any $k \leq 20\log n$, choosing $p = k^{-1}\log^3 n$ (here again, we assume that $p$ is an integer without loss of generality), the above display is bounded by
\begin{equation*}
    p \max \left\{ \left(\frac{1280}{c_2^{2}}\right)^{20\log n}, \frac{1}{(kp)^{k}} \left(\frac{320^2}{c_2^2}\right)^{20\log n} \right\},
\end{equation*}
which is bounded by $O(n^{-40} \log^3 n)$ for some constant $c_2$ sufficiently large. 
Therefore, applying a union bound over $2 \leq k \leq 20 \log n$, we have Equation~\eqref{eq:sparse-bound} holds for all $2 \leq k \leq 20 \log n$ with probability at least $1 - O(n^{-40} \log^4 n)$.  
Noting that $n = e^{\log n}$, we obtain Equation~\eqref{eq:sparse-bound}. 
\end{proof}
\section{PROOF OF THEOREM~\ref{thm:main}}
\begin{proof}
We start by grouping the entries in $\vx$ and $\vy$ according to their magnitudes. 
To avoid the clumsy notation $\lfloor \frac{1}{2}\log n \rfloor$, we again assume that $\frac{1}{2}\log n$ is an integer. 
Extending the proof to the case where $\frac{1}{2}\log n$ is not an integer is straightforward and we omit the details here.

Set $m = \frac{1}{2}\log n$ and let 
\begin{equation*}
    I_1 = \left(e^{-1}, 1\right], \; I_2 = \left(e^{-2}, e^{-1}\right], \cdots, I_{m} = \left(e^{-m}, e^{-m+1}\right], 
\end{equation*}
and
\begin{equation*}
    I_{m+1} = \left[0, e^{-\frac{\log n}{2}}\right] = \left[0, \frac{1}{\sqrt{n}}\right]. 
\end{equation*}
For any $i_1, i_2 \in [m+1]$, we construct vectors $\vxb^{(i_1)}$ and $\vyb^{(i_2)}$ by 
\begin{equation*}
    \xb^{(i_1)}_l = 
    \begin{cases}
        x_l, & \mbox{ if } |x_l| \in I_{i_1},\\
        0 & \mbox{ otherwise }
    \end{cases}
~~~\text{ and } ~~~
    \yb^{(i_2)}_l = 
    \begin{cases}
        y_l, & \mbox{ if } |y_l| \in I_{i_2}, \\
        0 & \mbox{ otherwise }.
    \end{cases}
\end{equation*}
We can then decompose $\vx$ and $\vy$ into
\begin{equation} \label{eq:unit:decomposition}
    \vx = \sum_{i_1=1}^{m+1} \vxb^{(i_1)} \quad \text{and} \quad \vy = \sum_{i_2=1}^{m+1} \vyb^{(i_2)}.
\end{equation}

By definition, we have
\begin{equation} \label{eq:decomposition:max-bound}
    \left\|\vxb^{(r)}\right\|_{\infty} \leq e\cdot e^{-r} \quad \text{for all } r \in [m], \quad \text{and } \left\|\vxb^{(m+1)}\right\|_{\infty} \leq n^{-1/2}. 
\end{equation}
Since $\|\vx\|^2_2 = 1$, it follows that for all $r \in [m]$, we have $e^{-2r}\left\|\vxb^{(r)}\right\|_{0} \leq 1$, from which
\begin{equation} \label{eq:decomposition:l0-bound}
\left\|\vxb^{(r)}\right\|_{0} \leq e^{2r}
\end{equation}
and trivially, $\|\vxb^{(m+1)}\|_{0} \leq n$. 

Thus, by Equation~\eqref{eq:unit:decomposition}, we have the decomposition
\begin{equation*}
\begin{aligned}
    &\vx^\top \mH^k \vy - \E \vx^\top \mH^k \vy = \sum_{i_1 \in [m+1]} \sum_{i_2 \in [m+1]} \left(\vxb^{(i_1) \top} \mH^k \vyb^{(i_2)} - \E \vxb^{(i_1) \top} \mH^k \vyb^{(i_2)} \right).
\end{aligned}
\end{equation*}
Taking absolute value on both sides and applying the triangle inequality, we have
\begin{equation*}
\left| \vx^\top \mH^k \vy - \E \vx^\top \mH^k \vy \right|
\le
    \left(\frac{\log n}{2}+1\right)^2 \max_{i_1, i_2 \in [m+1]} \left|\vxb^{(i_1) \top} \mH^k \vyb^{(i_2)} - \E \vxb^{(i_1) \top} \mH^k \vyb^{(i_2)} \right|.
\end{equation*}
Combining Equation~\eqref{eq:sparse-bound} in Corollary~\ref{cor:sparse-bound} with the union bound, the above is bounded up to constant factors by
\begin{equation}\label{eq:main-middle}
\begin{aligned}
    \left( c_2^k \log^{2} n \right) \max_{i_1, i_2 \in [m+1]}&\left\|\vxb^{(i_1)}\right\|_{\infty}\left\|\vyb^{(i_2)}\right\|_{\infty} \Biggl\{
    (B\log^3 n)^k,
    \left(\sigma^2 n \log^3 n\right)^{\frac{k}{2}} \sqrt{\frac{1}{n}} \left(\left\|\vxb^{(i_1)}\right\|_{0} \left\|\vyb^{(i_2)}\right\|_0\right)^{\frac{1}{2}}
     \Biggr\} 
\end{aligned}
\end{equation}
with probability at least $1 - O(n^{-40} \log^6 n)$. 
By Equations~\eqref{eq:decomposition:max-bound} and~\eqref{eq:decomposition:l0-bound}, we have for all $1 \leq i_1, i_2 \leq m+1$,
\begin{equation} \label{eq:construction}
    \left(\left\|\vxb^{(i_1)}\right\|_{0} \left\|\vyb^{(i_2)}\right\|_0\right)^{\frac{1}{2}}  \left\|\vxb^{(i_1)}\right\|_{\infty}\left\|\vyb^{(i_2)}\right\|_{\infty} \leq e^2,
\end{equation} %
Thus, up to some constant factor, the bounding quantity in Equation~\eqref{eq:main-middle} can be simplified to 
\begin{equation}\label{eq:simplified-I}
\begin{aligned}
    &\left( c_2^k \log^2 n \right)  \max_{i_1, i_2 \in [m+1]}\Biggl\{\left(\sigma^2 n \log^3 n\right)^{\frac{k}{2}} \sqrt{\frac{1}{n}}, (B\log^3 n)^k \left\|\vxb^{(i_1)}\right\|_{\infty}\left\|\vyb^{(i_2)}\right\|_{\infty} \Biggr\}.
\end{aligned}
\end{equation}
Since by construction, for all $1 \leq i_1, i_2 \leq m+1$, 
\begin{equation*}
    \left\|\vxb^{(i_1)}\right\|_{\infty}\left\|\vyb^{(i_2)}\right\|_{\infty} \leq \left\|\vx\right\|_{\infty}\left\|\vy\right\|_{\infty}, 
\end{equation*}
substitute $\|\vx\|_{\infty}\|\vy\|_{\infty}$ into the above bound yields Equation~\eqref{eq:main-bound}. 
\end{proof}
\section{PROOF OF LEMMAS~\ref{lem:xHky-mean} and~\ref{lem:xHy-bound}}
Here, we prove Lemmas~\ref{lem:xHky-mean} and~\ref{lem:xHy-bound}, which are used to control the expectation of $\vx^\top \mH^k \vy$.
To establish Lemma~\ref{lem:xHky-mean}, we first establish a coarse bound in the following technical lemma. 
\begin{lemma} \label{lem:xHky-mean-sparse}
    Under the same setting as Theorem~\ref{thm:IExHky-moment}, for all $2\leq k\leq 20\log n$, 
    \begin{equation*}
    \begin{aligned}
        &\left| \E \left[\vx^\top \mH^k \vy\right] \right| \lesssim (2k)^k \|\vx\|_{\infty} \|\vy\|_{\infty} \max\left\{ \frac{\sigma^2 B^{k-2} S_{xy}}{k}, \left(\frac{\sigma^2 n}{k}\right)^{k/2} \frac{\|\vx\|_0 \|\vy\|_0}{n^2}\right\},
    \end{aligned}
    \end{equation*}
    where $S_{x y}=\left|\left\{i: x_i \neq 0, y_i \neq 0, i \in[n]\right\}\right|$.
\end{lemma}
\begin{proof}[Proof of Lemma~\ref{lem:xHky-mean-sparse}]
Expanding the term $\E \left[\vx^\top \mH^k \vy\right]$, we have 
\begin{equation} \label{eq:expand-k}
\E \left[\vx^\top \mH^k \vy\right] =
    \sum_{\vj \in [n]^{p+1}} x_{j_1} y_{j_{k+1}} \E \vzeta_{\vj}.
\end{equation}
where $\vzeta_{\vj}$ is defined in Equation~\eqref{eq:vzeta-vj-define}. 
For $\E \vzeta_{\vj}$, similar to the proof of Theorem~\ref{thm:IExHky-moment}, we create a graph $G$ with $k$ vertices.
By Lemma~\ref{lem:G-cc-count}, the number $L$ of connected components of $G$ is bounded by $k/2$. 
Accordingly, we rewrite Equation~\eqref{eq:expand-k} as
\begin{equation*}
\E \left[\vx^\top \mH^k \vy\right] 
= \sum_{L=1}^{k/2} \sum_{G:c(G)=L} \sum_{\vj \in \psi^{-1}(G)} x_{j_1} y_{j_{k+1}} \E \vzeta_{\vj}
\end{equation*}
Taking the absolute value and applying the triangle inequality to the above display yields 
\begin{equation*}
\left| \E \left[\vx^\top \mH^k \vy\right] \right|
\lesssim
    2^k \|\vx\|_{\infty} \|\vy\|_{\infty} \sum_{L=1}^{k/2} \left(\frac{\sigma^2 k}{2}\right)^{L} \left(\frac{Bk}{2}\right)^{k-2 L}  \tau_L 
\end{equation*}
up to some constant, where $\tau_L$ is as defined in Equation~\eqref{eq:term-bound} of Lemma~\ref{lem:term-bound} in Section~\ref{sec:results}. 
Similar to the proof of Theorem~\ref{thm:IExHky-moment}, we let $N_x=\|\vx\|_0$ and $N_y=\|\vy\|_0$.
Applying Lemma~\ref{lem:term-bound} (see Section~\ref{sec:thm2:tech} below), the right hand side of the above display is bounded by
\begin{equation*}
\begin{aligned}
\left| \E \left[\vx^\top \mH^k \vy\right] \right|
&\lesssim
    \|\vx\|_{\infty} \|\vy\|_{\infty} 2^k (Bk)^{k} \left( \frac{\sigma^2 S_{xy}}{B^2 k} + \frac{N_x N_y}{n^2} \sum_{L=2}^{k/2}  \left(\frac{\sigma^{2} n}{B^{2} k}\right)^L\right)\\
    &\lesssim (2k)^k \|\vx\|_{\infty} \|\vy\|_{\infty} \left( \frac{\sigma^2 S_{xy}}{k} B^{k-2} + \frac{ \sigma^k N_x N_y n^{k/2-2} }{ k^{k/2} } \right),
\end{aligned}
\end{equation*}
where the second inequality follows from the assumption that $B \ll \sigma \sqrt{n/\log^3 n}$. 
Lemma~\ref{lem:xHky-mean-sparse} then follows from applying the trivial inequality $a+b \leq 2\max\{a, b\}$ to the above display. 
\end{proof}

\begin{proof}[Proof of Lemma~\ref{lem:xHky-mean}]
    Similar to the proof of Theorem~\ref{thm:main}, we adopt the decomposition in Equation~\eqref{eq:unit:decomposition}. 
    Expanding $\E \left[\vx^\top \mH^k \vy\right]$ and applying the triangle inequality,
    \begin{equation} \label{eq:middle_step}
    \begin{aligned}
        \left|\E \left[\vx^\top \mH^k \vy\right]\right| &\leq \sum_{i_1 \in [m+1]} \sum_{i_2 \in [m+1]} \left|\E \vxb^{(i_1)\top} \mH^k \vyb^{(i_2)}\right|\\
        &\leq (m+1)^2 \max_{1\leq i_1, i_2 \leq m+1} \left\{\left|\E \vxb^{(i_1)\top} \mH^k \vyb^{(i_2)}\right|\right\}.
    \end{aligned}
    \end{equation}
    Let $S^{(i_1, i_2)}$ denote the intersection of the support of $\vx$ and $\vy$.
    Noting that 
    \begin{equation*}
        S^{(i_1, i_2)} \left\|\vx^{(i_1)}\right\|_{\infty} \left\|\vy^{(i_2)}\right\|_{\infty} \leq \left(\left\|\vx^{(i_1)}\right\|_{0} \left\|\vy^{(i_2)}\right\|_{0}\right)^{1/2}\left\|\vx^{(i_1)}\right\|_{\infty} \left\|\vy^{(i_2)}\right\|_{\infty} 
        \leq e^2, 
    \end{equation*}
    where the last inequality follows from Equation~\eqref{eq:construction}. 
    Combining the above display with Lemma~\ref{lem:xHky-mean-sparse}, Equation~\eqref{eq:middle_step} is bounded by
    \begin{equation*}
    \begin{aligned}
        (\log n)^2 (2k)^k \max_{1\leq i_1, i_2 \leq m+1}\left\{ \frac{\sigma^2 B^{k-2}}{k} , \left(\frac{\sigma^2 n}{k}\right)^{k/2} \frac{(\left\|\vx^{(i_1)}\right\|_{0} \left\|\vy^{(i_2)}\right\|_{0})^{1/2}}{n^2}\right\}.
    \end{aligned}
    \end{equation*}
    up to some constant factors, which is further bounded by
    \begin{equation*} \begin{aligned}
        (\log n)^2 (2k)^k \max \left\{ \frac{\sigma^2 B^{k-2}}{k} , \left(\frac{\sigma^2 n}{k}\right)^{k/2-2}\right\},
    \end{aligned}     \end{equation*}
    and completes the proof.
\end{proof}

\begin{proof}[Proof of Lemma~\ref{lem:xHy-bound}]
By Assumption~\ref{assump:1}, 
\begin{equation*}
    \max_{1\leq i, j \leq n}\left|x_i y_j H_{ij} \right| \leq \|\vx\|_{\infty} \|\vy\|_{\infty} B
~~~\text{ and }~~~
\sum_{1\leq i, j \leq n} \E \left[(x_i y_j )^2 H_{ij}^2\right] \leq \sigma^2.
\end{equation*}
Setting $L = \|\vx\|_{\infty} \|\vy\|_{\infty} B \quad$ and $\nu = \sigma^2$
and applying Bernstein's inequality \citep[see, e.g.,][ or refer to Lemma~\ref{lem:bernstein} proved below in Section~\ref{sec:aux}]{vershynin2018HDP}, we have for any $t > 0$,
\begin{equation*}
    \Pr\left\{\left|\vx^\top \mH \vy\right| \geq t\right\} \leq 2\exp\left\{-\min\left(\frac{t^2}{4\sigma^2}, \frac{3t}{4\|\vx\|_{\infty} \|\vy\|_{\infty} B}\right)\right\}.
\end{equation*}
Taking $t = c_2\max\left\{\sigma \sqrt{\log n}, \|\vx\|_{\infty} \|\vy\|_{\infty} B \log n\right\}$ for $c_2$ chosen sufficiently large ensures that Equation~\eqref{eq:xHy-bound} in Lemma~\ref{lem:xHy-bound} holds with probability at least $1 - O(n^{-40})$.  
\end{proof}

\section{PROOFS FOR THEOREM~\ref{thm:symm-concentration} AND LEMMA~\ref{lem:symm-mean}}
\begin{proof}[Proof of Theorem~\ref{thm:symm-concentration}]
    Following the proof of Lemma 6.5 in \cite{erdos2013spectral}, we decompose $\mW$ as 
    \begin{equation*}
        \mW = \mH_1 + \mH_2,
    \end{equation*}
    where $\mH_1$ is the lower triangular part of $\mW$ (including the diagonal), $\mH_2$ is the upper triangular part of $\mW$ (excluding the diagonal). 
    As a result, we have 
    \begin{equation*}
        \left|\vx^\top \mW^k \vy - \E \vx^\top \mW^k \vy\right| \leq \sum_{d \in \{1,2\}^{k}} \left|\vx^\top \left(\prod_{i=1}^k \mH_{d_i}\right) \vy - \E \vx^\top \left(\prod_{i=1}^k \mH_{d_i}\right) \vy\right|. 
    \end{equation*}
    By Lemma 6.5 in \cite{erdos2013spectral}, all the $2^k$ terms can be handled in exactly the same manner.
    Thus, applying a union bound over all $2^k$ terms, we need only consider 
    \begin{equation*}
        \left|\vx^\top \mH_{1}^k \vy - \E \vx^\top \mH_{1}^k \vy\right|.
    \end{equation*}
    This is handled by Theorem~\ref{thm:main} and Lemma~\ref{lem:xHy-bound}.
    Taking a union bound over all $2^k \leq n^{20}$ terms, we obtain that Equation~\eqref{eq:symm-high-prob-bound} holds for all $2\leq k\leq 20\log n$ with probability at least $1 - O(n^{-20}\log^6 n)$. 
\end{proof}

\begin{proof}[Proof of Lemma~\ref{lem:symm-mean}]
    It is straightforward to see that for all $i \in [n]$, the term $\E \left(\ve_i^\top \mW^k \ve_i\right)$ shares the same upper bound in terms of $B$ and $\sigma$ and it suffices to provide an upper bound for 
\begin{equation*}
    \frac{1}{n} \sum_{i=1}^n \E \left(\ve_i^\top \mW^k \ve_i\right) = \frac{1}{n} \E \tr(\mW^k).
\end{equation*}
Alternatively, one can easily adapt the proof for $\frac{1}{n} \E \tr(\mW^k)$ to $\E \left(\ve_i^\top \mW^k \ve_i\right)$ by fixing the first index to be $i$ and see that they share the same bound; we omit the details here.  
The term $\frac{1}{n} \E \tr(\mW^k)$ is a well-studied quantity in random matrix theory for establishing semicircle law.
See Chapter 2 in \cite{bai2010spectral} and Chapter 2 in \cite{tao2012topics}. 
By Theorem 2.3.16 and Remark 2.3.17 in \cite{tao2012topics}, we have
\begin{equation*}
    \frac{1}{n} \E \tr(\mW^k) \lesssim C_{k/2} (\sigma^2 n)^{k/2}
\end{equation*}
when $k$ is an even integer, and $\E \tr(\mW^k) = 0$ when $k$ is odd.

When $i \neq j$, we have 
\begin{equation} \label{eq:ei-ej-mean-zero}
    \E \left(\ve_i^\top \mW^k \ve_j\right) = 0. 
\end{equation}
To see this, we take $i = 1, j=2$ and expand the term $\ve_1^\top \mW^k \ve_2$, 
\begin{equation*} 
    \E \ve_1^\top \mW^k \ve_2 = \sum_{j_1 \in [n]} \sum_{j_2 \in [n]} \cdots \sum_{j_{k-1} \in [n]} \E W_{j_0 j_1} W_{j_1 j_2} \cdots W_{j_{k-1} j_k}, 
\end{equation*}
where $j_0 = 1$ and $j_{k} = 2$.  
Suppose that for some $0 \leq r \leq k-1$, $W_{j_r j_{r+1}}$ is the last term where $j_r = j_0 = 1$. It is straightforward to see that $1$ appears an odd number of times in the sequence $(1,j_1), (j_1, j_2), \cdots, (j_r, j_{r+1})$ since if $j_s = 1$ for some $1 \leq s \leq r$, then $1$ appears both in $(j_{s-1}, j_s)$ and $(j_s, j_{s+1})$, while $1$ also appears in $(j_0, j_1) = (1, j_1)$. 
For 
\begin{equation*}
    \E W_{j_0 j_1} W_{j_1 j_2} \cdots W_{j_{k-1} j_k}
\end{equation*}
to be nonzero, each $W_{j_r j_{r+1}}$ or $W_{j_{r+1} j_{r}}$ must appear an even number of times by the symmetric condition in Assumption~\ref{assump:1}. 
Thus, for $\E W_{j_0 j_1} W_{j_1 j_2} \cdots W_{j_{k-1} j_k}$ to be nonzero, every index must also appear an even number of times. 
Since $1$ appears an odd number of times, we have 
\begin{equation*}
    \E \ve_1^\top \mW^k \ve_2 = 0
\end{equation*}
and Equation~\eqref{eq:ei-ej-mean-zero} follows similarly. 

Thus, for any fixed vector $\vx, \vy \in \R^n$, we have 
\begin{equation}\label{eq:middle-step}
\begin{aligned}
    \E \left[\vx^\top \mW^k \vy\right] &= \sum_{i=1}^n x_i y_i \E \left[\ve_i^\top \mW^k \ve_i\right] + \sum_{1\leq i\neq j\leq n} x_i y_j \E \left[\ve_i^\top \mW^k \ve_j\right]\\
    &= \sum_{i=1}^n x_i y_i \E \left[\ve_i^\top \mW^k \ve_i\right],
\end{aligned}
\end{equation}
and Equation~\eqref{eq:Poff-k} follows immediately. 
When $\mW$ has i.i.d.~on-diagonal and i.i.d.~off-diagonal entries, $\E \left[\ve_i^\top \mW^k \ve_i\right]$ are all the same for all $i \in [n]$ and we have 
\begin{equation*}
    \E \left[\ve_i^\top \mW^k \ve_i\right] = \frac{1}{n} \E \tr(\mW^k).
\end{equation*}
As a result, Equation~\eqref{eq:iid-k-moment} follows from the above display immediately.
Applying the Cauchy-Schwarz inequality to Equation~\eqref{eq:middle-step}, we have
\begin{equation*}
\begin{aligned}
    \left|\sum_{i=1}^n x_i y_i \E \left[\ve_i^\top \mW^k \ve_i\right]\right| &\leq \sum_{j=1}^n \frac{1}{2}\left(x_j^2 + y_j^2\right) \cdot \max_{i\in[n]} \left|\E \left[\ve_i^\top \mW^k \ve_i\right]\right|\\
    &\lesssim C_{k / 2}\left(\sigma^2 n\right)^{k / 2},
\end{aligned}
\end{equation*}
which yields Equation~\eqref{eq:on-diag}, completing the proof. 
\end{proof}

\section{PROOFS FOR SECTION~\ref{sec:applications}}
In this section, we provide proofs for the theoretical results derived in Section~\ref{sec:applications}.  

\subsection{Proof of Theorem~\ref{thm:linear-form}}

\begin{proof}
    The proof follows largely from the proof of Theorem 3 in \cite{chen2021asymmetry}.
    We outline a few key steps here.
    In view of Lemma~\ref{lem:W-op}, we have 
    \begin{equation*}
        \|\mH\| \lesssim \max \left\{B \log n, \sqrt{n \sigma^2 \log n}\right\}
    \end{equation*}
    with probability at least $1 - O(n^{-20})$, which together with condition~\eqref{eq:lambdastar:assump} in Theorem~\ref{thm:linear-form}, 
    and Lemma~\ref{lem:lambda-asymm} in Section~\ref{sec:aux} yields that $\|\mH\| \leq \lambda /3$ and 
    \begin{equation} \label{eq:lambda-bound}
        \frac{ 1 }{ 2 } \leq \left| \frac{ \lambda }{ \lambdastar } \right| \leq 2.
    \end{equation} 
    Therefore,
    \begin{equation}\label{eq:high-order-geometric}
    \begin{aligned}
        \sum_{k: k \geq 20\log n} \left(\frac{\|\mH\|}{\lambda}\right)^k &\lesssim \frac{\|\mH\|}{|\lambdastar|} \sum_{k: k\geq 20\log n - 1} \left(\frac{\|\mH\|}{\lambda}\right)^k \leq \frac{\|\mH\|}{|\lambdastar|} \sum_{k: k\geq 20\log n - 1} \left(\frac{1}{3}\right)^k\\
        &\lesssim \left|\lambdastar\right|^{-1} \max \left\{B \log n, \sigma \sqrt{n \log n}\right\} \cdot n^{-10}. 
    \end{aligned}
    \end{equation}
    By Theorem~\ref{thm:neumann}, we have
    \begin{equation} \label{eq:intermediate}
        \left|\va^\top \left(\vu - \frac{\vustart \vu}{\lambda/\lambdastar} \vustar\right)\right| = \left|\frac{\vustart \vu}{\lambda/\lambdastar} \sum_{k=1}^{\infty} \frac{\va^\top \mH^k \vustar}{\lambda^k}\right|
        \lesssim \sum_{k=1}^{20 \log n} \frac{1}{\lambda^k}\left|\va^{\top} \mH^k \vustar\right|+\sum_{k=20 \log n}^{\infty}\left(\frac{\|\mH\|}{\lambda}\right)^k.
    \end{equation}
    Combining Theorem~\ref{thm:main} with Lemma~\ref{lem:single-moment-bound}
    By condition~\eqref{eq:B-sigma-bound}, there exists an integer $k_0 \geq 0$ such that 
    \begin{equation*}
        (\sigma^2 n \log^3 n)^{k/2} \geq (B \log^3 n)^k \mu^{1/2} \|\va\|_{\infty},
    \end{equation*}
    holds for all $k \geq k_0$, and it suffices to take 
    \begin{equation*}
        k_0 = \left\lceil\frac{\log \mu + 2\log \|\va\|_{\infty}}{2\log \left(\sigma/B\right) +  \log n - 3\log \log n}\right\rceil. 
    \end{equation*}
    Since $\mu \leq n$ and $\log (\sigma / B) \gg -\log n$, we have 
    \begin{equation} \label{eq:k0:O1}
        k_0 = O(1). 
    \end{equation}
    Therefore, applying the bound in Equation~\eqref{eq:xHky-bound}, the first right-hand term in Equation~\eqref{eq:intermediate} is with probability at least $1 - \Otilde(n^{-40})$ bounded by 
    \begin{equation*}
    \begin{aligned}
        \sum_{k=1}^{20\log n} & \max \left\{\left(\frac{c_2^2\sigma^2 n \log^3 n}{\lambda^2}\right)^{k / 2} \sqrt{\frac{1}{n}}, \quad \left(\frac{c_2 B \log^3 n}{\lambda}\right)^k \sqrt{\frac{\mu}{n}} \|\va\|_{\infty}\right\}(\log n)^2\\
        & ~ = \left(\sum_{k\leq k_0} \left(\frac{c_2 B \log^3 n}{\lambda}\right)^k \sqrt{\frac{\mu}{n}} \|\va\|_{\infty} + \sum_{k = k_0+1}^{20\log n} \left(\frac{c_2^2\sigma^2 n \log^3 n}{\lambda^2}\right)^{k / 2} \sqrt{\frac{1}{n}}\right) (\log n)^2\\
        & ~ \lesssim \frac{c_2 B \log^5 n}{|\lambdastar|}\sqrt{\frac{\mu}{n}} \|\va\|_{\infty} + \left(\frac{c_2\sigma \sqrt{n \log^3 n}}{|\lambdastar|}\right)^{k_0} \frac{\log^2 n}{\sqrt{n}},
    \end{aligned}
    \end{equation*}
    where the first equality follows from the definition of $k_0$ and the last inequality follows from Equation~\eqref{eq:lambda-bound}. 
    Using the above bound and Equation~\eqref{eq:high-order-geometric}, we continue from Equation~\eqref{eq:intermediate} and obtain
    \begin{equation*}
        \left|\va^\top \left(\vu - \frac{\vustart \vu}{\lambda/\lambdastar} \vustar\right)\right| \lesssim
        \max\Biggl\{\frac{B \log^5 n}{|\lambdastar|} \sqrt{\frac{\mu}{n}} \|\va\|_{\infty} , \left(\frac{\sigma \sqrt{n \log^3 n}}{|\lambdastar|}\right)^{k_0} \frac{\log^2 n}{\sqrt{n}} \Biggr\} + \frac{\max \left\{B \log n, \sigma \sqrt{n \log n}\right\}}{|\lambdastar| n^{10}},
    \end{equation*}
    where the inequality follows from the na\"ive bound $a+b \leq 2\max\{a, b\}$ and the fact that $k_0 = O(1)$ in Equation~\eqref{eq:k0:O1}.
    Noting that the last term in the above display is of smaller order compared to the first term concludes the proof. 
\end{proof}

\subsection{Proof of Corollary \ref{cor:eigenvalue:one}}
\begin{proof}
    Taking $\va = \vustar$ in Theorem~\ref{thm:linear-form}, we have 
    \begin{equation} \label{eq:a-ustar}
        \left|\vustart\vu\right| \left|1 - \frac{1}{\lambda/\lambdastar}\right| \lesssim \frac{c_2 B\mu \log^5 n}{n\left|\lambdastar\right|} + \left(\frac{c_2 \sigma \sqrt{n} \log^{3/2} n}{\left|\lambdastar\right|}\right)^{k_0} \frac{\log^2 n}{\sqrt{n}}.
    \end{equation}
    From Lemma~\ref{lem:W-op}, and condition~\eqref{eq:lambdastar:assump} in Theorem~\ref{thm:linear-form}, we have
    \begin{equation*}
        \frac{ \|\mH\| }{ |\lambdastar| } \leq \frac{ 1 }{ 4 }
    \end{equation*}
    holds with probability at least $1 - O(n^{-20})$. 
    Taking the above bound into Equation~\eqref{eq:inner-prod} in Lemma~\ref{lem:inner-product} stated in Section~\ref{sec:aux}, we obtain $|\vustart\vu| \geq 1/2$. 
    Combined with Equation~\eqref{eq:lambda-bound}, we have
    \begin{equation*}
        |\lambda - \lambdastar| = \frac{|\lambdastar|}{\left|\vustart\vu\right|} \cdot \frac{|\lambda|}{|\lambdastar|} \left( \left|\vustart\vu\right| \left|1 - \frac{1}{\lambda/\lambdastar}\right| \right)
        \lesssim \frac{c_2 B\mu \log^5 n}{n} + \left(\frac{c_2 \sigma \sqrt{n} \log^{3/2} n}{\left|\lambdastar\right|}\right)^{k_0} \frac{|\lambdastar|\log^2 n}{\sqrt{n}},
    \end{equation*}
    where the last inequality follows from Equation~\eqref{eq:a-ustar}. 
    Applying a na\"ive bound $a+b \leq 2\max\{a, b\}$ and noting that $k_0 = O(1)$ by Equation~\eqref{eq:k0:O1} completes the proof.
\end{proof}

\subsection{Proof of Theorem~\ref{thm:linear-form-rank-r}}

\begin{proof}
    The proof follows largely from the proof of Theorem 4 in \cite{chen2021asymmetry}.
    By Lemma~\ref{lem:lambda-asymm} (proved below in Section~\ref{sec:aux}), one has
    \begin{equation*}
        |\lambda_l| \geq |\lambdastar_{\min}| - \|\mH\| = \kappa^{-1} |\lambdastar_{\max}| - \|\mH\|.
    \end{equation*}
    Combined with Lemma~\ref{lem:W-op}, and condition~\eqref{eq:lambdastar-max} in Theorem~\ref{thm:linear-form-rank-r}, it holds for all $l \in [r]$ with probability at least $1 - O(r n^{-20})$ that
    \begin{equation*}
        \frac{ \|\mH\| }{ |\lambda_l| } \leq \frac{ 1 }{ 3 },
    \end{equation*}
    and
    \begin{equation} \label{eq:lambda_l:ratio}
       \frac{ 1 }{ 2\kappa } \leq \frac{ |\lambda_l| }{ |\lambdastar_{\max}| } \leq 2.
    \end{equation}
    By the Neumman trick (Theorem~\ref{thm:neumann}) and Equation (50) in \cite{chen2021asymmetry}, we have
    \begin{equation}\label{eq:neumman-trick-step}
    \begin{aligned}
        \left|\va^\top \vu_l - \sum_{j=1}^r \frac{\lambdastar_j \vustart_j \vu_l}{\lambda_l} \va^\top \vustar\right| &\leq \left(\sum_{j=1}^r \frac{|\lambdastar_j|}{|\lambda_l|} \left|\vustart_j \vu_l\right|\right) \left\{\max_{j \in [r]} \sum_{k=1}^\infty \frac{\left|\va^\top \mH^k \vustar_j\right|}{|\lambda_l|^k}\right\}\\
        &\leq  \frac{\sqrt{r}|\lambdastar_1|}{|\lambda_l|} \cdot \left\{ \max_{j\in[r]} \sum_{k=1}^\infty \frac{\left|\va^\top \mH^k \vustar_j\right|}{|\lambda_l|^k} \right\}.
    \end{aligned}
    \end{equation}
    Following the same proof as Theorem~\ref{thm:linear-form}, 
    we have for all $j \in [r]$, 
    \begin{equation*}
        \sum_{k=20\log n}^\infty \left(\frac{\|\mH\|}{|\lambda_l|}\right)^k \lesssim \frac{ \kappa }{ |\lambdastar_{\max}| } \max \{B \log n, \sigma \sqrt{n \log n}\} \cdot n^{-10}. 
    \end{equation*} 
    and
    \begin{equation*}
    \begin{aligned}
        \sum_{k=1}^{20\log n} \frac{1}{|\lambda_l|^k} \left|\va^\top \mH^k \vustar_j\right| &\lesssim \frac{B \log ^5 n}{|\lambda_l|} \sqrt{\frac{\mu}{n}} \|\va\|_{\infty} + \left(\frac{\sigma \sqrt{n \log ^3 n}}{\left|\lambda_l\right|}\right)^{k_0} \frac{\log^2 n}{\sqrt{n}}\\
        &\leq \frac{\kappa B \log^5 n}{\lambdastar_{\max}} \sqrt{\frac{\mu}{n}} \|\va\|_{\infty} + \left(\frac{\kappa \sigma \sqrt{n \log ^3 n}}{\lambdastar_{\max}}\right)^{k_0}\frac{\log^2 n}{\sqrt{n}}
    \end{aligned}
    \end{equation*}
    where the last inequality follows from Equation~\eqref{eq:lambda_l:ratio}.
    Taking the above two displays into Equation~\eqref{eq:neumman-trick-step} and noting that $\lambdastar_{\max} = \kappa \lambdastar_{\min}$ yields the bound in Equation~\eqref{eq:master:rank-r}, 
    which completes the proof. 
\end{proof}

\subsection{Proof of Theorem~\ref{thm:eigenvalue:rank-r}}

\begin{proof}
Our proof follows largely from Section B.1 in \cite{cheng2021tackling}, with some minor differences. 
In particular, whenever possible, we use our improved bound in Theorem~\ref{thm:main}.
We only outline the proof here and refer the reader to Section B.1 in \cite{cheng2021tackling} for further details.

The proof leverages results from complex analysis and random matrix theory.
Define 
\begin{equation*}
    f(z) := \operatorname{det}\left(\mI + \mUstart (\mH - z\mI)^{-1} \mUstar \mLambdastar\right),
\end{equation*}
\begin{equation*}
    g(z) := \operatorname{det}\left(\mI + \mUstart (-z)^{-1} \mUstar \mLambdastar\right).
\end{equation*}
These two functions are commonly used to locate the eigenvalues of $\mM$ \citep{tao2013outliers, o2023matrices, wang2024analysis}. 
Indeed, one can see such a relation from Claim~\ref{claim:1}:
\begin{claim}\label{claim:1}
    If $\lambdastar_{\max} > 2\kappa \|\mH\|$, then the zeros of $f$ (resp.~$g$) on the region $\calK := \left\{z \in \bbC: |z| > \|\mH\|\right\} \cup \{\infty\}$ are exactly the $r$ leading eigenvalues of $\mM$ (resp.~$\mMstar$). 
\end{claim}
As a remark, Claim~\ref{claim:1} is established using the identity $\operatorname{det}(\mI+\mA \mB)=\operatorname{det}(\mI+\mB \mA)$. 
The details are in Section B.1 in \cite{cheng2021tackling}.

Let $\calD(\gamma)$ be a $\gamma$-neighborhood of $\{\lambdastar_1, \lambdastar_2, \cdots, \lambdastar_r\}$ defined as
\begin{equation*}
    \calD(\gamma) := \bigcup_{k=1}^r \calB(\lambdastar_k, \gamma),
\end{equation*}
where $\calB(\lambda, \gamma):= \{z\in \bbC: |z-\lambda| \leq \gamma\}$. 
A key step to locate the eigenvalues of $\mM$ is to show that in each connected component (in the topological sense) of $\calD(\gamma)$, functions $f$ and $g$ have the same number of zeros. 
This is often established through Rouch\'{e}s theorem. 

\begin{theorem}[Rouch\'{e}'s theorem]\label{thm:Rouche}
    Let $f$ and $g$ be two complex-valued functions that are holomorphic inside a region $\calR$ with closed contour $\partial \calR$. 
    If $|f(z) - g(z)| \leq |g(z)|$ for all $z \in \partial \calR$, then $f$ and $g$ have the same number of zeros inside $\calR$. 
\end{theorem}

By Rouch\'{e}'s theorem, it suffices to show that $|f(z) - g(z)| < |g(z)|$ on $\partial \calD(\gamma)$. 
In view of Lemma~\ref{lem:W-op}, and condition~\eqref{eq:lambdastar-max} in Theorem~\ref{thm:linear-form-rank-r}, we have 
\begin{equation} \label{eq:H-op:kappa}
    \|\mH\| \leq \frac{ \lambdastar_{\max} }{4 \kappa }
\end{equation}
with probability at least $1 - O(n^{-20})$. 
For the moment, we assume that the following bounds hold for $\gamma$:
\begin{equation} \label{eq:gamma-condition}
    \gamma <  \frac{ \lambdastar_{\max} }{4 \kappa } \quad \text{and} \quad \gamma < \frac{ \Deltastar_l }{ 2 }.
\end{equation}
Under Equations~\eqref{eq:H-op:kappa} and~\eqref{eq:gamma-condition}, it follows that for all $z \in \calD(\gamma)$,
\begin{equation*}
    |z| \geq \frac{ \lambdastar_{\max} }{ \kappa } - \gamma > \|\mH\|.
\end{equation*}
As a result, we are able to apply the Neumann series $(\mH-z\mI)^{-1} = -\sum_{k=0}^\infty z^{-k-1} \mH^k$ and obtain
\begin{equation*}
    f(z) = g(z)\operatorname{det} \Bigg(\mI - {(\mI-\mLambdastar/z)^{-1} \sum_{k=1}^\infty z^{-k-1} \mUstart \mH^k \mUstar \mLambdastar}\Bigg).
\end{equation*}
Defining
\begin{equation*}
\mDelta = (\mI-\mLambdastar/z)^{-1} \sum_{k=1}^\infty z^{-k-1} \mUstart \mH^k \mUstar \mLambdastar
\end{equation*}
for ease of notation, one can verify that showing $|f(z) - g(z)| < |g(z)|$ on $\partial \calD(\gamma)$ is equivalent to showing that 
\begin{equation*}
    \left|\operatorname{det}\left(\mI - \mDelta\right) - 1\right| < 1,
\end{equation*}
which holds as long as $\|\mDelta\| < 1/(2r)$. 
As a remark, similar perturbation results on determinants can be found in \cite{ipsen2008perturbation}. 
By Equation (78) in \cite{cheng2021tackling}, one has
\begin{equation*}
    \|\mDelta\| \leq \frac{\lambda_{\max } r}{\gamma} \sum_{k=1}^{\infty} \frac{\max _{1 \leq i, j \leq r}\left|\vu_i^{\star \top} \mH^k \vu_j^{\star}\right|}{\left(\frac{3}{4} \lambda_{\min }\right)^k}.
\end{equation*}
Following the same proof as Theorem~\ref{thm:linear-form-rank-r} and applying a union bound over $i,j \in [r]$ to the above display, we have 
\begin{equation*}
\begin{aligned}
    \|\mDelta\| &\lesssim \frac{\lambdastar_{\max } r}{\gamma} \left(\frac{\kappa \mu B \log ^3 n}{n\lambdastar_{\max }} +\left(\frac{c_2 \kappa \sigma \sqrt{n \log ^3 n}}{\lambdastar_{\max }}\right)^{k_0} \sqrt{\frac{1}{n}}\right)\\
    &\leq \gamma^{-1}\left(\frac{\kappa \mu r B \log ^3 n}{n} +\left(\frac{c_2 \kappa \sigma \sqrt{n \log ^3 n}}{\lambdastar_{\max }}\right)^{k_0} \lambdastar_{\max } r\sqrt{\frac{1}{n}}\right)
\end{aligned}
\end{equation*}
holds with probability at least $1 - O(n^{-20}r^2)$ under condition~\eqref{eq:lambdastar-max}. 
Taking 
\begin{equation*}
    \gamma = c_3 \left(\frac{\kappa \mu r^2 B \log ^3 n}{n} +\left(\frac{c_2 \kappa \sigma \sqrt{n \log ^3 n}}{\lambdastar_{\max }}\right)^{k_0} \lambdastar_{\max } r^2 \sqrt{\frac{1}{n}}\right)
\end{equation*}
for some constant $c_3 > 0$ sufficiently large, one can verify that Equation~\eqref{eq:gamma-condition} holds under condition~\eqref{eq:gap-condition} stated in Theorem~\ref{thm:eigenvalue:rank-r}, and condition~\eqref{eq:lambdastar-max} given in Theorem~\ref{thm:linear-form-rank-r}, and we have $\|\mDelta\| \leq 1/(2r)$ as desired.

Thus, we conclude that $f$ and $g$ have the same number of zeros in $\calD(\gamma)$. 
Following the same arguments after Remark 10 in Section B.1.1 in \cite{cheng2021tackling}, we obtain that $\lambda_l$ is real-valued and $\lambda_l \in \calB(\lambdastar_l, \gamma)$, which finishes the proof. 
\end{proof}

\section{TECHNICAL LEMMAS FOR THEOREM~\ref{thm:IExHky-moment}} \label{sec:thm2:tech}
In this section, we provide technical lemmas for Theorem~\ref{thm:IExHky-moment} and their detailed proof.
To begin with, we provide more discussion on the graph $G$ introduced in the proof of Theorem~\ref{thm:IExHky-moment}.
We remind the reader that the graph $G$ has $pk$ vertices identified with elements of the set $\calI$ given in Equation~\eqref{eq:def:calI}. 
Given $\mJ\in [n]^{p(k+1)}$, the edges of $G$ are constructed according to Rule~\ref{rule:(i)}, and we define the mapping $\psi:[n]^{p(k+1)} \to \calG(\calI)$ given by
\begin{equation}\label{eq:psi-define}
    \psi(\mJ) = G.
\end{equation}

An example graph 
constructed from the Rule~\ref{rule:(i)} when $\mJ\in[n]^{p(k+1)}$ for $p=4$ and $k=2$ satisfying the constraints
\begin{equation*}
    \left(j^{(1)}_{1}, j^{(1)}_{2}\right) = \left(j^{(2)}_{2}, j^{(2)}_{3}\right), \left(j^{(2)}_{1}, j^{(2)}_{2}\right) = \left(j^{(4)}_{2}, j^{(4)}_{3}\right), \left(j^{(1)}_{2}, j^{(1)}_{3}\right) = \left(j^{(3)}_{1}, j^{(3)}_{2}\right)
\end{equation*}
and 
\begin{equation*}
    \left(j^{(3)}_{2}, j^{(3)}_{3}\right) = \left(j^{(4)}_{1}, j^{(4)}_{2}\right)
\end{equation*}
is shown in Figure~\ref{fig:example}. 
\begin{figure}[ht]
    \centering
    \begin{tikzpicture}[scale=1, every node/.style={draw, circle, minimum size=0.8cm}]
        \node (v11) at (0, 0) {$(1,1)$};
        \node (v12) at (2, 0) {$(2,1)$};
        \node (v13) at (4, 0) {$(3,1)$};
        \node (v14) at (6, 0) {$(4,1)$};
        
        \node (v21) at (0, -2) {$(1,2)$};
        \node (v22) at (2, -2) {$(2,2)$};
        \node (v23) at (4, -2) {$(3,2)$};
        \node (v24) at (6, -2) {$(4,2)$};
        
        \draw (v11) -- (v22);
        \draw (v12) -- (v24);
        \draw (v13) -- (v21);
        \draw (v14) -- (v23);
    \end{tikzpicture}
    \caption{A graph created following Rule~\ref{rule:(i)} when $k = 2$ and $p = 4$.}\label{fig:example}
\end{figure}
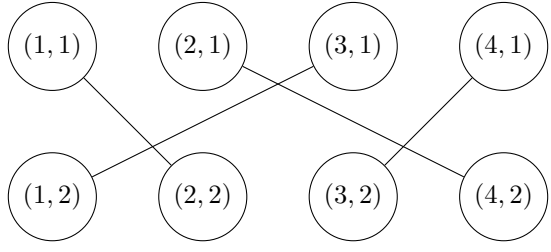
This example shows that different $\mJ$ can correspond to the same graph $G$, as long as $\mJ$ satisfies the constraints imposed by $G$. 
One concrete example of $\mJ$ satisfying the above constraints would be $\vjj{1} = (1,2,3), \vjj{2} = (4,1,2), \vjj{3} = (2,3,4), \vjj{4} = (3,4,1)$, so we have $\E \prod_{r=1}^4 (\vzeta_{\vjj{r}} - \E \vzeta_{\vjj{r}})$  given by
\begin{equation*}
\begin{aligned}
     \E \! \left[(H_{12}H_{23} \!-\! \E H_{12}H_{23})(H_{41}H_{12} \!-\! \E H_{41}H_{12})(H_{23}H_{34} \!-\! \E H_{23}H_{34})(H_{34}H_{41} \!-\! \E H_{34}H_{41})\right]
    = \E \!\left[H_{12}^2 H_{23}^2 H_{34}^2 H_{41}^2\right]. 
\end{aligned}
\end{equation*}
It is straightforward to see that the choice $\vjj{1} = (5,6,1), \vjj{2} = (4,5,6), \vjj{3} = (6,1,4), \vjj{4} = (1,4,5)$ would also produce the same $G$, which corresponds to $\E[H_{56}^2 H_{61}^2 H_{14}^2 H_{45}^2 ]$. 

We will demonstrate that the graph $G$ effectively groups similar $\mJ$ together, in the sense that for all $\mJ \in \psi^{-1}(G)$, the terms given by Equation~\eqref{eq:single-moment} all share the same upper in the noise parameter $\sigma$ and $B$. 
Moreover, this upper bound only depends on $c(G)$, the number of connected components in $G$ (Lemmas~\ref{lem:G-cc-size-2} and~\ref{lem:single-moment-bound}).
We then bound the number of graphs that have the same number of connected components in Lemma~\ref{lem:G-cc-count}.

Before we proceed, we introduce one extra notation. 
Note that every vertex $(r,l) \in \calI$ in the same connected component $G_\ell$ of $G$, 
$H_{j_{l}^{(r)} j_{l+1}^{(r)}}$ must all correspond to the same entry in $\mH$ by Rule~\ref{rule:(i)}. 
Denote this entry as 
\begin{equation}\label{eq:phi:define}
    H_{\phi(G_\ell)}, \quad \phi(G_{\ell}) \in [n] \times [n]. 
\end{equation}

\subsection{Details of Technical Lemmas~\ref{lem:G-cc-size-2}, ~\ref{lem:single-moment-bound} and~\ref{lem:G-cc-count}} \label{sec:moment-bounds}

We introduce the following condition:
\begin{condition}\label{cond:graph:I}
    The size of each connected component of the graph $G$ is an even number.  
\end{condition}
Lemma~\ref{lem:G-cc-size-2} shows that Condition~\ref{cond:graph:I} holds for all $G = \psi(\mJ)$ such that Equation~\eqref{eq:single-moment} is nonzero.

\begin{lemma} \label{lem:G-cc-size-2}
    Recall that $\vzeta_{\vjj{r}}$ is defined in Equation~\eqref{eq:vzeta-vj-define}. 
    For Equation~\eqref{eq:single-moment} to be nonzero, every connected component in $G = \psi(\mJ)$ must have a size of even number, where $\psi(\mJ)$ is defined in Equation~\eqref{eq:psi-define}. 
\end{lemma}
\begin{proof}
    Suppose that $G_\ell$ is a connected component in $G$ that has an odd size.
    Recall that $\phi(G_{\ell})$ is introduced in Equation~\eqref{eq:phi:define}. 
    There must be a $r_0 \in [p]$ such that $H_{\phi(G_{\ell})}$ appears an odd number of times in $\vzeta_{\vjj{r_0}}$, and therefore,
    \begin{equation*}
        \E \vzeta_{\vjj{r_0}} = 0.
    \end{equation*}
    Hence, we can rewrite Equation~\eqref{eq:single-moment} as 
    \begin{equation*}
        \E \Bigg[\prod_{r=1}^p \left(\vzeta_{\vjj{r}} - \E \vzeta_{\vjj{r}}\right)\Bigg]
        = \E \Bigg[\vzeta_{\vjj{r_0}}\prod_{r \neq r_0} \left(\vzeta_{\vjj{r}} - \E \vzeta_{\vjj{r}}\right)\Bigg]
        = \sum_{\vd \in \{0,1\}^{p-1} } \E \left[\vzeta_{\vjj{r_0}} \prod_{r\neq r_0} F_{d_r}\left(\vzeta_{\vjj{r}}\right) \right]
    \end{equation*}
    where in the last equality, we have expanded the product in the first equality and defined
    \begin{equation*}
        F_0\left(\vzeta_{\vjj{r}}\right) = \vzeta_{\vjj{r}} \quad \text{and} \quad F_1\left(\vzeta_{\vjj{r}}\right) = -\E\vzeta_{\vjj{r}}.
    \end{equation*}
    We observe that the term 
    \begin{equation*}
        \E \left[\prod_{r\neq r_0, d_r = 1} F_1\left(\vzeta_{\vjj{r}}\right)\right] = (-1)^{\sum_{r\neq r_0} d_r} \prod_{\substack{r\neq r_0 \\ d_r = 1}} \E\vzeta_{\vjj{r}}
    \end{equation*}
    is zero unless for all $r \neq r_0$ and $d_r = 1$, the term $\vzeta_{\vjj{r}}$ contains an even number of $H_{\phi(G_{\ell})}$.
    On the other hand, the term
    \begin{equation*}
        \E\left[\vzeta_{\vjj{r_0}} \prod_{r\neq r_0, d_r = 0} F_0\left(\vzeta_{\vjj{r}}\right)\right] = \E\left[\vzeta_{\vjj{r_0}} \prod_{r\neq r_0, d_r = 0} \vzeta_{\vjj{r}}\right]
    \end{equation*}
    is zero unless $\vzeta_{\vjj{r_0}} \prod_{r\neq r_0, d_r = 0} \vzeta_{\vjj{r}}$ contains an even number of $H_{\phi(G_{\ell})}$.
    Since $G_\ell$ is of odd size by assumption, there is an odd number of $H_{\phi(G_{\ell})}$ terms in total.
    Thus, by combining the above two displays, we must have 
    \begin{equation*}
\E \left[\vzeta_{\vjj{r_0}} \prod_{r\neq r_0} F_{d_r}\left(\vzeta_{\vjj{r}}\right) \right]
=\E \left[\prod_{r\neq r_0, d_r = 1} F_1\left(\vzeta_{\vjj{r}}\right)\right]\E\left[\vzeta_{\vjj{r_0}} \prod_{r\neq r_0, d_r = 0} F_0\left(\vzeta_{\vjj{r}}\right)\right] = 0.
    \end{equation*}
    Since this must hold for all $\vd \in \{0,1\}^{p-1}$, it follows that
\begin{equation*}
    \E \Bigg[\prod_{r=1}^p \left(\vzeta_{\vjj{r}} - \E \vzeta_{\vjj{r}}\right)\Bigg] = 0. 
\end{equation*}
Thus, 
for a graph $G$ that makes Equation~\eqref{eq:single-moment} nonzero, its connected components must all have an even number size, as we set out to show.
\end{proof}

Decomposing $G = \cup_{\ell=1}^L G_{\ell}$ into its nonempty connected components, we have $1\leq L \leq pk/2$ by Lemma~\ref{lem:G-cc-size-2}. 
Lemma~\ref{lem:single-moment-bound} shows that for any graph $G$ constructed from Rule~\ref{rule:(i)}, the terms in Equation~\eqref{eq:single-moment} with $\mJ \in \psi^{-1}(G)$ all share the same upper bound, and this bound only depends on $G$ through $c(G)$, the number of connected components in $G$.  

\begin{lemma} \label{lem:single-moment-bound}
    For any graph $G$ constructed according to Rule~\ref{rule:(i)} with $c(G) = L \in [1, pk/2]$ connected components, any term in Equation~\eqref{eq:single-moment} with $\mJ \in \psi^{-1}(G)$ satisfies the bound
    \begin{equation*}
        \left|\E \Bigg[\prod_{r=1}^p \left(\vzeta_{\vjj{r}} - \E \vzeta_{\vjj{r}}\right)\Bigg]\right| \leq 2^p \sigma^{2L} B^{pk - 2L}.
    \end{equation*}
\end{lemma}
\begin{proof}
    Applying triangle inequality to Equation~\eqref{eq:single-moment}, we have 
    \begin{equation} \label{eq:zeta-prod}
        \left|\E \Bigg[\prod_{r=1}^p \left(\vzeta_{\vjj{r}} - \E \vzeta_{\vjj{r}}\right)\Bigg]\right|
        \leq \E \Bigg[\prod_{r=1}^p \left(\prod_{l=1}^k \left|H_{\jr{l} \jr{l+1}}\right| + \E \prod_{l=1}^k \left|H_{\jr{l} \jr{l+1}}\right|\right)\Bigg]
    \end{equation}
    By Lyapunov's inequality, for any $a > b > 0$ and any random variable $X$, we have  
    \begin{equation*}
        \E |X|^{a} = \left(\E|X|^{a}\right)^{\frac{b}{a}} \left(\E|X|^{a}\right)^{\frac{a-b}{a}} \geq \E |X|^{b} \E |X|^{a-b}.
    \end{equation*}
    For any connected component $G_\ell$ of $G$, recall that all of its indices are assigned the same $\phi(G_\ell) \in [n]\times [n]$.
    Let $\gamma_\ell$ denote the size of $G_\ell$.
    Expanding Equation~\eqref{eq:zeta-prod} into $2^p$ terms, and using the decomposition $G = \cup_{\ell=1}^L G_{\ell}$ together with Lyapunov's inequality, it follows that Equation~\eqref{eq:zeta-prod} is further bounded as
    \begin{equation*}
        \left|\E \Bigg[\prod_{r=1}^p \left(\vzeta_{\vjj{r}} - \E \vzeta_{\vjj{r}}\right)\Bigg]\right|
        \le 2^p \E \Bigg[\prod_{r=1}^p \prod_{l=1}^k \left|H_{\jr{l} \jr{l+1}}\right|\Bigg]
        = 2^p \prod_{\ell=1}^L \E \left|H_{\phi(G_\ell)}\right|^{\gamma_\ell}.
    \end{equation*}
    Since $\gamma_\ell \geq 2$ by Lemma~\ref{lem:G-cc-size-2}, it follows that 
    \begin{equation} \label{eq:cc-bound}
    \left|\E \Bigg[\prod_{r=1}^p \left(\vzeta_{\vjj{r}} - \E \vzeta_{\vjj{r}}\right)\Bigg]\right|
    \le 2^p \prod_{\ell=1}^L \sigma^2 B^{\gamma_\ell - 2}.
    \end{equation}
    Using the fact that $\sum_{\ell=1}^L \gamma_\ell = pk$, we conclude that
    \begin{equation*}
    \left|\E \Bigg[\prod_{r=1}^p \left(\vzeta_{\vjj{r}} - \E \vzeta_{\vjj{r}}\right)\Bigg]\right|
    \le 2^p \sigma^{2L} B^{pk - 2L},
    \end{equation*}
    completing the proof.
\end{proof}

\begin{lemma}\label{lem:G-cc-count}
    The number of different graphs constructed from $pk$ vertices with $L$ connected components where each connected component has an even size  is bounded by
    \begin{equation*}
        2^{p k} L^{p k-L}. 
    \end{equation*}
\end{lemma}
\begin{proof}
    This is Lemma 7 in \cite{chen2021asymmetry}. 
    For completeness, we include the proof here. 
    We first pick $L$ vertices and assign them to $L$ different components, which can be achieved in $\binom{pk}{L}$ ways. 
    We still have $pk-L$ vertices left, and the number of ways to assign them to $L$ components is upper bounded by $L^{pk-L}$. 
    Thus, the number of different graphs is bounded by
    \begin{equation*}
        \binom{pk}{L} L^{pk-L} \leq 2^{pk} L^{pk-L},
    \end{equation*}
    as we set out to show.
\end{proof}

\subsection{Bounding the cardinality of $\psi^{-1}(G)$} \label{sec:bound}

In view of Equation~\eqref{eq:grand-decomposition}, it remains to bound $|\psi^{-1}(G)|$ to complete the upper bound of Equation~\eqref{eq:moment-i}.
The difficulty arises from the fact that we have three different types of indexes, $j_{1}^{(r)}$, $j_{k+1}^{(r)}$ and $j_{l}^{(r)}$ for $2\leq l \leq k$. 
While $j_{l}^{(r)}$ for $2\leq l \leq k$ can generally take any value from $[n]$, the values of $j_1^{(r)}$ and $j_{k+1}^{(r)}$ are restricted by the support of $\vx$ and $\vy$, respectively. 
The prior literature does not take into account that $N_x := \|\vx\|_0$ and $N_y := \|\vy\|_0$ may be much smaller than $n$ and we will provide a finer counting argument considering $N_x$ and $N_y$. 
In Section~\ref{sec:new-graph}, we set the stage and obtain preliminary results. 
The results controlling $|\psi^{-1}(G)|$ are then given in Section~\ref{sec:results}.

\subsubsection{Creating a new graph from $G$}\label{sec:new-graph}

To start with, we note that there is a critical constraint of $\mJ$ not captured by $G$. 
Recall the notation $\phi(G_\ell)$ introduced in Equation~\eqref{eq:phi:define}.
If a connected component $G_{\ell_1}$ contains $(r, l) \in \calI$ and another connected component $G_{\ell_2}$ contains $(r, l+1) \in \calI$, then we must have that the second value of $\phi(G_{\ell_1})$ and the first value of $\phi(G_{\ell_2})$ are the same. 
To enforce this constraint, we construct a new {\em directed} graph, $\Gnew$, from the graph $G$. 
Letting $G = \cup_{\ell \in [L]} G_\ell$ be the decomposition of $G$ into its connected components, we identify the vertices of $\Gnew$ with these $L$ connected components.
Since we identify $G_1,G_2,\dots,G_L$ with the vertices of $\Gnew$, we abuse notation slightly and simply use $G_{\ell}$ to denote the $\ell$-th node of $\Gnew$.
Whether we intend to refer to a node of $\Gnew$ or a connected component of $G$ will be clear from the context.
By construction, $\Gnew$ has $L$ vertices.
We connect pairs of vertices in $\Gnew$ according to Rule~\ref{rule:(ii)}:
\begin{enumerate}[label=(\roman*)]\addtocounter{enumi}{1}
    \item \label{rule:(ii)} Form a {\em directed} edge from node $G_{\ell_1}$ to $G_{\ell_2}$ in $\Gnew$ if there exists some $(r,l) \in \calI$ such that $(r,l)$ is in $G_{\ell_1}$ and $(r,l+1) \in \calI$ is in $G_{\ell_2}$. 
\end{enumerate} 
Figure~\ref{fig:example-new} depicts an example graph constructed from the graph in Figure~\ref{fig:example} according to Rule~\ref{rule:(ii)}.

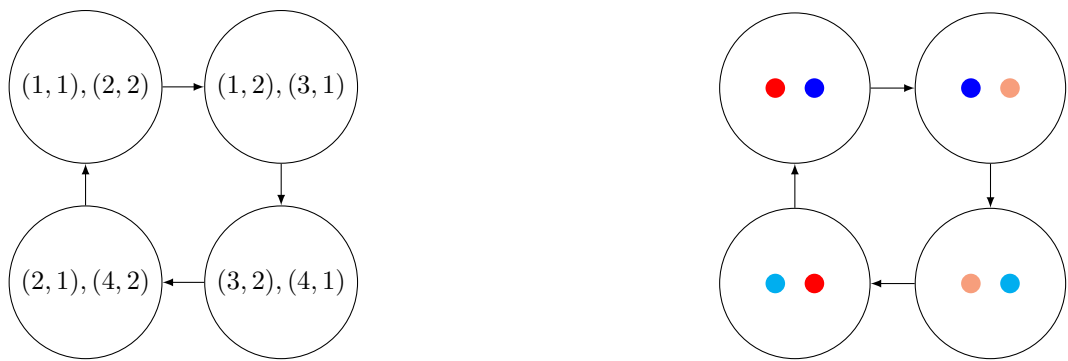
\begin{figure}[ht]
    \centering
    \begin{subfigure}{0.45\textwidth}
    \centering
    \begin{tikzpicture}[scale=1.3, every node/.style={draw, circle, minimum size=0.7cm}]
        \node (v1122) at (0, 0) {$(1,1), (2,2)$};
        \node (v1231) at (2, 0) {$(1,2), (3,1)$};
        \node (v2142) at (0, -2) {$(2,1), (4,2)$};
        \node (v3241) at (2, -2) {$(3,2), (4,1)$};
        
        \draw[-latex] (v1122) -- (v1231);
        \draw[-latex] (v1231) -- (v3241);
        \draw[-latex] (v2142) -- (v1122);
        \draw[-latex] (v3241) -- (v2142);
    \end{tikzpicture}
    \end{subfigure}
    \hfill
    \begin{subfigure}{0.45\textwidth}
    \centering
    \begin{tikzpicture}[scale=1.3, every node/.style={draw, circle, minimum size=2cm}]
        \node (v1122) at (0, 0) {};
        \fill[red] (-0.2,0) circle (0.1cm);
        \fill[blue] (0.2,0) circle (0.1cm);
        \node (v1231) at (2, 0) {};
        \fill[blue] (2-0.2,0) circle (0.1cm);
        \fill[Melon] (2+0.2,0) circle (0.1cm);
        \node (v2142) at (0, -2) {};
        \fill[cyan] (0-0.2,-2) circle (0.1cm);
        \fill[red] (0+0.2,-2) circle (0.1cm);
        \node (v3241) at (2, -2) {};
        \fill[Melon] (2-0.2,-2) circle (0.1cm);
        \fill[cyan] (2+0.2,-2) circle (0.1cm);
        
        \draw[-latex] (v1122) -- (v1231);
        \draw[-latex] (v1231) -- (v3241);
        \draw[-latex] (v2142) -- (v1122);
        \draw[-latex] (v3241) -- (v2142);
    \end{tikzpicture}
    \end{subfigure}
    \caption{A graph $\Gnew$ created following Rule~\ref{rule:(ii)} when $p = 2$ from the graph $G$ in Figure~\ref{fig:example}. 
    The left plot shows the labels included in nodes of $\Gnew$ (which corresponds to connected components of $G$, see Figure~\ref{fig:example}). The right plot displays a possible coloring of $\Gnew$ (see Section~\ref{sec:coloring-scheme} below).}\label{fig:example-new}
\end{figure} 

\subsubsection{Coloring scheme} \label{sec:coloring-scheme}

To more explicitly encode the constraints that certain indices have to take on the same value, it is helpful to introduce a coloring scheme.
\begin{definition}[Coloring scheme] \label{def:coloring}
    For $\mJ \in [n]^{p(k+1)}$, if two indices $j^{(r_1)}_{\ell_1}$ and $j^{(r_2)}_{\ell_2}$ are assigned the same value from $[n]$, we associate these two indexes with the same color.
    Otherwise, they are associated with different colors. 
\end{definition}

This way, any connected component $G_{\ell}$ of $G$ and any node in $\Gnew$ receives a tuple of two colors.
For example, if $(r,l) \in V(G_{\ell})$, then the first color of $G_{\ell}$ is the color associated with $j_{l}^{(r)}$, and the second color is associated with $j_{l+1}^{(r)}$.
An example of a colored graph is shown in the right subplot of Figure~\ref{fig:example-new}. 

Following the coloring scheme, different connected components have different pairs of colors (the order of the pair of colors also matters) since they must correspond to different entries of $\mH$. 
In addition, we have the following observation:
\begin{observation} \label{obs:color}
    For any two nodes $G_1, G_2 \in V(\Gnew)$, if there is an edge from $G_1$ to $G_2$, then the first color of $G_2$ must match the second color of $G_1$. 
\end{observation}

We introduce a few definitions that will be used repeatedly throughout the proof to facilitate the discussion. 
\begin{definition}\label{def:start-end} 
    A node $v \in V(\Gnew)$ is called a \textbf{start node} if, for some $r \in [p]$, $(r, 1) \in v$, and an \textbf{end node} if $(r, k) \in v$, for some $r \in [p]$. 
\end{definition}

To distinguish the contribution of the three different types of indexes $j_{1}^{(r)}$, $j_{k+1}^{(r)}$ and $j_{l}^{(r)}$ for $2\leq l \leq k$ mentioned at the start of Section~\ref{sec:bound}, we introduce the following definitions: 
\begin{definition}\label{def:colors}
    We refer to the first color of a start node as a \textbf{start color}, the second color of an end node as an \textbf{end color}. 
If a color is neither a start nor end color, we call it an \textbf{internal color}.
If a color is either a start or end color, we call it a \textbf{terminal color}. 
\end{definition}
Definition~\ref{def:colors} helps us to count $\psi^{-1}(G)$.
Recall from our coloring scheme in Definition~\ref{def:coloring} that corresponding indices must take the same value from $[n]$ if their colors are the same. 
More precisely, we have the following observation:
\begin{observation} \label{obs:color-value}
    If an index $j_l^{(r)}$ for some $(r,l) \in \calI$ receives a start color, then for some $s \in [p]$, $j_l^{(r)} = j_1^{(s)}$, and thus, its value has $N_x$ possible choices; if $j_l^{(r)}$ receives an end color, then $j_l^{(r)} = j_{k+1}^{(s)}$ for some $s \in [p]$, and its value has $N_y$ choices; if $j_l^{(r)}$ receives an internal color, it has $n$ choices of values. 
\end{observation}
By counting the number of distinct colors of each type that can be assigned to $\Gnew$, we can derive a bound for $|\psi^{-1}(G)|$.
We first derive a result that controls the total number of distinct colors that can be assigned to a weakly connected component of $\Gnew$ in Lemma~\ref{lem:color-bound}, disregarding the types of colors. 

\begin{lemma} \label{lem:color-bound}
    For a weakly connected component $\Gnew_l$ in $\Gnew$ having $\beta_l$ nodes, we can assign at most $\beta_l+1$ distinct colors to $\Gnew_l$. 
\end{lemma}
\begin{proof}
    Recall that a weakly connected component is a connected component when we view the directed graph as an undirected graph. 
    To prove Lemma~\ref{lem:color-bound}, we take a spanning tree of $\Gnew_l$ and pick a root node arbitrarily. 
    We can consider the spanning tree as an undirected graph.
    The only thing that matters to us is that if there is an edge between two nodes, these two nodes must share one color following Observation~\ref{obs:color}.
    An example spanning tree is shown in Figure~\ref{fig:example-spt}. 
    We start at the root node, which can be assigned a pair of distinct colors. 
    Suppose the number of child nodes at depth $h$ is $\beta_{l,h}$, and the depth of the spanning tree is $h_0$. 
    By Observation~\ref{obs:color}, it is straightforward to see that after fixing the colors of the root node, the number of distinct colors at depth $1$ is $\beta_{l,1}$.
    Similarly, fixing the child nodes at depth $h$ and moving to depth $h+1$, we see that the number of distinct colors that can be assigned at depth $h+1$ is $\beta_{l,h+1}$. 
    Since the total number of nodes in the spanning tree of $\Gnew_l$ is given by
    \begin{equation*}
       1 + \sum_{h=1}^{h_0} \beta_{l, h} = \beta_l,
    \end{equation*}
    the maximum number of distinct colors that can be assigned to the spanning tree of $\Gnew_l$ is given by
    \begin{equation*}
        2 + \sum_{h=1}^{h_0} \beta_{l, h} = \beta_l+1,
    \end{equation*}
    as we set out to show.
    \end{proof}
    
    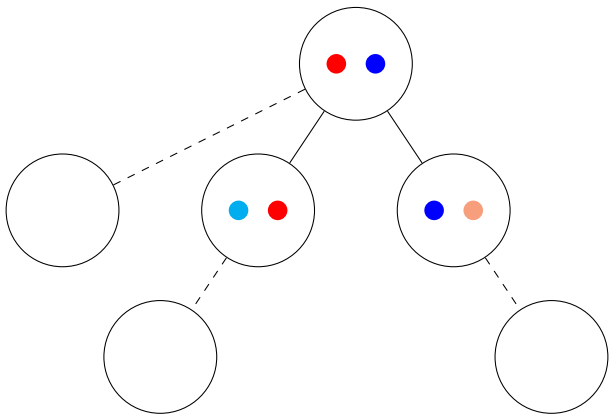
\begin{figure}[ht]
    \centering
    \begin{tikzpicture}[scale=1.3, every node/.style={draw, circle, minimum size=1.5cm}]
        \node (r) at (0, 1.5) {};
        \fill[red] (-0.2,1.5) circle (0.1cm);
        \fill[blue] (0.2,1.5) circle (0.1cm);
        \node (v1) at (1, 0) {};
        \fill[blue] (1-0.2,0) circle (0.1cm);
        \fill[Melon] (1+0.2,0) circle (0.1cm);
        \node (v2) at (-1, 0) {};
        \fill[cyan] (-1-0.2,0) circle (0.1cm);
        \fill[red] (-1+0.2,0) circle (0.1cm);
        \node (v3) at (2, -1.5) {};
        \node (v4) at (-2, -1.5) {};
        \node (v5) at (-3, 0) {};
        \draw (r) -- (v1);
        \draw (r) -- (v2);
        \draw[dashed] (v1) -- (v3);
        \draw[dashed] (v2) -- (v4);
        \draw[dashed] (r) -- (v5);
    \end{tikzpicture}
    \caption{An example spanning tree. To color it, we start from the root node and then assign colors level by level. Dashed edges and empty nodes represent the remaining subtrees of the spanning tree. }\label{fig:example-spt}
\end{figure}

\subsubsection{Details for graph-counting lemmas: Lemmas~\ref{lem:tree-terminals},~\ref{lem:forest-color} and~\ref{lem:loops}} \label{sec:preliminary}

Our proof relies on graphical tools, and we pause to review some general graph terminologies that will be used frequently throughout the rest of the proof.

In a directed graph, a \textbf{source} is a node with no incoming edge, a \textbf{sink} is a node with no outgoing edge, and a \textbf{terminal} is either a source or a sink. 
A \textbf{branch node} has at least two outgoing edges or two incoming edges.
A \textbf{leaf} is a node with degree one (either in-degree or out-degree).
For any two nodes $u$ and $v$ in a directed graph $\Gamma$, we denote the \emph{shortest path} from $u$ to $v$ in a directed graph as $P(u, v; \Gamma)$.
When there are multiple shortest paths, we pick one path arbitrarily.
The \textbf{directed distance} $d(u,v; \Gamma)$ is the number of edges in the shortest path from $u$ to $v$. 
The \textbf{undirected distance} $d_0(u,v; \Gamma)$ is the number of edges in the shortest path from $u$ to $v$ when we view the directed graph as an undirected graph. 
When it is clear from context, we suppress the dependence on $\Gamma$ and write $d(u,v; \Gamma)$ and $d_0(u,v; \Gamma)$ as $d(u,v)$ and $d_0(u,v)$, respectively.
Denote the closest source node to $u$ in a graph $\Gamma$ as 
\begin{equation} \label{eq:source}
    s_{\Gamma}(u)
\end{equation}
and the closest sink node to $u$ as 
\begin{equation} \label{eq:sink}
    e_{\Gamma}(u). 
\end{equation}
If there are multiple closest nodes, we choose one arbitrarily.

Before proceeding to the details of Lemmas~\ref{lem:tree-terminals}-\ref{lem:loops}, we provide a roadmap.
Lemma~\ref{lem:color-bound} establishes a general bound on the total number of distinct colors, without considering the specific color types defined in Definition~\ref{def:colors}. 
The goal of Lemmas~\ref{lem:tree-terminals} through \ref{lem:loops} is to further refine this bound by determining the number of distinct colors of each type that can be assigned.
In view of the proof of Lemma~\ref{lem:color-bound}, we see that trees are easier to handle, as they have a well-defined hierarchical structure that allows for assigning colors level-by-level. 
This structure also makes it relatively straightforward to distinguish and separate each node type defined in Definition~\ref{def:start-end} within the tree.
Therefore, we first prove results for trees in Lemma~\ref{lem:tree-terminals}. 
The next closest graphical object to a tree is a \emph{forest}, and Lemma~\ref{lem:forest-color} extends the results of Lemma~\ref{lem:tree-terminals} from trees to forest. 
Finally, in Lemma~\ref{lem:loops}, we prove results for arbitrary graphs by reducing them to forests. 

In Lemmas~\ref{lem:tree-terminals} through \ref{lem:loops}, we will often refer to graphs that satisfy the following conditions:
\begin{condition}\label{cond:start-end}
    Every source is a start node, and every sink is an end node. 
    The reverse also holds: every start node is a source, and every end node is a sink.
\end{condition}
\begin{condition}\label{cond:k-forest}
    For every node $u$ in the graph, there exists a source node $u_s$ and a sink node $u_e$ in the graph such that $d(u_s, u) + d(u, u_e) \leq k-1$. 
\end{condition}
As a remark, the integer $k$ in Condition~\ref{cond:k-forest} is the same $k$ in the term $\vx^\top \mH^k \vy$. 
Condition~\ref{cond:k-forest} captures the combinatorial structure that every index $(j^{(r)}_l, j^{(r)}_{l+1})$ for $(r, l) \in \calI$ is in a index sequence
\begin{equation*}
    (j_1^{(r)}, j_2^{(r)}), (j_2^{(r)}, j_3^{(r)}), \cdots, (j_k^{(r)}, j_{k+1}^{(r)})
\end{equation*}
of length $k$, so for every node $u \in V(\Gnew)$, there will be a path with at most $k-1$ edges, which contains a start node, an end node, and the node $u$.
If $\Gnew$ satisfies Condition~\ref{cond:start-end}, then the start node is a source, and the end node is a sink, so $\Gnew$ satisfies Condition~\ref{cond:k-forest}. 
Although $\Gnew$ does not always satisfy Condition~\ref{cond:start-end}, it is possible to find a graph that satisfies both Conditions~\ref{cond:start-end} and~Condition~\ref{cond:k-forest}. 
The distinct number of colors of each type in this graph provides an upper bound for the corresponding number of colors in $\Gnew$.
This forms the main idea behind Lemmas~\ref{lem:tree-terminals} through \ref{lem:loops}.



\begin{lemma}\label{lem:tree-terminals}
For integers $k \ge 2$ and $t \ge 1$, consider a directed tree $T$ with $\beta$ nodes where $\beta \geq \max\left\{(t-1)k+1, 2\right\}$ such that Condition~\ref{cond:k-forest} holds with the integer $k$.
$T$ must have at least $t+1$ terminal nodes. 
\end{lemma}
\begin{proof}
    We refer the reader to Section~\ref{sec:preliminary} for a summary of the definitions.
    To start, we observe that every leaf node in a tree $T$ is either a source or a sink, and thus, every leaf node is a terminal.
    For $t = 1$, the claim in Lemma~\ref{lem:tree-terminals} holds trivially since every tree with at least $2$ nodes has at least $2$ leaves, which are both terminals.
    We prove the claim for $t\geq 2$ by induction. 
    For the base case $t=2$, if $T$ only has $2$ terminals, it is a path that can not contain more than $k$ nodes by Condition~\ref{cond:k-forest}. 
    Therefore, our claim holds for $t=2$ by contradiction.
    
    Suppose the statement in Lemma~\ref{lem:tree-terminals} holds for all $t \in \{2,3,\dots, t'\}$. 
    We show that it also holds for $t = t'+1$.
    Any tree has at least one source node. 
    When there is only one source, from the argument above, the tree must have at least two sinks. 
    Reversing the direction of all the edges in the tree, we have at least two sources,
    so without loss of generality, we assume that there are at least two sources.
    
    We break the remaining proof into four steps. 
    In Step 1, we divide $T$ into two disjoint trees $T_1$ and $T_2$. 
    We verify in Step 2 that $T_1$ satisfies Condition~\ref{cond:k-forest}, and in Step 3 that $T_2$ also satisfies Condition~\ref{cond:k-forest}. 
    We then use the induction hypothesis on $T_1$ and $T_2$ in Step 4 to complete the proof.
    
    \paragraph{Step 1. Divide $T$ into $T_1$ and $T_2$.}
    Starting from any source node $v_1$, we follow one of its outgoing edges to reach another node $v'$, and then continue by following one of the outgoing edges of $v'$.
    Repeating this process, eventually, we reach a node $u_1$ that has at least two incoming edges. 
    We denote this directed path from $v_1$ to $u_1$ as $P_1$.
    Tracing back along one of the other incoming edges of $u_1$, we reach a node $u_0$.
    By tracing back along any of the incoming edges of $u_0$ and continuing this process, we eventually arrive at a node with no incoming edge, which is another source node, which we denote $v_2$.
    We denote this directed path from $v_2$ to $u_0$ as $P_2$.
    A graphical illustration is shown in Figure~\ref{fig:tree-illustrate}.

    \begin{figure}[ht]
        \centering
        \begin{subfigure}{0.45\textwidth}
        \begin{tikzpicture}[scale=1.5, roundnode/.style={draw, circle, minimum size=7mm},]
            \node[roundnode] (v1) at (-1, 1) {$v_1$};
            \node[roundnode] (v2) at (1, 1) {$v_2$};
            \node[roundnode] (v') at (-0.5, 0.5) {$v'$};
            \node[roundnode] (u1) at (0, -1) {$u_1$};
            \node[roundnode] (u0) at (1, -0.3) {$u_0$};

            \draw [dashed] plot [smooth cycle] %
            coordinates {(-1.3, 1.5) (-1.5, 0.8)  (-0.2, -1.3) (0.3, -1) (-0.3, 0.8)};
            \node at (-0.8, 0) {$T_1$};

            \draw [dashed] plot [smooth cycle] %
            coordinates {(1, 1.5) (0.6, 1.2)  (0.6, -0.5) (1.4, -0.5) (1.4, 1.2) };
            \node at (1.2, 0.3) {$T_2'$};
            
            \draw[-latex] (v1) -- (v');
            \draw[dashed, -latex] (v') -- (u1);
            \draw[-latex] (u0) -- (u1);
            \draw[dashed, -latex] (v2) -- (u0);
        \end{tikzpicture}
        \end{subfigure}
        \begin{subfigure}{0.45\textwidth}
        \begin{tikzpicture}[scale=1.5, roundnode/.style={draw, circle, minimum size=7mm},]
            \node[roundnode] (v1) at (-1, 1) {$v_1$};
            \node[roundnode] (v2) at (2, 1) {$v_2$};
            \node[roundnode] (v') at (-0.5, 0.5) {$v'$};
            \node[roundnode] (u1) at (0, -1) {$u_1$};
            \node[roundnode] (u0) at (1.5, 0) {$u_0$};
            \node[roundnode] (u2) at (1.5, -1) {$u_2$};

            \draw [dashed] plot [smooth cycle] %
            coordinates {(-1.3, 1.5) (-1.5, 0.8)  (-0.2, -1.3) (0.3, -1) (-0.3, 0.8)};
            \node at (-0.8, 0) {$T_1$};

            \draw [dashed] plot [smooth cycle] %
            coordinates {(1.2, 1.5) (1, -1.2) (2, -1) (2.4, 1.2) };
            \node at (1.4, 0.6) {$T_2$};
            
            \draw[-latex] (v1) -- (v');
            \draw[dashed, -latex] (v') -- (u1);
            \draw[-latex] (u0) -- (u2);
            \draw[dashed, -latex] (v2) -- (u0);
        \end{tikzpicture}
        \end{subfigure}
        \caption{Left: An graphical illustration of the relation between $v_1$, $v_2$, $u_0$ and $u_1$ in the tree $T$. Right: A graphical illustration of how we create the two trees $T_1$ and $T_2$ from $T$. We remind the reader that these two plots do not include all the nodes in $T$.}
        \label{fig:tree-illustrate}
    \end{figure}
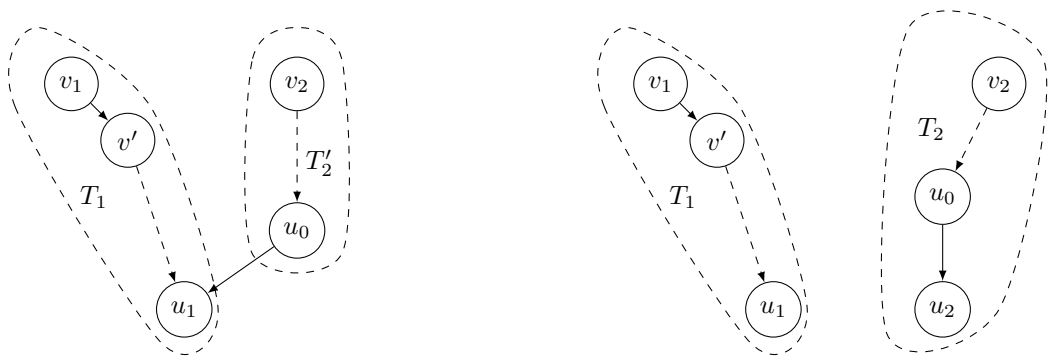

    Recall the definition of $s_\Gamma(\cdot)$ and $e_{\Gamma}(\cdot)$ for any graph $\Gamma$ from Equations~\eqref{eq:source} and~\eqref{eq:sink}.
    Since there are at least two incoming edges to $u_1$, at least one of the two edges is not in $P(s_T(u_1), u_1; T)$, the shortest path from $s_T(u_1)$ to $u_1$. 
    Without loss of generality, we assume that $P(s_T(u_1), u_1; T)$ does not contain the edge $u_0 \to u_1$. 
    Removing the edge $u_0 \to u_1$, the tree $T$ becomes two trees $T_1$ and $T_2'$, where $T_1$ contains $v_1$ and $u_1$, $T_2'$ contains $v_2$ and $u_0$.
    We add an extra node $u_2$ to $T_2'$ by assigning an outgoing edge from $u_0$ to $u_2$, obtaining another tree $T_2$. 
    Noting that $u_2$ is a sink in $T_2$.
    See a graphical illustration in the right plot of Figure~\ref{fig:tree-illustrate}.
    
    We claim that $T_1$ and $T_2$ satisfy Condition~\ref{cond:k-forest}. A proof follows.

    \paragraph{Step 2. $T_1$ satisfies Condition~\ref{cond:k-forest}.}
    Note that the edge $u_0 \to u_1$ is the only edge connecting $T_2$ to $T_1$ and there is no outgoing edge from $T_1$ to $T_2$.
    An immediate result is that for any node $w \in T_1$, the path $P(w, e_T(w); T)$ does not contain any node in $T_2$.
    It follows that
    \begin{equation} \label{eq:eT1}
        e_{T_1}(w) = e_T(w), \quad \text{and} \quad d(w, e_{T_1}(w); T_1) = d(w, e_T(w); T).
    \end{equation}
    Let $s_2$ be any source node in $T_2$. 
    Note that for every node $w$ in $T_1$, the path $P(s_2, w; T)$, if it exists, must go through the edge $u_0 \to u_1$.
    In particular, since $u_1 \in V(T_1)$, the path $P(s_2, u_1; T)$ must contain the edge $u_0 \to u_1$. 
    Therefore, for every node $w$ in $T_1$,
    \begin{equation*}
        d(s_T(u_1), w; T) \leq d(s_T(u_1), u_1; T) + d(u_1, w; T) \leq d(s_2, u_1; T) + d(u_1, w; T) = d(s_2, w; T),
    \end{equation*}
    where the first inequality holds from the definition of $d(\cdot, \cdot)$; the second inequality holds by the assumption that $P(s_T(u_1), u_1; T)$ does not contain the edge $u_0 \to u_1$, while $P(s_2, u_1; T)$, if exists, must contain $u_0 \to u_1$;
    the last equality holds since any shortest path from $s_2$ to $w$ must include the edge $u_0 \to u_1$, therefore can be decomposed into a path from $s_2$ to $u_1$, and a path from $u_1$ to $w$.  
    Furthermore, by definition of $s_T(w)$ in Equation~\eqref{eq:source},
    \begin{equation*}
        d(s_T(w), w; T) \leq d(s_T(u_1), w; T) \leq d(s_2, w; T).
    \end{equation*}
    Therefore, we conclude that $s_{T}(w)$ is in $T_1$, and we have 
    \begin{equation}\label{eq:sT1}
        d(s_{T_1}(w), w; T_1) = d(s_{T}(w), w; T).
    \end{equation}
    Combining Equations~\eqref{eq:eT1} and~\eqref{eq:sT1}, we see that removing the edge $u_0 \to u_1$ does not violate Condition~\ref{cond:k-forest} for $T_1$, since for every node $w \in T_1$,
    \begin{equation*}
        d(s_{T_1}(w), w; T_1) + d(w, e_{T_1}(w); T_1) = d(s_{T}(w), w; T) + d(e_{T}(w), w; T) \leq k-1. 
    \end{equation*}

    \paragraph{Step 3. $T_2$ satisfies Condition~\ref{cond:k-forest}.}
    For any node $w' \in T_2$, we must have $s_T(w') \in T_2$ since there is no outgoing edge from $T_1$ to $T_2$.
    Therefore, we have 
    \begin{equation}\label{eq:sT2}
        s_T(w') = s_{T_2}(w') \quad \text{and} \quad d\left(s_{T_2}(w'), w'; T_2\right) = d\left(s_{T}(w'), w'; T\right).
    \end{equation}
    If $P(w', e_T(w'); T)$, the shortest path from $w'$ to $e_T(w')$, does not contain the edge $u_0 \to u_1$, then we have $e_{T_2}(w') = e_T(w')$ and by Equation~\ref{eq:sT2}, Condition~\ref{cond:k-forest} holds for $w'$.
    Otherwise, if $P(w', e_T(w'); T)$ contains $u_0 \to u_1$, since $u_0$ has an outgoing edge $u_0 \to u_1$ in $T$, $u_0$ is not a sink, and we must have 
    \begin{equation}\label{eq:eTu0}
        e_T(w') = e_T(u_0) \quad \text{and} \quad d(u_0, e_T(u_0); T) \geq 1.
    \end{equation}
    It follows that
    \begin{equation*}
    \begin{aligned}
        d\left(w', u_2; T_2\right) &= d\left(w', u_0; T_2\right) + 1 \\
        &\leq d\left(w', u_0; T\right) + d\left(u_0, e_T(u_0); T\right) = d(w', e_{T}(w'); T),
    \end{aligned}
    \end{equation*}
    where the first equality follows from the fact that $P(w', u_2; T_2)$ can be decomposed into the shortest path from $w'$ to $u_0$ and the edge $u_0 \to u_2$; the first and last inequalities follow from Equation~\eqref{eq:eTu0}.
    Since $u_2$ is a sink in $T_2$, we conclude that 
    \begin{equation*}
        d(w', e_{T_2}(w'); T_2) \leq d\left(w', e_T(w'); T\right).
    \end{equation*}
    Thus, combining the above display with Equation~\eqref{eq:sT2}, we conclude that Condition~\ref{cond:k-forest} also holds for $T_2$.

    \paragraph{Step 4. Finishing the proof.}
    We have verified that both $T_1$ and $T_2$ satisfy Condition~\ref{cond:k-forest}, thus, we can make use of the induction hypothesis on both $T_1$ and $T_2$.
    We proceed by considering three different cases:
    \begin{itemize}
        \item $T_2$ contains less than $k+1$ nodes
        \item $T_1$ contains less than $k+1$ nodes
        \item Both $T_1$ and $T_2$ contain at least $k+1$ nodes
    \end{itemize}
    If $T_2$ contains less than $k+1$ nodes, since $u_2 \neq v_2$, it must be that $T_2$ has at least two leaves.
    Since 
    \begin{equation*}
        |T_1| = |T| + 1 - |T_2| = \beta+ 1 - |T_2| \geq t' k + 2 - k > (t'-1)k + 1,    
    \end{equation*}
    applying the induction hypothesis, it follows that there are $t'+1$ terminals in $T_1$. 
    Note that every terminal in $T_1$ is also a terminal in $T$,
    and every terminal in $T_2$ except $u_2$ is a terminal in $T$. 
    Combining $T_1$ and $T_2$ together, we see that there are at least $(t'+1) + 2 - 1 = t'+2$ terminals in $T$, establishing the claim for $t = t'+1$.

    On the other hand, if $T_1$ contains less than $k+1$ nodes, since $v_1, u_1 \in T_1$ and $u_1 \neq v_1$, $T_1$ also has at least two leaves. 
    Following the same argument as above, we conclude that there are $t'+2$ terminals in $T$.
    
    Finally, if both $T_1$ and $T_2$ contain at least $k+1$ nodes, suppose that $|T_1| = t_1 k + l_1$ for $1\leq l_1 < k$.
    We have 
    \begin{equation*}
        |T_2| = \beta+ 1 - |T_1| \geq (t'-t_1)k + 2 - l_1 > (t'-t_1-1)k + 2.
    \end{equation*}
    By the induction hypothesis, there are at least $t_1+2$ terminals in $T_1$ and at least $t'-t_1+1$ terminals in $T_2$.
    Thus, $T$ contains at least $(t_1+2)+(t'-t_1+1)-1 = t'+2$ terminals.
    Therefore, the claim also holds for $t = t'+1$, and the proof is complete.
\end{proof}


Lemma~\ref{lem:forest-color} extends Lemma~\ref{lem:tree-terminals} to forests. 

\begin{lemma}\label{lem:forest-color}
    Consider a forest $\Gamma$ with $\beta$ nodes such that Condition~\ref{cond:start-end} holds.
    If $\beta= 1$, we cannot assign any internal colors to this forest.
    If $\beta\geq \max\left\{(t-1)k+1, 2\right\}$ for $t \geq 1$ and $k \geq 2$, such that Condition~\ref{cond:k-forest} holds for $\Gamma$ with $k$, we can assign at most $\beta-t$ unique internal colors to this forest.
\end{lemma}
\begin{proof}
    We remind the reader that summaries of all the definitions and terminologies we introduce are given in Sections~\ref{sec:coloring-scheme} and~\ref{sec:preliminary}. 
    When $\beta = 1$, there is only one node in the forest. 
    By Condition~\ref{cond:start-end} and Definition~\ref{def:colors}, we see that $\Gamma$ accepts no internal color. 
    
    When $\beta > 1$, we first consider $\Gamma$ to be a single tree $T$. 
    We assign colors to the terminals first.
    We note that the first color of a source or the second color of a sink can be assigned independently of all other colors in a graph,
    and by Condition~\ref{cond:start-end}, the first color of a source or the second color of a sink is not an internal color. 
    By Lemma~\ref{lem:tree-terminals}, a tree has at least $t+1$ terminals, which accepts $t+1$ distinct terminal colors. 
    Since each tree with $\beta$ nodes accepts at most $\beta+1$ distinct colors by Lemma~\ref{lem:color-bound}, fixing the sources' first colors and the sinks' second colors, we can assign at most 
    \begin{equation} \label{eq:tree-internal-color}
        (\beta+1)-(t+1) = \beta - t   
    \end{equation}
    distinct internal colors.

    If $t = 1$ and $2 \leq \beta\leq k$, we can also prove Lemma~\ref{lem:forest-color} directly. 
    Decomposing the forest into disjoint trees $\cup_{l=1}^K T_l$ where $K\geq 1$, by Lemma~\ref{lem:color-bound} and the fact that each tree accepts at least $2$ terminal colors,
    we see that the forest accepts at most 
    \begin{equation*}
        \sum_{l=1}^K (|T_l|+1-2) = \beta-K \leq \beta-1 = \beta - t
    \end{equation*} 
    distinct internal colors.

    We now prove the general case by induction.
    For the base case $t = 2$, if the forest $\Gamma$ is a single tree, the previous discussion has proved this statement true. 
    Otherwise, $\Gamma$ can be decomposed into disjoint trees $\cup_{l=1}^K T_l$ where $K\geq 2$.
    By Lemma~\ref{lem:color-bound}, it accepts at most $\sum_{l=1}^K (|T_l|+1-2) = \beta-K \leq \beta-2$ distinct internal colors, which establishes the claim for $t=2$.

    Suppose that our claim holds for all $t \in \{2,3,\dots,t'\}$. 
    We show that it also holds for $t = t'+1$.
    The claim has already been proved if the forest is a single tree.
    We now discuss two separate cases.
    \begin{enumerate}
        \item If the forest contains a tree $T_0$ such that $(t_0-1)k+1 \leq |T_0|\leq t_0 k$ for $2\leq t_0 \leq t'+1$, by Equation~\eqref{eq:tree-internal-color}, $T_0$ accepts at most $|T_0|-t_0$ distinct internal colors. 
    There remains 
    \begin{equation*}
        \beta-|T_0| \geq t'k+1-t_0 k = (t'-t_0+1-1)k+1
    \end{equation*} 
    nodes in $\Gamma\backslash T_0$.
    
    If $\beta-|T_0| \geq 2$, then by the induction hypothesis, $\Gamma\backslash T_0$ accepts at most $\beta-|T_0|-(t'-t_0+1)$ internal colors. 
    Thus, the forest accepts at most
    \begin{equation*}
        |T_0|-t_0 + \beta-|T_0|-t'+t_0-1 = \beta- t' - 1 = \beta- t
    \end{equation*}
    distinct internal colors as desired. 
    
    Otherwise, we have $|T_0| = \beta-1$, $t_0 = t'+1$ and $\Gamma\backslash T_0$ accepts no internal colors. 
    The forest accepts at most $(\beta-1) - (t'+1) < \beta-t$ in this case.
    \item If all the trees in the forest have a size at most $k$, we decompose the forest into disjoint trees $\cup_{l=1}^K T_l$. 
    It is straightforward to see that $K \geq t$ since
    \begin{equation*}
        (t-1)k+1 \leq \beta = \sum_{l=1}^K |T_l| \leq K \cdot k.
    \end{equation*}
    By Lemma~\ref{lem:color-bound}, the forest accepts $\sum_{l=1}^K (|T_l|-1) = \beta-K \leq \beta-t$ distinct internal colors, yields the desired bound.
    \end{enumerate}
    Combining these two cases, we see that under the induction hypothesis, our claim holds for $t = t'+1$, and the proof is complete.
\end{proof}


With the groundwork laid in Lemmas~\ref{lem:tree-terminals} and~\ref{lem:forest-color}, we can now prove a general statement for any weakly connected component of $\Gnew$.  

\begin{lemma}\label{lem:loops}
    Consider a weakly connected component $\Gnew_l$ of a graph $\Gnew$ constructed according to Rule~\ref{rule:(ii)} with $\beta_l$ nodes.
    If $\beta_l= 1$, we cannot assign any internal colors to $\Gnew_l$.
    If $\beta_l\geq \max\left\{(t-1)k+1, 2\right\}$ for $t \geq 1$ and $k \geq 2$, we can assign at most $\beta_l-t$ distinct internal colors to $\Gnew_l$.
\end{lemma}
\begin{proof}
    The case when $\beta_l=1$ follows immediately from Lemma~\ref{lem:forest-color}. 
    To prove the statement for $\beta_l\geq \max\left\{(t-1)k+1, 2\right\}$, we perform a series of operations to obtain a forest with at least as many internal colors as $\Gnew_l$ and apply Lemma~\ref{lem:forest-color} to obtain the desired bound. 
    The proof is involved, so we divide it into four steps.
    In Step 1, we convert $\Gnew_l$ into a new graph $\Gnewb_l$ satisfying Condition~\ref{cond:start-end}. 
    In Step 2, we verify that $\Gnewb_l$ satisfies Condition~\ref{cond:k-forest}. 
    In Step 3, we ascertain that there are two types of cycles in $\Gnewb_l$. 
    In Step 4, we show that we can remove these cycles and reduce $\Gnewb_l$ to a forest without changing the bounds on the number of distinct internal colors. 
    We then apply Lemma~\ref{lem:forest-color} to bound the number of internal colors in $\Gnew$ and finish the proof.
    We remind the reader that a summary of the terminologies and definitions used below can be found in Sections~\ref{sec:coloring-scheme} and~\ref{sec:preliminary}. 
    
    \paragraph{Step 1. Preprocess $\Gnew_l$.} We perform the following procedure on $\Gnew_l$:
    \begin{enumerate}
        \item \label{op:mark} For every node $u \in V(\Gnew_l)$, if $u$ connects to a start node $v_s$ through an directed edge $u \to v_s$, we mark $u$ as an end node. 
        If $u$ receives an edge $v_e \to u$ from an end node $v_e$, we mark $u$ as a start node. 
        \item \label{op:case1:op} Repeat the above procedure recursively until there is no new start or end node, then remove all incoming edges to the start nodes and all outgoing edges from the end nodes. 
    \end{enumerate} 
    The justification for the two operations is that if we first assign colors to the original start and end nodes in $\Gnew_l$, the first or second colors of these new start and end nodes, respectively, are fixed and are not internal colors.
    Thus, performing Operations~\ref{op:mark} and~\ref{op:case1:op} will not change the number of distinct internal colors.
    We denote the new graph as $\Gnewb_l$.
    It is straightforward to verify that Condition~\ref{cond:start-end} holds for $\Gnewb_l$ by Operation~\ref{op:case1:op}. 
    
    \paragraph{Step 2. Verify Condition~\ref{cond:k-forest} for $\Gnewb_l$.} 
    Take any node $v_0$ in $\Gnewb_l$.
    We will find a source and a sink in $\Gnewb_l$ such that Condition~\ref{cond:k-forest} holds for $v_0$.
    Suppose that $(r, \ell) \in v_0$ for some $(r, \ell) \in \calI$.
    We take the start node $u_s$ containing $(r, 1)$ and the end node $u_e$ containing $(r,k)$.
    By construction of $\Gnew$ in Rule~\ref{rule:(ii)}, we have  
    \begin{equation} \label{eq:v0-gnewl-k}
        d\left(u_s, v_0; \Gnew_l\right) + d\left(v_0, u_e; \Gnew_l\right) \leq k-1
    \end{equation}
    in $\Gnew_l$.
    Consider the path $P(u_s, v_0; \Gnew_l)$.
    If this path contains an edge $u_1 \to u_2$ removed in Operation~\ref{op:case1:op}, then $u_2$ must be both a start node and a source in $\Gnew_l$. 
    We only need to consider the path $P(u_2, v_0; \Gnew_l)$ since $u_2$ is a closer source to $v_0$ than $u_s$ in $\Gnew_l$.
    Repeating the above argument, there must be a source $\bar{u}_s$ in $\Gnewb_l$, satisfying 
    \begin{equation*}
        d(\bar{u}_s, v_0; \Gnewb_l) \leq d(u_s, v_0; \Gnew_l).
    \end{equation*} 
    By a similar argument, there must be a sink $\bar{u}_e$ in $\Gnewb_l$, satisfying 
    \begin{equation*}
        d(v_0, \bar{u}_e; \Gnewb_l) \leq d(v_0, \bar{u}_e; \Gnew_l).
    \end{equation*} 
    Thus, combined the above two displays with Equation~\eqref{eq:v0-gnewl-k}, we have 
    \begin{equation*}
        d(\bar{u}_s, v_0; \Gnewb_l) + d(v_0, \bar{u}_e; \Gnewb_l) \leq k-1
    \end{equation*}
    establishing Condition~\ref{cond:k-forest} for $\Gnewb_l$.

    \paragraph{Step 3. Reduce $\Gnewb_l$ to a forest.}
    We observe that there are two kinds of cycles in $\Gnewb_l$: cycles contain at least one branch node, and cycles contain no branch node. 
    We handle these two cases separately.
    \begin{enumerate}[label=\Roman*.]
        \item Suppose that $\Gnewb_l$ contains a cycle with a branch node $v_0$. 
        Without loss of generality, we assume that $v_0$ has two outgoing edges $v_0 \to v_1$ and $v_0 \to v_2$ in the cycle.
        If $v_0$ receives two incoming edges, we can reverse the directions of all the edges in $\Gnewb_l$ and apply the same argument described below. 
        Figure~\ref{fig:step3-case1} presents a graphical illustration of this case. 

        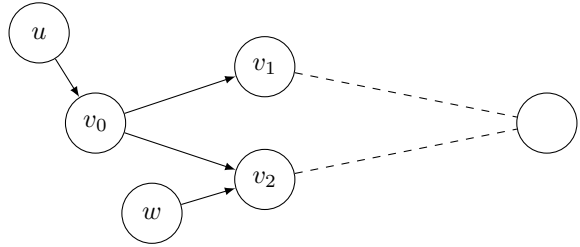
\begin{figure}[hb!]
            \centering
            \begin{tikzpicture}[scale=1.5, every node/.style={draw, circle, minimum size=0.8cm}]
            \node (v0) at (0, 0) {$v_0$};
            \node (u) at (-0.5, 0.8) {$u$};
            \node (v1) at (1.5, 0.5) {$v_1$};
            \node (v2) at (1.5, -0.5) {$v_2$};
            \node (w) at (0.5, -0.8) {$w$};
            \node (ve) at (4, 0) {};
            
            \draw[-latex] (u) -- (v0);
            \draw[-latex] (v0) -- (v1);
            \draw[-latex] (v0) -- (v2);
            \draw[-latex] (w) -- (v2);
            \draw[dashed] (v1) -- (ve);
            \draw[dashed] (v2) -- (ve);
            \end{tikzpicture}
            \caption{A graphical illustration of the first case, where $v_0$ is a branch node with two outgoing edges. Dashed lines and the empty node represent the other edges and nodes in the cycle. The directions of the dashed lines do not matter and are left unspecified. As a reminder, the figure does not necessarily contain every node and edge in $\Gnewb_l$.}\label{fig:step3-case1}
        \end{figure}

        Recall the definition of $s_\Gamma(\cdot)$ and $e_{\Gamma}(\cdot)$ for any graph $\Gamma$ from Equations~\eqref{eq:source} and~\eqref{eq:sink}.
        Without loss of generality, we assume that 
        \begin{equation}\label{eq:v1-v2-shorter}
            d\left(v_1, e_{\Gnewb_l}(v_1); \Gnewb_l\right) \leq d\left(v_2, e_{\Gnewb_l}(v_2); \Gnewb_l\right),
        \end{equation}
        and take the following operation:
        \begin{enumerate}[label=(\Roman{enumi}.\alph*)]
            \item \label{op:step3-case1-op1} For every node $w$ with an edge $w \to v_2$, we change this edge to $w \to v_1$.
        \end{enumerate}
        Since the vertices $v_1$ and $v_2$ both receive an incoming edge from $v_0$, they must share the same first color.
        Thus, Operation~\ref{op:step3-case1-op1} does not modify the color of $w$, and so does not affect the number of distinct internal colors we can assign. 
        It is straightforward to see that Condition~\ref{cond:start-end} still holds after this operation. 
        By Equation~\eqref{eq:v1-v2-shorter}, the distance $d(w, e_{\Gnewb_l}(w); \Gnewb_l)$ does not increase, and $d(s_{\Gnewb_l}(w), w; \Gnewb_l)$ remains the same.
        Therefore, Condition~\ref{cond:k-forest} still holds.
        
        We then take the following operation: 
        \begin{enumerate}[label=(\Roman{enumi}.\alph*)]\addtocounter{enumii}{1}
            \item \label{op:step3-case1-op2} Remove the edge $v_0 \to v_2$ and mark $v_2$ as a start node. 
        \end{enumerate}
        An illustration of the resulting graph is shown in Figure~\ref{fig:step3-case1-after}. 
        \begin{figure}[ht]
            \centering
            \begin{tikzpicture}[scale=1.5, every node/.style={draw, circle, minimum size=0.8cm}]
            \node (v0) at (0, 0) {$v_0$};
            \node (u) at (-0.5, 0.8) {$u$};
            \node (v1) at (1.5, 0.5) {$v_1$};
            \node (v2) at (1.5, -0.5) {$v_2$};
            \node (w) at (0.5, 0.8) {$w$};
            \node (ve) at (4, 0) {};
            
            \draw[-latex] (u) -- (v0);
            \draw[-latex] (v0) -- (v1);
            \draw[-latex] (w) -- (v1);
            \draw[dashed] (v1) -- (ve);
            \draw[dashed] (v2) -- (ve);
            \end{tikzpicture}
            \caption{A graphical illustration of the first case after performing the Operations~\ref{op:step3-case1-op1} and~\ref{op:step3-case1-op2}.}\label{fig:step3-case1-after}
        \end{figure}
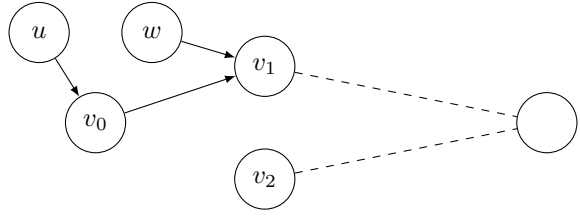
        In $\Gnewb_l$, the first color of $v_2$ is fixed once we fix the second color of $v_0$ by Observation~\ref{obs:color}. 
        Removing the edge $v_0 \to v_2$ and changing $v_2$ to a start node does not affect the number of internal colors, so the new graph obtained from Operation~\ref{op:step3-case1-op2} provides a valid upper bound for the number of distinct internal colors.

        Denote the new graph obtained from performing Operations~\ref{op:step3-case1-op1} and~\ref{op:step3-case1-op2} on $\Gnewb_l$ as $\Gnewt_l$.
        Verifying that $\Gnewt_l$ satisfies Condition~\ref{cond:start-end} is straightforward.
        We now demonstrate that Condition~\ref{cond:k-forest} holds.
        
        First, we observe that for any node $u_0$, if neither of the paths $P(u_0, e_{\Gnewb_l}(u_0); \Gnewb_l)$ nor $P(s_{\Gnewb_l}(u_0), u_0; \Gnewb_l)$ passes through $v_2$, then these two paths still exist in $\Gnewt_l$ and Condition~\ref{cond:k-forest} holds for $u_0$ in $\Gnewt_l$.
        Now consider the case where there is an edge $w \to v_2$ either in $P(u_0, e_{\Gnewb_l}(u_0); \Gnewb_l)$ or $P(s_{\Gnewb_l}(u_0), u_0; \Gnewb_l)$.
        
        In the first case, we must have $e_{\Gnewb_l}(u_0) = e_{\Gnewb_l}(v_2)$.
        It follows that 
        \begin{equation*}
        \begin{aligned}
            d\left(u_0, e_{\Gnewb_l}(u_0); \Gnewb_l\right) &= d\left(u_0, w; \Gnewb_l\right) + d\left(w, v_2; \Gnewb_l\right) + d\left(v_2, e_{\Gnewb_l}(v_2); \Gnewb_l\right)\\
            &\geq d\left(u_0, w; \Gnewb_l\right) + 1 + d\left(v_1, e_{\Gnewb_l}(v_1); \Gnewb_l\right),
        \end{aligned}
        \end{equation*}
        where the inequality follows from Equation~\eqref{eq:v1-v2-shorter}. 
        Since the shortest path from $u_0$ to $w$ does not contain any incoming edge to $v_2$ and $d\left(w, v_1; \Gnewt_l\right) = 1$ by Operation~\ref{op:step3-case1-op1}, the right hand side in the above display is equal to
        \begin{equation*}
        \begin{aligned}
            d\left(u_0, w; \Gnewt_l\right) + d\left(w, v_1; \Gnewt_l\right) + d\left(v_1, e_{\Gnewb_l}(v_1); \Gnewt_l\right)
            \geq d\left(u_0, e_{\Gnewt_l}(u_0); \Gnewt_l\right).
        \end{aligned}
        \end{equation*}
        Combining the above two displays, we conclude that
        \begin{equation}\label{eq:u0et}
            d\left(u_0, e_{\Gnewb_l}(u_0); \Gnewb_l\right) \geq d\left(u_0, e_{\Gnewt_l}(u_0); \Gnewt_l\right). 
        \end{equation}

        If there is an edge $w \to v_2$ in $P(s_{\Gnewb_l}(u_0), u_0; \Gnewb_l)$, then clearly
        \begin{equation*}
            d\left(s_{\Gnewb_l}(u_0), u_0; \Gnewb_l\right) \geq d\left(v_2, u_0; \Gnewb_l\right).
        \end{equation*}
        Since $v_2$ is a new start node and source in $\Gnewt_l$ by Operation~\ref{op:step3-case1-op2}, and the shortest path $P(v_2, u_0; \Gnewb_l)$ does not contain any edge incoming to $v_2$, from the above display we immediately have 
        \begin{equation}\label{eq:stu0}
            d\left(s_{\Gnewt_l}(u_0), u_0; \Gnewt_l\right) \leq d\left(v_2, u_0; \Gnewt_l\right) \leq d\left(s_{\Gnewb_l}(u_0), u_0; \Gnewb_l\right).
        \end{equation}
        
        Combining Equations~\eqref{eq:u0et} and~\eqref{eq:stu0}, we conclude that Condition~\ref{cond:k-forest} is not violated in $\Gnewt_l$.
        
        \item Suppose that $\Gnewb_l$ contains a cycle with no branch node. 
        Since Step 2 shows that $\Gnewb_l$ satisfies Conditions~\ref{cond:start-end}, no start or end node can exist within such a cycle. 
        There must be both an incoming and an outgoing edge connecting nodes inside the cycle to those outside it.

        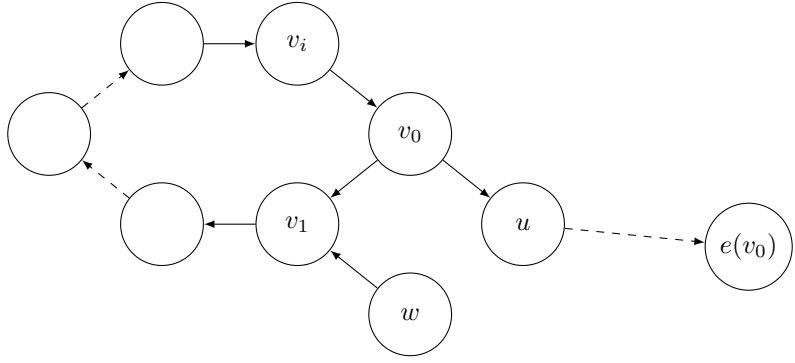
\begin{figure}[ht]
            \centering
            \begin{tikzpicture}[scale=1.5, every node/.style={draw, circle, minimum size=1.1cm}]
            \node (vs) at (0, 0) {};
            \node (vi1) at (1, 0.8) {};
            \node (vi2) at (2+0.2, 0.8) {$v_{i}$};
            \node (v0) at (3+0.2, 0) {$v_0$};
            \node (w) at (3+0.2, -1.6) {$w$};
            \node (v1) at (2+0.2, -0.8) {$v_1$};
            \node (vi3) at (1, -0.8) {};
            \node (u) at (4+0.2, -0.8) {$u$};
            \node (ev0) at (6+0.2, -1) {$e(v_0)$};
            
            \draw[dashed, -latex] (vs) -- (vi1);
            \draw[-latex] (vi1) -- (vi2);
            \draw[-latex] (vi2) -- (v0);
            \draw[-latex] (v0) -- (v1);
            \draw[-latex] (v0) -- (u);
            \draw[-latex] (v1) -- (vi3);
            \draw[-latex] (w) -- (v1);
            \draw[dashed, -latex] (vi3) -- (vs);
            \draw[dashed, -latex] (u) -- (ev0);
            \end{tikzpicture}
            \caption{A graphical illustration of the second case. Dashed edges indicate paths, which means that there might be multiple nodes and edges along them. As a reminder, the figure does not necessarily include all the nodes and edges in the graph.}\label{fig:step3-case2}
        \end{figure}

        Denote the cycle as $P_0$.
        Suppose that $v_0 \in V(P_0)$ such that for any other node $v_i \in V(P_0)$,
        \begin{equation} \label{eq:v0-shortest}
            d\left(v_0, e_{\Gnewb_l}(v_0); \Gnewb_l\right) \leq d\left(v_i, e_{\Gnewb_l}(v_i); \Gnewb_l\right).
        \end{equation}
        Then there must be an outgoing edge from $v_0$ to a node $u \notin V(P_0)$, and
        \begin{equation} \label{eq:u-in-v0ev0}
            u \in V(P(v_0, e_{\Gnewb_l}(v_0); \Gnewb_l)). 
        \end{equation}
        A graphical illustration is included in Figure~\ref{fig:step3-case2}. 
        
        Denote the vertex following $v_0$ in $P_0$ as $v_1$, meaning that the edge $v_0 \to v_1$ is in $P_0$.
        Then $u$ and $v_1$ share the same first color.
        We perform operations similar to Operation~\ref{op:step3-case1-op1} on $\Gnewb_l$:
        \begin{enumerate}[label=(\Roman{enumi}.\alph*)]
            \item \label{op:step3-case2-op1} For every node $w \notin V(P_0)$ with an edge $w \to v_1$, we change this edge to $w \to u$.
            \item \label{op:step3-case2-op2} Remove the edge $v_0 \to v_1$ and mark $v_1$ as a start node. 
        \end{enumerate}
        An illustration of the resulting graph is shown in Figure~\ref{fig:step3-case2-after}.
        Since $u$ and $v_1$ share the same first color, Operation~\ref{op:step3-case2-op1} does not modify the color of $w$, and therefore, it does not affect the number of unique internal colors we can assign. 
        In $\Gnewb_l$, the first color of $v_1$ is fixed once we determine the second color of $v_0$.
        Marking $v_1$ as a start node does not alter the number of internal colors, ensuring that the new graph still provides a valid bound on the number of unique internal colors.

        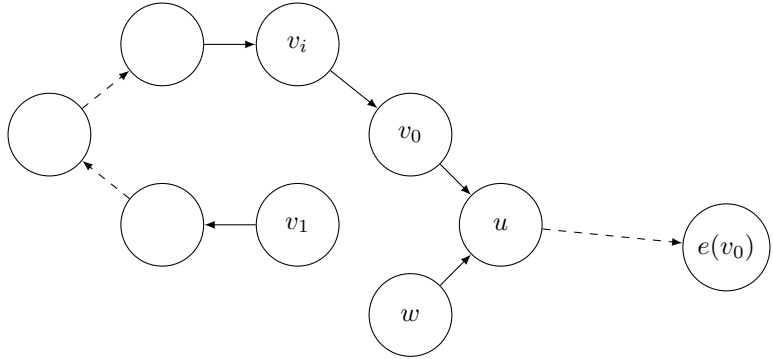
\begin{figure}[ht]
            \centering
            \begin{tikzpicture}[scale=1.5, every node/.style={draw, circle, minimum size=1.1cm}]
            \node (vs) at (0, 0) {};
            \node (vi1) at (1, 0.8) {};
            \node (vi2) at (2+0.2, 0.8) {$v_{i}$};
            \node (v0) at (3+0.2, 0) {$v_0$};
            \node (w) at (3+0.2, -1.6) {$w$};
            \node (v1) at (2+0.2, -0.8) {$v_1$};
            \node (vi3) at (1, -0.8) {};
            \node (u) at (4, -0.8) {$u$};
            \node (ev0) at (6, -1) {$e(v_0)$};
            
            \draw[dashed, -latex] (vs) -- (vi1);
            \draw[-latex] (vi1) -- (vi2);
            \draw[-latex] (vi2) -- (v0);
            \draw[-latex] (v0) -- (u);
            \draw[-latex] (v1) -- (vi3);
            \draw[-latex] (w) -- (u);
            \draw[dashed, -latex] (vi3) -- (vs);
            \draw[dashed, -latex] (u) -- (ev0);
            \end{tikzpicture}
            \caption{A graphical illustration of the second case after performing Operations~\ref{op:step3-case2-op1} and~\ref{op:step3-case2-op2}. Dashed edges indicate paths, meaning multiple nodes and edges might be along them. As a reminder, the figure does not necessarily include all the nodes and edges in the graph.}\label{fig:step3-case2-after}
        \end{figure}

        Denote the new graph obtained from performing Operations~\ref{op:step3-case2-op1} and~\ref{op:step3-case2-op2} on $\Gnewb_l$ as $\Gnewt_l$. 
        Verifying that $\Gnewt_l$ satisfies Condition~\ref{cond:start-end} is straightforward, and it remains to show that $\Gnewt_l$ satisfies Condition~\ref{cond:k-forest}.
        
        Similar to the discussion in the first case, for any node $u_0$ in $\Gnewb_l$, if neither $P(u_0, e_{\Gnewb_l}(u_0); \Gnewb_l)$ nor $P(s_{\Gnewb_l}(u_0), u_0; \Gnewb_l)$ passes through $v_1$ in $\Gnewb_l$, then these two paths still exist in $\Gnewt_l$ and Condition~\ref{cond:k-forest} holds for such $u_0$ in $\Gnewt_l$.
        
        Consider the case where there is an edge $w \to v_1$ in either $P(u_0, e_{\Gnewb_l}(u_0); \Gnewb_l)$ or $P(s_{\Gnewb_l}(u_0), u_0; \Gnewb_l)$.
        In the first case, we must have $e_{\Gnewb_l}(u_0) = e_{\Gnewb_l}(v_1)$, and it follows that 
        \begin{equation}\label{eq:u0eu0}
        \begin{aligned}
            d\left(u_0, e_{\Gnewb_l}(u_0); \Gnewb_l\right) &= d\left(u_0, w; \Gnewb_l\right) + d\left(w, v_1; \Gnewb_l\right) + d\left(v_1, e_{\Gnewb_l}(v_1); \Gnewb_l\right)\\
            &\geq d\left(u_0, w; \Gnewb_l\right) + 1 + d\left(v_0, e_{\Gnewb_l}(v_0); \Gnewb_l\right),
        \end{aligned}
        \end{equation}
        where the inequality follows from Equation~\eqref{eq:v0-shortest}. 
        Since the paths $P(u_0, w; \Gnewb_l)$ and $P(v_0, e_{\Gnewb_l}(v_0); \Gnewb_l)$ do not contain any incoming edge to $v_1$ (otherwise, they are not shortest paths), we have 
        \begin{equation*} 
            P(v_0, e_{\Gnewt_l}(v_0); \Gnewt_l) = P(v_0, e_{\Gnewb_l}(v_0); \Gnewb_l) \quad \text{and} \quad P(u_0, w; \Gnewt_l) = P(u_0, w; \Gnewb_l).
        \end{equation*}
        Applying the above display and Equation~\eqref{eq:u-in-v0ev0} to Equation~\eqref{eq:u0eu0}, we obtain 
        \begin{equation*}\label{eq:u0eu0-rhs}
        \begin{aligned}
        d\left(u_0, e_{\Gnewb_l}(u_0); \Gnewb_l\right)
        &\ge
            d\left(u_0, w; \Gnewt_l\right) + 1 + d\left(v_0, u; \Gnewt_l\right) + d\left(u, e_{\Gnewb_l}(v_0); \Gnewt_l\right). 
        \end{aligned}
        \end{equation*}
        From the definition of $d(\cdot, \cdot)$, we have 
        \begin{equation*}
        \begin{aligned}
            d\left(u_0, w; \Gnewt_l\right) + d\left(w, u; \Gnewt_l\right) + d\left(u, e_{\Gnewb_l}(v_0); \Gnewt_l\right)
            \geq d\left(u_0, e_{\Gnewt_l}(u_0); \Gnewt_l\right). 
        \end{aligned}
        \end{equation*}
        Since $d\left(w, u; \Gnewt_l\right) = 1$ by Operation~\ref{op:step3-case2-op1}, combining the above two displays yields 
        \begin{equation}\label{eq:gnewl-u0-e}
            d\left(u_0, e_{\Gnewt_l}(u_0); \Gnewt_l\right) \leq d\left(u_0, e_{\Gnewb_l}(u_0); \Gnewb_l\right). 
        \end{equation}

        For the second case, if there is an edge $w \to v_1$ in $P(s_{\Gnewb_l}(u_0), u_0; \Gnewb_l)$, then clearly,
        \begin{equation*}
            d\left(s_{\Gnewb_l}(u_0), u_0; \Gnewb_l\right) \geq d\left(v_1, u_0; \Gnewb_l\right).
        \end{equation*}
        Since $v_1$ is a new start node and source in $\Gnewt_l$ by Operation~\ref{op:step3-case2-op2}, and the path $P(v_1, u_0; \Gnewb_l)$ does not contain any incoming edge to $v_1$ (otherwise, it is not the shortest path), it follows from the above display that
        \begin{equation}\label{eq:gnewl-s-u0}
            d\left(s_{\Gnewt_l}(u_0), u_0; \Gnewt_l\right) \leq d\left(v_1, u_0; \Gnewt_l\right) \leq d\left(s_{\Gnewb_l}(u_0), u_0; \Gnewb_l\right).
        \end{equation}
        Combining Equations~\eqref{eq:gnewl-u0-e} and~\eqref{eq:gnewl-s-u0}, we conclude that Condition~\ref{cond:k-forest} holds for $\Gnewt_l$.
    \end{enumerate}

    \paragraph{Step 4. Apply Lemma~\ref{lem:forest-color}.} Finally, by performing Operations~\ref{op:step3-case1-op1}, \ref{op:step3-case1-op2}, \ref{op:step3-case2-op1} and~\ref{op:step3-case2-op2}, we remove one edge each time. 
    These steps gradually simplify the graph structure while maintaining the necessary bounds on the number of distinct internal colors.
    Since there are only a finite number of edges, the above procedure eventually produces a cycle-free graph. 
    Noting that no new node is introduced in the above procedure, we successfully reduce $\Gnew_l$ to a forest with $\beta_l$ nodes, satisfying Conditions~\ref{cond:start-end} and~\ref{cond:k-forest}. 
    Applying Lemma~\ref{lem:forest-color} yields the desired result. 
\end{proof}

\subsubsection{Results for bounding $|\psi^{-1}(G)|$}\label{sec:results}

With results in Lemmas~\ref{lem:color-bound} through \ref{lem:loops}, we can provide bounds for $|\psi^{-1}(G)|$. 
Based on Rule~\ref{rule:(ii)}, each $G$ corresponds to a $\Gnew$. 
Given a $\Gnew$, one can also recover $G$ by identifying which vertices of $G$ are included in the same vertex of $\Gnew$.
Thus, there is a bijection $\eta: G \mapsto \Gnew$ and bounding $|\psi^{-1}(G)|$ can be achieved by bounding $|(\eta \circ \psi)^{-1} (\Gnew)|$.
For notational simplicity, we define a mapping $\pi = \eta \circ \psi$, which maps $\mJ$ to $\Gnew$. 

From the coloring scheme described in Definition~\ref{def:coloring} and Observation~\ref{obs:color-value}, we see that $|\pi^{-1}(\Gnew)|$ can be controlled by bounding the number of distinct colors of each type defined in Definition~\ref{def:colors}. 
Since each connected component $\Gnew_l$ of $\Gnew$ can be colored independently, we can first bound the number of terms associated with each $\Gnew_l$ separately, then combine to get a bound for $\Gnew$. 
Lemma~\ref{lem:term-bound-cc} first considers the bound for each $\Gnew_l$. 

\begin{lemma}\label{lem:term-bound-cc}
    Consider a weakly connected component $\Gnew_l$ of a graph $\Gnew$ constructed according to Rule~\ref{rule:(ii)} with $\beta_l$ nodes.
    Without loss of generality, assume that $0\leq N_x \leq N_y \leq n$, where $N_x = \|\vx\|_0$, $N_y = \|\vy\|_0$. Let $S_{xy} = |\{i \in [n]: x_i \neq 0, y_i \neq 0\}|$.
    Recall that the integer $k \geq 2$ is the same as in the term $\vx^\top \mH^k \vy$. 
    The number of terms associated with $\Gnew_l$ is bounded by
    \begin{equation}\label{eq:term-bound-ccs}
    \bar{\tau}_{\beta_l} := \begin{cases}
        S_{xy}, & \mbox{ for } \beta_l = 1,\\
        N_x N_y n^{\beta_l-2}, & \mbox{ for } 2 \leq \beta_l < k,\\
        N_x N_y n^{\beta_l-1}, & \mbox{ for } \beta_l = k,\\
        N_x N_y^{t} n^{\beta_l-t},  & \mbox{ for } \beta_l \geq (t-1)k+1, t\geq 2, t \mbox{ is an integer}.
    \end{cases}
    \end{equation}
\end{lemma}
\begin{proof}
    We refer the reader to Sections~\ref{sec:coloring-scheme} and~\ref{sec:preliminary} for summaries of the definitions and terminologies introduced for the proof. 

    When $\beta_l \geq  \max\left\{k, (t-1)k+1\right\}$ for some integers $t \geq 1$ and $k \geq 2$, by Lemma~\ref{lem:loops}, $\Gnew_l$ accepts at most $\beta_l-t$ internal colors, which contributes a factor of $n^{\beta_l - t}$ to the number of terms. 
    By Lemma~\ref{lem:color-bound}, $\Gnew_l$ accepts at most $\beta_l + 1$ total unique colors, so we can assign at most $t+1$ additional unique terminal colors to it. 
    Since every graph must have at least a start and an end node, we can assign a start color and $t$ end color and obtain the third and last bounds.

    It remains to obtain the bound for $1 \leq \beta_l < k$. 
    Suppose that a node in $\Gnew_l$ contains the vertex $(r, \ell) \in \calI$. Then, the vertex sequence
    \begin{equation*}
        (r,1), (r,2), \ldots, (r,k)
    \end{equation*}
    appear in vertices of $\Gnew_l$. 
    Since there are fewer than $k$ nodes in $\Gnew_l$, at least two vertices in the sequence above must be in the same node of $\Gnew_l$. 
    When $\beta_l = 1$, all vertices in the sequence above are in the same node of $\Gnew_l$, so $(j_l^{(r)}, j_{l+1}^{(r)})$ for $l \in [k]$ must all be the same, which further implies that $j_1^{(r)} = j_2^{(r)} = \cdots = j_{k+1}^{(r)}$. 
    Denote the value of these index as $\phi^{(0)}$, it is straightforward to see that $\phi^{(0)} \in \{i \in [n]: x_i \neq 0, y_i \neq 0\}$, yielding the first bound.

    Consider $2\leq \beta_l < k$.
    Suppose that $v_s$ and $v_e$ are two nodes in $\Gnew_l$, such that $(r,1) \in v_s$ and $(r,k) \in v_e$. 
    We note that if there is an edge $v_e \to v_s$, it can be removed without loss of generality: since there is always a path in $\Gnew_l$ that travels from $v_s$ to $v_e$, $\Gnew_l$ remains weakly connected after removing the edge $v_e \to v_s$.
    Additionally, since edges restrict how colors can be assigned by Observation~\ref{obs:color}, removing edges in $\Gnew_l$ reduces these constraints and does not decrease the number of distinct colors allowed in $\Gnew_l$, providing an upper bound on the number of distinct colors. 
    Hence, we assume without loss of generality that the edge $v_e \to v_s$ does not exist in $\Gnew_l$, so we can always assign a start color and an end color to $\Gnew_l$. 
    
    If there is no cycle in $\Gnew_l$, since there are less than $k$ nodes in $\Gnew_l$, there must be a $(r, \ell) \in \calI$ such that both $(r, \ell)$ and $(r, \ell+1)$ are in the same node $v$ of $\Gnew_l$. 
    By Rule~\ref{rule:(ii)} and the coloring scheme in Section~\ref{sec:coloring-scheme}, the first and second colors of $v$ must be the same color. 
    Following the proof of Lemma~\ref{lem:color-bound}, one sees that the number of total unique colors $\Gnew_l$ accepts is bounded by $\beta_l$. 

    If there is at least one cycle in $\Gnew_l$, we apply the techniques introduced in Step 3 of the proof of Lemma~\ref{lem:loops}.
    We remove an edge from $\Gnew_l$ and mark one node as a start or end node. 
    This keeps $\Gnew_l$ weakly connected while allowing it to have one extra terminal color.
    Removing this terminal color and applying Lemma~\ref{lem:color-bound}, we see that the total number of unique colors in $\Gnew_l$ is bounded by $\beta_l$. 
    Thus, after assigning the start and end colors to $\Gnew_l$, the number of unique internal colors is bounded by $\beta_l-2$.
    And it follows that the total number of terms contributed by $\Gnew_l$ is bounded by $N_x N_y n^{\beta_l - 2}$ as desired.
\end{proof}

Using the bound for the number of terms associated with each connected component of $\Gnew$ in Lemma~\ref{lem:term-bound-cc}, we obtain a bound for $|\pi^{-1}(\Gnew)|$, which is also a bound for $|\psi^{-1}(G)|$.

\begin{lemma} \label{lem:term-bound}
    Without loss of generality, suppose that $0\leq N_x \leq N_y\leq n$ (otherwise, we switch the role of $\vx$ and $\vy$).
    For a graph $\Gnew$ constructed from Rule~\ref{rule:(ii)} with $L$ nodes, the number of terms associated with $\Gnew$ in Equation~\eqref{eq:moment-i}, $|\pi^{-1}(\Gnew)|$ , is bounded by 
    \begin{equation} \label{eq:term-bound}
    \tau_L := \begin{cases}
        S_{xy}, & \mbox{ for } L = 1, \\
        N_x N_y n^{L-2}, & \mbox{ for } 2 \leq L < k, \\
        N_x^{t-1} N_y^{t} n^{L-t} & \mbox{ for } (t-1)k+1 \leq L < tk, \; 2 \leq t \leq p/2,\\
        N_x^{t} N_y^{t} n^{L-t}, & \mbox{ for } L = tk, \; 1 \leq t \leq p/2.
    \end{cases}
    \end{equation}
\end{lemma}
\begin{proof}
    The case $L=1$ follows directly from Lemma~\ref{lem:term-bound-cc}. 
    For $L \geq 2$, suppose that $\Gnew = \cup_{l=1}^K \Gnew_l$ contains $K$ connected components each of size $\beta_l$.
    Since $S_{xy} \leq \min\left\{N_x, N_y\right\}$, if any $\beta_l = 1$, $\Gnew$ contributes less terms compared to the case with no $\beta_l = 1$. 
    Thus, we only need to consider $\beta_l \geq 2$ for all $\Gnew_l$. 
    
    Suppose that for $l \in \calC_0 \subseteq [K]$, $2\leq \beta_l < k$; for $l \in \calC_1\subseteq [K]$, $\beta_l = k$, and for $l \in \calC_t\subseteq [K]$, $tk \geq \beta_l \geq(t-1) k+1$ for $2 \leq t \leq p/2$. 
    By Lemma~\ref{lem:term-bound-cc}, we obtain the bound for $|\pi^{-1}(\Gnew)|$
    \begin{equation}\label{eq:num-terms}
    \begin{aligned}
        \prod_{t=0}^{p/2}\prod_{l \in \calC_t} \bar{\tau}_{\beta_l}
        &= \left(n^{-2} N_x N_y\right)^{|\calC_0|} n^{\sum_{l \in \calC_0} \beta_l} ~ \prod_{t=1}^{p/2} \left(n^{-t} N_x N_y^t\right)^{|\calC_t|} n^{\sum_{l \in \calC_t} \beta_l} \\
        &= n^{L} ~ \left(\frac{N_x N_y}{n^2}\right)^{|\calC_0|} ~\prod_{t=1}^{p/2} \left(\frac{N_x N_y^t}{n^t}\right)^{|\calC_t|},
    \end{aligned}
    \end{equation}
    where $\bar{\tau}_{\beta_l}$ is defined in Equation~\eqref{eq:term-bound-ccs}, which bounds the number of terms associated with $\Gnew_l$ when $\Gnew_l$ has $\beta_l$ nodes. 
    
    For $2 \leq L < k$, the above display is given by
    \begin{equation*}
        n^{L} \cdot \left(\frac{N_x N_y}{n^2}\right)^{|\calC_0|},
    \end{equation*}
    which is maximized when $|\calC_0| = 1$ since $N_x \leq N_y \leq n$ and $|\calC_0| \geq 1$, yielding the second bound in Equation~\eqref{eq:term-bound}.

    We now consider the case $L \geq k$. 
    Note that for any $\beta$ nodes in $\Gnew$, with $(t-1)k \leq \beta < tk$ for $t \geq 2$, these $\beta$ nodes can form
    \begin{enumerate}
        \item one connected component $\Gnew_l$ with $l \in \calC_t$, 
        \item $t-2$ different connected components $\Gnew_l$ with $l \in \calC_1$ and one connected component $\Gnew_{l'}$ with $l' \in \calC_2$,
        \item $t-1$ different connected components $\Gnew_l$ with $l \in \calC_1$ and one connected component $\Gnew_{l'}$ with $l' \in \calC_0$.
    \end{enumerate}
    Since for any $t \geq 3$,
    \begin{equation*}
        \frac{N_x N_y^t}{n^t} \leq \left(\frac{N_x N_y}{n}\right)^{t-2} \frac{N_x N_y^2}{n^2} = \frac{N_x^{t-1} N_y^t}{n^t},
    \end{equation*}
    and 
    \begin{equation*}
        \frac{N_x^2 N_y^2}{n^3} \leq \frac{N_x N_y^2}{n^2},
    \end{equation*}
    the first and third choices contribute fewer terms than the second choice. 
    Thus, Equation~\eqref{eq:num-terms} is maximized when we take $|\calC_0|=0$ and $|\calC_t| = 0$ for all $t \geq 3$, and maximizing $|\calC_1|$.
    Hence, taking $|\calC_0| = 0$ and $|\calC_t| = 0$ for $t \geq 3$, Equation~\eqref{eq:num-terms} becomes
    \begin{equation*}
        n^{L} \cdot \left(\frac{N_x N_y}{n}\right)^{|\calC_1|} \cdot \left(\frac{N_x N_y^2}{n^2}\right)^{|\calC_2|}, 
    \end{equation*}
    which is maximized when $|\calC_1| = t$ and $|\calC_2| = 0$ if $L = tk$, and $|\calC_1| = t-2$, $|\calC_2| = 1$ if $(t-1)k+1 \leq L < tk$. 
    Thus, we obtain the last two bounds in Equation~\eqref{eq:term-bound}, completing the proof.
\end{proof}

\section{AUXILIARY RESULTS} \label{sec:aux}
\begin{lemma} \label{lem:lambda-asymm}
    Suppose $\mMstar$ is a rank-$r$ symmetric matrix whose top-$r$ eigenvalues obey $|\lambdastar_1| \ge \cdots \ge |\lambdastar_r| > 0$.
    Let $\lambda_1,\lambda_2,\dots,\lambda_r$ be the analogous $r$ leading eigenvalues of $\mM = \mMstar + \mH$, also sorted by modulus.
    If $\|\mH\| < |\lambdastar_r| / 2$, then for any $1 \leq l \leq r$, there exists $j \in [r]$ such that 
    \begin{equation*}
        |\lambda_l - \lambdastar_j| \leq \|\mH\| .
    \end{equation*}
    In addition, if $r=1$, then both the leading eigenvalue and the leading eigenvector of $\boldsymbol{M}$ are real-valued.
\end{lemma}
\begin{proof}
    See Lemma 2 in \cite{chen2021asymmetry}. 
\end{proof}

\begin{lemma} \label{lem:inner-product}
    Suppose $\mMstar$ is a rank-$r$ symmetric matrix with $r$ non-zero eigenvalues obeying $\lambdastar_{\max} = |\lambdastar_1| \geq \cdots \geq |\lambdastar_r| = |\lambdastar_{\min}| > 0$ and associated eigenvectors $\vustar_1, \cdots, \vustar_r$. 
    Define the condition number $\kappa = \lambdastar_{\max}/\lambdastar_{\min}$. If $\|\mH\| \leq \lambdastar_{\max}/(4\kappa)$, then the top $r$ eigenvectors $\vu_1, \vu_2, \dots, \vu_r$ of $\mM = \mMstar + \mH$ obey
    \begin{equation} \label{eq:inner-prod}
        \sum_{j=1}^r \left|\vustart_j \vu_l\right|^2 \geq 1 - \frac{64\kappa^4}{9\left(\lambdastar_{\max}\right)^2} \|\mH\|^2, \quad l \in [r].
    \end{equation}
    In addition, if $r=1$, then one further has
    \begin{equation*}
        \min\left\{ \|\vu - \vustar_1\|_2, \|\vu + \vustar_1\|_2 \right\} \leq \frac{8\sqrt{2}}{3|\lambdastar_1|} \|\mH\|. 
    \end{equation*}
\end{lemma}
\begin{proof}
    See Lemma 3 in \cite{chen2021asymmetry}. 
\end{proof}

\begin{lemma}[Bernstein's Inequality]\label{lem:bernstein}
    Let $X_1, \ldots, X_n$ be independent, centered, real random variables, and assume that each one is uniformly bounded:
    \begin{equation*}
        \E X_k = 0, \quad \text{and} \quad |X_k| \leq L \quad \text{for each } k \in [n]. 
    \end{equation*}
    Introduce the sum $Z=\sum_{k=1}^n X_k$, and let $\nu(Z)$ denote the variance of the sum:
    \begin{equation*}
        \nu(Z)=\E Z^2=\sum_{k=1}^n \E X_k^2
    \end{equation*}
    Then for all $t \geq 0$,
    \begin{equation*}
        \Pr\{|Z| \geq t\} \leq 2 \exp \left(\frac{-t^2 / 2}{\nu(Z)+L t / 3}\right) \leq 2\exp\left\{-\min\left( \frac{t^2}{4\nu(Z)}, \frac{3t}{4L} \right)\right\} 
    \end{equation*}
\end{lemma}
\begin{proof}
    For the first inequality, see Theorem 1.6.1 in \cite{tropp2015concentration}. The second inequality follows from the observation that for any $a,b,c > 0$, 
    \begin{equation*}
        -\frac{c}{a+b} \leq -\min\left(\frac{c}{2a}, \frac{c}{2b}\right). 
    \end{equation*}
\end{proof}

\end{document}